\newcommand{\NN}{{\mathbb{N}}}

\newcommand{\QQ}{{\mathbb{Q}}}
\newcommand{\RR}{{\mathbb{R}}}
\newcommand{\EE}{{\mathbb{E}}}
\newcommand{\PP}{{\mathbb{P}}}

\newcommand{\clb}{{\mathcal{B}}}
\newcommand{\clc}{{\mathcal{C}}}
\newcommand{\cld}{{\mathcal{D}}}
\newcommand{\cle}{{\mathcal{E}}}
\newcommand{\clf}{{\mathcal{F}}}

\newcommand{\cls}{{\mathcal{S}}}

\newcommand{\Om}{\Omega}

\newcommand{\taubar}{{\bar{\tau}}}

\newcommand{\bfZ}{\mathbf{Z}}
\newcommand{\bfB}{\mathbf{B}}
\newcommand{\bfe}{\mathbf{e}}
\newcommand{\bfz}{\mathbf{z}}
\newcommand{\bfL}{\mathbf{L}}
\newcommand{\sn}{\varsigma}
\newcommand{\lala}{\gamma}
\newcommand{\nunu}{\nu}
\newcommand{\en}{\varsigma^*}
\newcommand{\gn}{g}
\newcommand{\vsig}{\vartheta}
\newcommand{\vrr}{\varrho}
\newcommand{\BM}{\mbox{{\small BM}}}

\documentclass[10pt]{article}
\usepackage[portrait,margin=2.54cm]{geometry}

\usepackage{amsfonts}

\usepackage{amssymb}
\usepackage{mathrsfs}
\usepackage{soul}
\usepackage{hyperref}
\usepackage{amssymb,amsthm,amsmath,amsfonts,amsbsy,latexsym}
\usepackage{graphicx}
\usepackage[numeric,initials,nobysame]{amsrefs}
\usepackage{upref,setspace}
\usepackage{enumerate}
\usepackage{paralist}
\usepackage{color}

\usepackage{mathtools}

\definecolor{expcol}{rgb}{1.0,0.5,0.5}
\definecolor{ecol}{rgb}{0.0, 0.5, 1.0}

\newcommand{\editc}[1]{{\color{black} #1}}

\numberwithin{equation}{section}
\numberwithin{figure}{section}
\numberwithin{table}{section}
\newtheorem{lemma}{Lemma}[section]
\newtheorem{proposition}[lemma]{Proposition}

\newtheorem{theorem}[lemma]{Theorem}

\newtheorem{remark}[lemma]{Remark}

\newtheorem{corollary}[lemma]{Corollary}

\numberwithin{equation}{section}

\mathtoolsset{showonlyrefs}
\newcommand{\cg}{\color{black}}
\begin{document}

\title{The Inert Drift Atlas Model}
\author{Sayan Banerjee, Amarjit Budhiraja, Benjamin Estevez}
\maketitle

\begin{abstract}
Consider a massive (inert) particle impinged from above by N Brownian particles that are instantaneously reflected upon collision with the inert particle.  
 The velocity of the inert particle increases due to the influence of an external Newtonian potential (e.g. gravitation) and decreases in proportion to the total local time of collisions with the Brownian particles. 
 This system models a semi-permeable membrane in a fluid having microscopic impurities (Knight (2001)).
 We study the long-time behavior of the process $(V,\mathbf{Z})$, where $V$ is the velocity of the inert particle and $\mathbf{Z}$ is the vector of gaps between successive particles ordered by their relative positions. The system is not hypoelliptic, not reversible, and has singular form interactions. Thus the study of stability behavior of the system requires new ideas.  We show that this process has a unique stationary distribution that takes an explicit product form which is Gaussian in the velocity component and Exponential in the other components. We also show that convergence in total variation distance to the stationary distribution happens at an exponential rate. We further obtain certain law of large numbers results for the particle locations and intersection local times.

\noindent\newline


\end{abstract}

\section{Introduction}

\label{sec:intro}  

\editc{\subsection{Motivation and Model Description}}
In this work we study the long-time behavior of an interacting particle system comprising a massive (inert) particle that moves under the combined influence of an external Newtonian potential (eg. gravitation) and a non-Newtonian `inert drift' resulting from collisions with many microscopic (Brownian) particles. This serves as a simplified model for the motion of a semi-permeable membrane in a fluid having microscopic impurities (see \cite{KnightInert}). The membrane, which allows fluid molecules to pass but is impermeable to the impurities, plays the role of the inert particle.

Mathematically, this model consists of $N$-Brownian particles in $\RR$, with state processes denoted as $\{X_i(t), t \ge 0\}_{1\le i \le N}$, interacting with the inert particle, with state process $X_0(t)$, according to the following system of equations: For $t \ge 0$,
\begin{equation}\label{eq:unrankloc}
	\begin{aligned}
		X_0(t) &= x_0 + \int_0^t V(s) ds, \;\; V(t) = v_0+ gt - \sum_{i=1}^N \ell_i(t) ,\\
		X_i(t) &= x_i + W_i(t) + \ell_i(t), \; 1 \le i \le N.
	\end{aligned}
\end{equation}
Here \editc{$x_0 \le x_1 \le \cdots \le x_N$} denote the initial positions of the $N+1$ particles, $v_0$ the initial velocity of the inert particle, $\{W_i, 1 \le i \le N\}$ are mutually independent standard real Brownian motions, $g \in (0,\infty)$ denotes the gravitation constant and $\ell_i$ is  the collision local time between the $i$-th particle and the inert particle which, in particular, satisfies $\ell_i(t) = \int_0^t 1_{\{X_i(s) = X_0(s)\}} d\ell_i(s)$ for $1\le i \le N$ and $t\ge 0$. 
The local time interactions model the cumulative transfer of momentum when a Brownian particle collides with the inert particle `infinitely often' on finite time intervals, with each collision resulting in an infinitesimal momentum transfer. Such interactions lie at the heart of this model and interesting long time behavior results from the combined effect of the `soft' gravitational potential and `hard' collisions.

\editc{It follows from \cite{barnes2018strong} (see Theorem 2.5 and Proposition 2.10 therein) that there is a strong solution to the system of equations in 
\eqref{eq:unrankloc} and the solution satisfies $X_0(t) \le X_i(t)$ for all $t\ge 0$ and $1\le i \le N$ a.s. Using Gronwall's lemma and the Lipschitz property of the Skorohod map it is easy to verify that in fact the system of equations in 
\eqref{eq:unrankloc} has a unique strong solution. Given this unique  solution process $\{X_i(t),\; t \ge 0\}_{0\le i \le N}$ of \eqref{eq:unrankloc} it will be
 convenient} to consider the ordered particle system: 
$$X_{(0)}(t) \le X_{(1)}(t) \le \cdots \le X_{(N)}(t), \; t \ge 0,$$
where $\{X_{(i)}(t) : t \ge 0\}$ denotes the state process of the $i$-th particle from the bottom (\editc{note that the lowest particle, which we call the $0$-th particle from the bottom, is the inert particle,  in particular, $X_{(0)}(\cdot) = X_0(\cdot)$}). By an application of Tanaka's formula it is easy to verify that this ranked particle system satisfies the following system of equations: For $t \ge 0$,
\begin{equation}\label{eq:rankloc}
	\begin{aligned}
		X_{(0)}(t) &= x_0 + \int_0^t V(s) ds, \;\; V(t) = v_0+ gt -  L_1(t),\\
		X_{(1)}(t) &= x_1 + B_1(t)  -\frac{1}{2} L_2(t) + L_1(t), \\
		X_{(i)}(t) &= x_i + B_i(t)  -\frac{1}{2} L_{i+1}(t) + \frac{1}{2} L_{i}(t), \; 2 \le i \le N.
	\end{aligned}
\end{equation}
where \editc{$x_0 \leq x_1 \leq ... \leq x_N$,} $\{B_i, 1 \le i \le N\}$ are standard independent Brownian motions and  for $1\le i \le N$, $L_{i}$  denotes the collision  local time between the $i$-th and the $(i-1)$-th ranked particle which satisfies $L_i(t) = \int_0^t 1_{\{X_{(i)}(s) = X_{(i-1)}(s)\}} dL_i(s)$  and $L_{N+1}(t)=0$  for all $t\ge 0$ .

We are interested in the time asymptotic behavior of the {\em velocity and gap processes} associated with this system. Namely, denoting $Z_i(t)  \doteq X_{(i)}(t) - X_{(i-1)}(t)$, the object of interest is the stochastic process
$$(V(t), Z_1(t), \cdots, Z_N(t)).$$
This process is given by the system of equations
\begin{equation}\label{eq:gapproc}
	\begin{aligned}
		V(t) &= v_0+ gt -  L_1(t),\\
		Z_1(t) &= z_1 + B_1(t)  - \int_0^t V(s) ds - \frac{1}{2} L_2(t) + L_1(t), \\
		Z_2(t) &= z_2 + B_2(t)- B_1(t)  - \frac{1}{2} L_3(t) + L_2(t) - L_1(t),\\
		Z_i(t) &= z_i + B_i(t)  - B_{i-1}(t) -\frac{1}{2} L_{i+1}(t) +  L_{i}(t) - \frac{1}{2} L_{i-1}(t), \; 3 \le i \le N.
	\end{aligned}
\end{equation}

The model described by equations \eqref{eq:rankloc} (with gaps evolving as in \eqref{eq:gapproc}), which we call the \emph{inert drift Atlas model}, lies at the interface of two well-studied classes of interacting particle systems: \emph{inert drift models} and \emph{rank-based diffusions}, which we summarize below.

\editc{\subsection{Previous Work}}

The case where $N=1$ (namely the two particle system) with $g=0$ was analyzed in \cite{KnightInert}, which initiated the study of inert drift models. It was shown there that the inert particle progressively gains momentum from the local time interactions and eventually escapes the Brownian particle (no further collisions). When $g>0$, \cite{banerjee2019gravitation} showed that the two particles never escape each other. Among other results, the paper showed that the velocity of the inert particle and the gap between the two particles jointly converge in total variation distance to an explicit stationary distribution having a product form density (no rates of convergence were obtained). The two particle model with gravitation and fluid viscosity was investigated in \cite{BanBro}. In \cite{BBCH}, an inert drift model was considered where a particle moves as a diffusion process inside a bounded smooth domain and acquires inert drift when it hits the boundary of the domain. It was shown that the position of the particle and the cumulative inert drift have a product form stationary measure, which is unique under suitable conditions. A variety of related inert drift models have been studied in \cite{BurEtAl,WhiteInert,BurWhi}. When the term $\sum_{i=1}^N \ell_i(t)$ in \eqref{eq:unrankloc} is replaced by $N^{-1}\sum_{i=1}^N \ell_i(t)$ (mean field type interaction), the asymptotic behavior as $N\to \infty$
has been analyzed in \cite{BarHyd,barnes2018strong} where results on hydrodynamic limits and propagation of chaos have been obtained. Recently, unexpected connections have appeared between inert drift models and diffusion limits of load balancing systems like the Join-the-shortest-queue policy in heavy traffic \cite{eschenfeldt2018join},\cite{BanMuk},\cite{banerjee2020join}. More precisely, the joint evolution of the diffusion-scaled number of idle servers and busy servers converges in distribution to a diffusion that  resembles the two particle inert drift system with linear drift. Consequently, there are several common themes at the technical level between \cite{banerjee2019gravitation,BanBro} and \cite{BanMuk,banerjee2020join}.
Brownian particle systems of the form studied in the current work also arise as diffusion approximations of certain types of queuing systems in which each queue has the same finite capacity which is dynamically controlled in a manner that the increase in capacity is proportional to net job loss due to capacity constraints. In this model, currently under investigation, the individual queues play the role of Brownian particles whereas the dynamically changing queue capacity threshold represents the massive inert particle.

In a somewhat different vein, inspired by problems in mathematical finance, the study of rank-based diffusions
\cite{Atlas2},\cite{Atlas1},\cite{Atlas3},\cite{DJO},\cite{sarantsev2017stationary},\cite{AS},\cite{banerjee2021domains} have gained a lot of attention in recent years. These models consist of a collection of particles on the real line which evolve as diffusion processes where the drift and diffusivity of each particle is a function of its relative rank in the system. Closest in spirit to our model is the \emph{Atlas model} where the lowest ranked particle at any time moves as a Brownian motion with constant upward drift while the remaining particles evolve as standard Brownian motions (with zero drift).

\editc{\subsection{Analytical Challenges}}

The Atlas model and the model considered here are examples of particle systems with topological interactions in the terminology of \cite{CDGP}. In such particle systems, interactions between particles are determined by their relative positions. In particular, in both the Atlas model and in the particle system considered here, the lowest particle has different dynamical properties. Specifically, in the Atlas model the lowest particle gets a constant upward drift whereas in the model considered here the lowest particle experiences an {\em inert drift}. However there are some important differences between the two models. Unlike the Atlas model, where the collision local time of the lowest two particles enters directly in the position evolution of the lowest particle, here this local time impacts the velocity of the lowest particle. Indeed, this collision local time is the source of the inert drift of the lowest particle. Furthermore,
  %
  there is no Brownian noise in the equation for $X_{(0)}$ in \eqref{eq:rankloc}, unlike in the Atlas model. This results in the deterministic evolution of the velocity process in time periods with no collisions, making the full system, whose long-time behavior is of interest, non-elliptic (in fact, the driving diffusion process in the interior of the domain is not even hypoelliptic). More precisely, the law of $(V(t), Z_1(t), \cdots, Z_N(t))$ for any $t>0$ does not have a density with respect to Lebesgue measure, for general initial conditions. Also, we find that, unlike the Atlas model, the system considered here is not reversible.
  Hence,  standard techniques for studying ergodicity behavior of elliptic diffusion processes cannot be applied, and one needs new methods. 
  As noted above, inert two-particle systems have been studied in several previous works, however the current work is the first to study the ergodicity properties of a general $N$-particle system. There are fundamental differences in system behavior 
  as one goes from $N=1$ to $N>1$ which make the study of ergodicity behavior significantly more demanding.  \editc{In particular, as is crucially exploited in  \cite{banerjee2019gravitation, BanBro}, in the $N=1$ case, there is a basic \emph{regenerative structure}  arising from the fact that at points of decrease of the velocity process, the remaining state  coordinate, namely the one corresponding to $Z_1$, is fully determined (in fact equal to $0$). In the general $N$-particle system there is no such simple regenerative structure since, although the first gap coordinate $Z_1$ is once again $0$ at points of decrease of $V$, the remaining coordinates, namely $Z_2, \ldots , Z_N$ can be arbitrary.}

\editc{\subsection{Main Contributions}}

We now briefly describe the main contributions of this work.  Since the system is not hypoelliptic, one cannot apply standard existing theory to argue uniqueness of invariant measures. Our first main result says that the Markov proces $(V, \bfZ)= (V, Z_1, \ldots, Z_N)$ admits at most one stationary distribution. We then produce an explicit stationary distribution  for the system and together the two results (see Theorems \ref{thm:exisuniq} and \ref{thm:prodform}) prove existence and uniqueness of
stationary distributions of $(V, \bfZ)$. We in fact show that the unique stationary distribution takes a product form whose first component (corresponding to the velocity coordinate) is Gaussian and remaining are Exponential (see Theorem \ref{thm:prodform} for the precise form). \editc{In the case $N=1$, a Gaussian-Exponential product form stationary distribution has appeared in previous works \cite{WhiteInert},\cite{BBCH},\cite{banerjee2019gravitation}; however, this is the first work that finds such a product form structure for a general \editc{$N$-particle} system.} This stationary distribution also has striking similarities with the Atlas model where the stationary distribution is a product of exponentials with rates decreasing with the ranks of the particles (see, for example, \cite[Theorem 8]{Atlas1}). 

We next study the rate of convergence to stationarity. 
In Theorem \ref{thm:geomerg}, we show that the distribution of $(V(t), Z_1(t), \cdots, Z_N(t))$ converges to equilibrium exponentially fast (exponential ergodicity) as $t\to \infty$. \editc{To the best of our knowledge, this is the first result on  exponential ergodicity for any type of non-hypoelliptic reflected diffusion in dimensions higher than $2$.} 

 Finally in Theorem \ref{lln} we establish some law of large numbers type results. In particular, it is shown that the whole system `drifts' to infinity at speed $g/N$. Although this is an intuitive result to expect, our proof crucially hinges on the 
 rather technical result on exponential  moments of return times to certain compact sets that form the basis of the exponential ergodicity proof.
 We also find, somewhat surprisingly, that the intensity of collisions when $N \ge 3$ is maximum, in a certain sense, between the first two Brownian particles (rather than between the inert and the first Brownian particle); see Remark \ref{int12}.
 
 \editc{\subsection{Approach}}

A common approach to proving ergodicity or exponential rates of convergence to stationarity for diffusions in domains is by constructing a suitable Lyapunov function by analyzing the interplay between the ``interior drift vector field'' and the reflection vector field (cf. \cite{dupuis1994lyapunov}, \cite{atar2001positive}, \cite{BudLee}). For example, in polyhedral domains with constant (oblique) reflection on each face of the boundary, the key insight in the construction of a Lyapunov function is that the drift vector field for stable systems must lie in the interior of the cone generated by the negatives of the reflection directions. 
Note that  $\bfZ$ is a  
reflected diffusion in the positive orthant $\mathbb{R}^N_+$ with constant oblique reflection  at each face. The interior drift of this process is $V(t) \mathbf{e}_1$, where $\mathbf{e}_1$ is the unit vector with $1$ in the first coordinate. Due to the complicated dynamics of $V$, that includes in particular the local time for the first gap process $Z_1$, its behavior in relation to the reflection field seems hard to analyze which makes a direct construction of a explicit form Lyapunov function (as in the above cited works) hard. 

In this work we instead
take a pathwise approach. The stability in the particle system studied here arises as a result of interplay between the intersection local times for the various particles in the system. This interplay is distilled in Lemma \ref{zless} which identifies a stabilizing `singular' drift that prevents the gaps between the particles from being too large. This key lemma allows us to prove the
finiteness of exponential moments of hitting times to certain compact sets by analyzing excursions of the process between suitably chosen stopping times (see Sections \ref{highlev}-\ref{comphit}). 
In conjunction with results of \cite{DowMeyTwe} (see Proposition \ref{driftcondn} (a)), this analysis furnishes a general abstract form Lyapunov function, given in terms of exponential moments of these hitting times, which is key in the proof of exponential ergodicity. 
%
%
%
Another important ingredient in our proofs
is 
establishing a certain minorization estimate (see Proposition \ref{driftcondn} (b)).
For hypoelliptic diffusions such an estimate follows readily from the existence of a density for the process at each time $t>0$. However, in our case, establishing a suitable minorization bound involves substantial work and a careful exploitation of the properties of the collision local times of the particles in the system.
The proof of this estimate, which uses an intricate and novel 
 %
 pathwise analysis,  is the topic of  Section \ref{minosec}.


\editc{\subsection{Future Directions}}
 
The current work is the first step in our program of analyzing high-dimensional reflected diffusions with inert drift type interactions. The natural next step will be to investigate ergodicity properties of the infinite-dimensional analogue of our model. The corresponding vector of velocity and gap processes is expected to have at least one stationary distribution, given by the $N \rightarrow \infty$ limit of \eqref{eq:statdistn} below. It is unclear if this is the unique stationary distribution. Analogy with the Atlas model suggests infinitely many stationary distributions, each with a non-trivial domain of attraction \cite{sarantsev2017stationary},\cite{DJO},\cite{banerjee2021domains}. Another interesting question concerns the study of hydrodynamic limits of empirical occupation measures of the system and relate them to the path asymptotics of the bottom $k$ particles for $k \in \mathbb{N}$ (see \cite{dembo2017equilibrium} for related results on the Atlas model). Both these directions are currently under investigation.

\subsection{Notation and Preliminaries}

The following notation will be used.  For $d \in \NN$ and $T>0$, we denote by $\clc([0,T]: \RR^d)$ (resp. $\clc([0, \infty): \RR^d)$) the space of continuous functions on $[0,T]$ (resp. $[0,\infty)$) with values in $\RR^d$, equipped with the topology of uniform convergence (resp. local uniform convergence).
The spaces $\clc([0,T]: \RR_+^d)$ (resp. $\clc([0, \infty): \RR^d_+)$) of continuous functions with values in the nonnegative orthant $\RR_+^d$ are defined similarly. For $t \in [0,\infty)$ and $f \in \clc([0, \infty): \mathbb{R}^d)$, we define
$\|f\|_t \doteq \sup_{0 \leq s \leq t}|f(s)|$, where $|\cdot|$ is the Euclidean norm on $\mathbb{R}^{d}$.  
Borel $\sigma$-fields on a metric space $S$ will be denoted as $\clb(S)$. Inequalities for vectors and vector-valued random variables  are understood to be coordinatewise. An open set $G \subset \RR^d$ is said to have a $\clc^2$ boundary if 
each point in $\partial G$ has a neighborhood in which $\partial G$ is the graph of a  ${\mathcal{C}}^{2 }$ function of $d-1$ of the coordinates (cf. \cite[Section 6.2]{giltru}).
Throughout  $\lambda$ will denote the Lebesgue measure on a subset of a Euclidean space whose dimension will be clear from the context.

The following elementary estimate will be used several times. Suppose for $m\in \NN$, $\tilde B_1, \ldots\editc{,} \tilde B_m$ are mutually independent Brownian motions and $\alpha_1, \ldots\editc{,} \alpha_m \in \RR_+$.
Let $\tilde B_i^* (t)\doteq \sup_{0\le s \le t} |\tilde B_i(s)|$. Then there are $\varrho_1, \varrho_2 \in (0,\infty)$, such that
\begin{equation} \label{eq:elemconc}
	\mathbb{E}\left( e^{u\sum_{i=1}^m \alpha_i\tilde B_i^* (t)}\right) \le \varrho_1 e^{\varrho_2 u^2 t} \mbox{ for all } t \ge 0 \mbox{ and } u \ge 0.\end{equation}
The dependence of the constants $\varrho_1, \varrho_2$ on $m$ and $\alpha_i$ will usually be suppressed from the notation.

\editc{In the next section, we outline our main results. The organization of the paper is summarized at the end of the section.}

\section{Main Results}
\label{sec:mainres}
Define the $N \times N$ matrix
\begin{equation}\label{Rmat}
R \doteq
\begin{pmatrix}
 1 & -\frac{1}{2} & 0 & 0 & \cdots & 0 \\
 -1 & 1 & -\frac{1}{2} & 0 & \cdots & 0 \\
 0 & -\frac{1}{2} & 1 & -\frac{1}{2} & \cdots & 0 \\
 \vdots & \vdots & \vdots & \vdots & \vdots & \vdots \\
 0 & \cdots & \cdots & \cdots & -\frac{1}{2} & 1
\end{pmatrix}.
\end{equation}
It is easily checked that the matrix $U=I-R$ has the property that $U^T$ is a transient, substochastic matrix and thus has spectral radius strictly less than $1$. Consequently, $R$ is invertible and  $W= R^{-1}$ can be written as an infinite sum of matrices with nonnegative entries. 
The Skorohod problem associated with such matrices $R$ has been well studied and the following result is well known cf. \cite{harrei},\cite{dupish},\cite{dupram}.

We denote by  $\clc_0([0, \infty): \mathbb{R}^N)$ the space of continuous functions $f: [0,\infty) \rightarrow \mathbb{R}^N$ such that $f(0) \geq 0$.

\begin{proposition}\label{skorokhod}
 To each $x \in \clc_0([0, \infty):\mathbb{R}^N)$ there is a unique pair $(\eta,y) \in \clc([0, \infty):\mathbb{R}_+^N) \times \clc([0, \infty):\mathbb{R}_+^N)$ such that, 

\begin{enumerate}[(i)]
    \item for all $t\ge 0$, $y(t) = x(t) + R\eta(t)$,
    \item For each $i \in \{1,...,N\}$,
    \begin{inparaenum}
        \item  $\eta_i(0) = 0$,
        \item $\eta_i(t)$ is non-decreasing in $t$,
        \item $\int_0^{\infty}y_i(t)d\eta_i(t) = 0$.
    \end{inparaenum}
\end{enumerate}
The pair $(\eta,y)$ is called the solution to the Skorokhod problem for $x$ with respect to $R$. The map $\Gamma: \clc_0([0, \infty):\mathbb{R}^N) \rightarrow \clc([0, \infty):\mathbb{R}_+^N) \times \clc([0, \infty):\mathbb{R}_+^N)$ given by
\begin{align*}
    \Gamma(x) = (\eta,y) = (\Gamma_1(x),\Gamma_2(x))
\end{align*}
is Lipschitz in the sense that there is a $c_{\Gamma} \in (0,\infty)$ such that for $x,x' \in \clc_0([0, \infty):\mathbb{R}^N)$
and $t <\infty$,
$$\|\Gamma_1(x) - \Gamma_1(x')\|_t + \|\Gamma_2(x) - \Gamma_2(x')\|_t \le c_{\Gamma} \|x-x'\|_t.$$
\end{proposition}
For $x \in \clc_0([0, \infty):\mathbb{R}^N)$, we occasionally write $\Gamma_1(x) = (\Gamma_{11}(x), \ldots , \Gamma_{N1}(x))$.

The following result gives strong existence and uniqueness for the system of equations in \eqref{eq:gapproc}. Proof is given in Section \ref{sec:exisuniq}.
Let 
\begin{equation}\label{Amat}
A  \doteq
\begin{pmatrix}
1 & 0 & 0 & \cdots & 0 & 0 \\
-1 & 1 & 0 & \cdots & 0 & 0 \\
0 & -1 & 1 & \cdots & \vdots & \vdots \\
\vdots & \vdots & \vdots & \ddots & 1 & 0 \\
0 & 0 & 0 & \cdots & -1 & 1 \\
\end{pmatrix}.
\end{equation}
\begin{theorem}\label{thm:wellposed}
Let $(\bar \Om, \bar\clf, \{\bar \clf_t\}_{t\ge 0}, \bar \PP)$ be a filtered probability space on which are given $N$  mutually independent standard real $\bar \clf_t$-Brownian motions $B_1, \ldots, B_N$ and, $\bar\clf_0$-measurable random variables $V^0$ and
$\bfZ^0 = (Z_1^0, \ldots, Z_N^0)$ with values in $\RR$ and $\RR_+^N$ respectively. 
Then there is a continuous, $\bar\clf_t$-adapted, stochastic process
$(V(t), Z_1(t), \ldots, Z_N(t))_{0\le t < \infty}$ with values in $\RR \times \RR_+^N$ such that, for all $t\ge 0$,
\begin{equation}\label{eq:solskor}
\begin{aligned}
	V(t) &= V^0 + gt - L_1(t), \\
	\bfZ(t) &=\Gamma_2\left(\bfZ^0 - \bfe_1 \int_0^{\cdot} V(s) ds  + A\bfB(\cdot)\right)(t),\\
	L_1(t) &= \Gamma_{11}\left(\bfZ^0 - \bfe_1 \int_0^{\cdot} V(s) ds  + A\bfB(\cdot)\right)(t),\\
\end{aligned}
\end{equation}
where $\bfB = (B_1, \ldots, B_N)'$ and $\bfZ = (Z_1, \ldots , Z_N)'$.
Furthermore, if $(\tilde V(t), \tilde Z_1(t), \ldots\editc{,} \tilde Z_N(t))$ is another such process then
$$(\tilde V(t), \tilde Z_1(t), \ldots\editc{,} \tilde Z_N(t)) = ( V(t), Z_1(t), \ldots , Z_N(t)) \mbox{ for all } t \ge 0,  \mbox{ a.s. } $$
\end{theorem}
We remark that, with $\bfZ$, and $V$ as in the theorem, letting 
$$\bfL(t) = (L_1(t) , \ldots , L_N(t))= \Gamma_1\left(\bfZ^0 - \bfe_1 \int_0^{\cdot} V(s) ds  + A\bfB(\cdot)\right)(t) \mbox{ and } L_{N+1}(t)=0,\; $$
we have that the following system of equations holds:
\begin{equation}\label{eq:gapprocb}
	\begin{aligned}
		V(t) &= V^0+ gt -  L_1(t),\\
		Z_1(t) &= Z_1^0 + B_1(t)  - \int_0^t V(s) ds - \frac{1}{2} L_2(t) + L_1(t), \\
		Z_2(t) &= Z_2^0 + B_2(t)- B_1(t)  - \frac{1}{2} L_3(t) + L_2(t) - L_1(t),\\
		Z_i(t) &= Z_i^0 + B_i(t)  - B_{i-1}(t) -\frac{1}{2} L_{i+1}(t) +  L_{i}(t) - \frac{1}{2} L_{i-1}(t), \; 3 \le i \le N.
	\end{aligned}
\end{equation}
Consider the path space $\Om^* = \clc([0,\infty): \RR^N \times \RR \times \RR_+^N)$, $\clf^*$ the corresponding Borel $\sigma$-field on $\Om^*$. We also consider the space
$(\Om, \clf) \doteq (\clc([0,\infty):  \RR \times \RR_+^N), \clb(\clc([0,\infty):  \RR \times \RR_+^N))$.
On these two measurable spaces we
denote, for $(v, \bfz) \in \RR \times \RR_+^N$, by $\PP^*_{(v,\bfz)}$ [resp. $\PP_{(v,\bfz)}$],  the probability measures  induced by  $(\bfB, V, \bfZ)$ [resp. $(V, \bfZ)$] where $(V, \bfZ)$ is the solution
of \eqref{eq:solskor} when $(V^0, \bfZ^0) = (v, \bfz)$ a.s. 
Then from the unique solvability in the above theorem it follows that $\{\PP_{(v,\bfz)}\}_{(v,\bfz) \in \RR \times \RR_+^N}$ defines a strong Markov family. The next result concerns the stationary distribution of this Markov family.
\begin{theorem}\label{thm:exisuniq}
	There is a unique stationary distribution for the Markov family $\{\PP_{(v,\bfz)}\}_{(v,\bfz) \in \RR \times \RR_+^N}$.
\end{theorem}
In fact this unique stationary distribution takes an explicit product form as given by the theorem below.
Consider the probability measure $\pi$ on $\RR \times \RR_+^N$ given by the formula:
\begin{equation}\label{eq:statdistn}
\pi(dv,\, dz_1, \, \ldots , dz_N) \doteq  c_{\pi}e^{-(v-\frac{g}{N})^2}\prod_{i=1}^Ne^{-2g\editc{\left(\frac{N-i+1}{N}\right)}z_i}\,	dv\, dz_1\, \ldots , dz_N,
\end{equation}
where $c_{\pi}$ is the normalization constant. 
\begin{theorem}\label{thm:prodform}
The probability measure $\pi$ defined in \eqref{eq:statdistn} is the unique stationary distribution of 
$\{\PP_{(v,\bfz)}\}_{(v,\bfz) \in \RR \times \RR_+^N}$.
\end{theorem}
\editc{Note that while Theorem \ref{thm:prodform} implies Theorem \ref{thm:exisuniq}, we proceed by first showing that there exists at most one stationary distribution (Theorem \ref{atmostonesd}). The existence and explicit form of the stationary distribution is subsequently exhibited (in Section \ref{sec:prodform}) by proving that the density of $\pi$ solves the partial differential equation (with boundary conditions) arising from the basic adjoint relationship (see \eqref{s1}-\eqref{s4}). We have therefore separated out these results for clarity of exposition.}

Our third result gives exponential ergodicity of the Markov process. 
Write an element $\omega \in \Omega^*$ [resp. $\omega \in \Omega$] as $\omega = (\beta,\upsilon,\zeta)$ [resp. $\omega = (\upsilon,\zeta)$], where $\beta \in \clc([0,\infty): \RR^N)$, $\,\upsilon \in \clc([0,\infty): \RR)$ and $\,\zeta \in \clc([0,\infty): \RR_+^N)$. For $t \in [0,\infty)$, 
abusing notation, denote the coordinate processes  $\bfB(t), V(t)$ and $\bfZ(t)$ on $(\Om^*, \clf^*)$ [resp. $V(t)$ and $\bfZ(t)$ on $(\Om, \clf)$]  by the formulae
\begin{align*}
    \bfB(t)(\omega)   = \beta(t), \;\; 
    V(t)(\omega) = \upsilon(t), \;\;
   \bfZ(t)(\omega) = \zeta(t), \; \; t\ge 0.
\end{align*}
Also, we will write $B_i(t)$ and $Z_i(t)$ respectively for the projections of $\bfB(t)$ and $\bfZ(t)$ onto their $i^{\textnormal{th}}$ coordinates. 
Consider the transition probability kernel of the Markov family $\{\PP_{(v,\bfz)}\}_{(v,\bfz) \in \RR \times \RR_+^N}$ defined as
$$\PP^t((v, \bfz), A) \doteq \PP_{(v,\bfz)}((V(t), \bfZ(t)) \in A), \; t\ge 0, (v,\bfz) \in \RR \times \RR_+^N,\; A \in \clb(\RR \times \RR_+^N).$$
Also, for a bounded and measurable $\phi: \RR \times \RR_+^N \to \RR$ we write
$$\PP^t((v, \bfz),\, \phi) \doteq \int_{\RR \times \RR_+^N} \phi(\tilde v, \tilde \bfz)\, \PP^t((v, \bfz),\, d\tilde v\times d\tilde\bfz).$$
Similarly, for $\phi$ as above, $\pi(\phi) \doteq \int \phi(\tilde v, \tilde \bfz) \pi(d\tilde v\times d\tilde\bfz)$.
The following theorem shows the convergence of the transition probability kernel to the stationary distribution in the total variation distance at an exponential rate.
Denote by $\BM_1$ the class of all  measurable
	$\phi: \RR \times \RR_+^N \to \RR$ such that $\sup_{(v,\bfz) \in \RR \times \RR_+^N} |\phi(v,\bfz)|\le 1$.
\begin{theorem}\label{thm:geomerg}
	There is a $\gamma \in (0,1)$ and, for every  $(v,\bfz) \in \RR \times \RR_+^N$, a $\kappa(v,\bfz) \in (0,\infty)$, such that for all $t\ge 0$, 
$$ \sup_{\phi \in \BM_1} |\PP^t((v, \bfz),\, \phi) - \pi(\phi)| \le \kappa(v, \bfz) \gamma^t .$$
\end{theorem}

\editc{We note here that the proof of exponential ergodicity proceeds through establishing finiteness of exponential moments of certain hitting times. This, in turn, provides the tightness required to furnish an independent proof of existence of a stationary distribution.}

Finally, we prove a strong law of large numbers type result for the system. Recall the ranked particle system $\{X_{(i)}(\cdot)\}_{0 \le i \le N}$ from \eqref{eq:rankloc}. This process can be constructed on $(\Om^*, \clf^*, \PP^*_{(v,\bfz)})$ for any
$(v, \bfz) \in \RR \times \RR_+^N$ by solving the system of equations \editc{in \eqref{eq:solskor} (or equivalently\eqref{eq:gapprocb}),} whose unique pathwise solutions are guaranteed by Theorem \ref{thm:wellposed}, and then defining $X_{(i)}$ by the right side of \eqref{eq:rankloc}.

\begin{theorem}\label{lln}
For any $(v,\bfz) \in \RR \times \RR_+^N$, the following limits hold $\PP^*_{(v,\bfz)}$-almost surely:
\begin{align}
\lim_{t \rightarrow \infty} \frac{X_{(i)}(t)}{t} &= \frac{g}{N}, \ \ 0 \le i \le N,\label{le1}\\
\lim_{t \rightarrow \infty} \frac{L_1(t)}{t} &= g,\label{le2}\\
\lim_{t \rightarrow \infty} \frac{L_i(t)}{t} &= \frac{2(N-i+1)g}{N}, \ \ 2 \le i \le N.\label{le3}
\end{align}
\end{theorem}

\begin{remark}\label{int12}
It is natural to expect that the gaps become larger in some sense as one moves away from the inert particle. This heuristic is quantified in the stochastic monotonicity of the stationary gaps displayed in \eqref{eq:statdistn}. From this, it might appear that the growth rate of the local time $L_i(t)$ (which quantifies the intensity of collisions between the $(i-1)$th and $i$th particle) with $t$ should decrease as $i$ increases from $1$ to $N$. However, Theorem \ref{lln} shows that for $N \ge 3$, $L_2$ grows at a faster rate than $L_1$ and the expected decrease in rates holds from $i=2$ onwards. Hence, perhaps surprisingly, particles indexed $1$ and $2$ collide `more often' than particles $0$ and $1$ as time progresses.
\vspace{-0.5cm}
\end{remark}
\editc{\subsection{Organization}}
\editc{The rest of the paper is organized as follows. In Section \ref{sec:wellposed}, we provide the proof of Theorem \ref{thm:wellposed}. In Section \ref{minosec}, we show a technical estimate which will be integral to the proofs of our main results. In Section \ref{sec:exisuniq}, we show that there is at most one stationary distribution (Theorem \ref{atmostonesd}). In Section \ref{sec:prodform}, we prove Theorem \ref{thm:prodform}. Together, these two results also establish Theorem \ref{thm:exisuniq}. In Section \ref{sec:geomerg}, we give the proof of Theorem \ref{thm:geomerg}. Proofs of several technical results stated in 
Section \ref{sec:geomerg} (without proof) are provided in Section \ref{sec:techlem}.  In Section \ref{sec:lln}, we establish Theorem \ref{lln}.} 

\section{Existence and uniqueness of the process}\label{sec:wellposed}
In this section, we prove Theorem \ref{thm:wellposed}. The proof uses the Lipschitz property in Proposition \ref{skorokhod},   and a standard Picard iteration scheme. We provide a sketch. Fix $T<\infty$.
Let $(V^0, \bfZ^0)$ be as in the statement of the theorem. Define, for $n \in \NN_0$, continuous $\bar \clf_t$-adapted $\RR \times \RR_+^N \times \RR_+^N$ valued processes $\{(V^{(n)}(t), \bfZ^{(n)}(t)), \bfL^{(n)}(t))\}_{0\le t \le T}$, recursively, as follows.
Let
$$V^{(0)}(t) \doteq V^0, \;\; \bfZ^{(n)}(t)) \doteq \bfZ^0,\;\;  \bfL^{(n)}(t) = 0, \;\; 0 \le t \le T.$$
Having defined $\{(V^{(k)}(t), \bfZ^{(k)}(t)), \bfL^{(k)}(t))\}_{0\le t \le T}$ for $k = 0, \ldots \editc{,} n-1$, define
\begin{equation}\label{eq:picit}
\begin{aligned}
\bfZ^{(n)}(t) &=\Gamma_2\left(\bfZ^0 - \bfe_1 \int_0^{\cdot} V^{(n-1)}(s) ds  + A\bfB(\cdot)\right)(t),\\
	\bfL^{(n)}(t) &= \Gamma_{1}\left(\bfZ^0 - \bfe_1 \int_0^{\cdot} V^{(n-1)}(s) ds  + A\bfB(\cdot)\right)(t),\\
		V^{(n)}(t) &= V^0 + gt - L_1^{(n)}(t), 
\end{aligned}
\end{equation}
where $L_1^{(n)}(t)$ is the first coordinate of $\bfL^{(n)}(t)$.

From the Lipschitz property in Proposition \ref{skorokhod} it follows that, for any $n \ge 2$, and $t \in [0,T]$,
\begin{align*}
	\|\bfZ^{(n)} - \bfZ^{(n-1)}\|_t + \|\bfL^{(n)} - \bfL^{(n-1)}\|_t  \le c_{\Gamma}\int_0^t \|V^{(n-1)} - V^{(n-2)}\|_s ds
\end{align*}
and
\begin{align*}
\|V^{(n)} - V^{(n-1)}\|_t	=  \|L_1^{(n)} - L_1^{(n-1)}\|_t \le \|\bfL^{(n)} - \bfL^{(n-1)}\|_t \le c_{\Gamma}\int_0^t \|V^{(n-1)} - V^{(n-2)}\|_s ds.
\end{align*}
Letting 
$$\Delta_n(t) \doteq \|\bfZ^{(n)} - \bfZ^{(n-1)}\|_t + \|\bfL^{(n)} - \bfL^{(n-1)}\|_t  + \|V^{(n)} - V^{(n-1)}\|_t,$$
we have for $n \ge 2$ and $t \in [0,T]$, $\Delta_n(t) \le c_{\Gamma}\int_0^t \Delta_{n-1} (s) ds$.
Now  a standard argument shows that, a.s.,
$(V^{(n)}, \bfZ^{(n)}, \bfL^{(n)})$ is a Cauchy sequence in $\clc([0,T]: \RR \times \RR_+^N \times \RR_+^N)$. Let $(V,\bfZ, \bfL)$ denote the limit. It is easy to verify that this is a $\bar \clf_t$-adapted process.
Furthermore, sending 
$n\to \infty$ in \eqref{eq:picit} we see that $(V, \bfZ, \bfL)$ solve, for $0 \le t \le T$,
\begin{equation}\label{eq:picitfin}
\begin{aligned}
\bfZ(t) &=\Gamma_2\left(\bfZ^0 - \bfe_1 \int_0^{\cdot} V(s) ds  + A\bfB(\cdot)\right)(t),\\
	\bfL(t) &= \Gamma_{1}\left(\bfZ^0 - \bfe_1 \int_0^{\cdot} V(s) ds  + A\bfB(\cdot)\right)(t),\\
		V(t) &= V^0 + gt - L_1(t),
\end{aligned}
\end{equation}
where $L_1(t)$ is the first coordinate of $\bfL(t)$.  In particular, $(V, \bfZ)$ is a solution of \eqref{eq:solskor}. Since $T>0$ is arbitrary this proves the first part of the theorem.

Now suppose that $(V, \bfZ, \bfL)$ and $(\tilde V, \tilde \bfZ, \tilde \bfL)$ are two continuous $\RR \times \RR_+^N \times \RR_+^N$ valued
$\bar \clf_t$-adapted processes that solve \eqref{eq:picitfin}. Then, for $t \in [0, T]$,
\begin{align*}
\|\bfZ- \tilde \bfZ\|_t	+ \|\bfL- \tilde \bfL\|_t &\le c_{\Gamma} \int_0^t \|V-\tilde V\|_s ds =
 c_{\Gamma} \int_0^t \|L_1-\tilde L_1\|_s ds\\
 & \le  c_{\Gamma} \int_0^t (\|\bfZ- \tilde \bfZ\|_s	+ \|\bfL- \tilde \bfL\|_s) ds.
\end{align*}
Using Gr\"onwall's lemma, it then follows that $\bfZ(t)= \tilde \bfZ(t)$ and $\bfL(t) = \tilde \bfL(t)$ for all $t \in [0,T]$ a.s. which also says that $V(t) = \tilde V(t)$ for all $t \in [0,T]$ a.s.  The result follows.
\hfill \qed

\section{A Minorization Estimate}\label{minosec}
In this section we will establish a minorization estimate  for the  transition probability kernel $\PP^t((v, \bfz), A)$ introduced in the last section. This estimate will be a key ingredient in the proofs of Theorems \ref{thm:exisuniq} and \ref{thm:geomerg}. The deterministic motion of the bottom (inert) particle when $Z_1>0$ results in very singular behavior of our diffusion process manifested, in particular, by the lack of a density of $(V(t),\mathbf{Z}(t))$ with respect to Lebesgue measure for any $t >0$ when the initial condition satisfies $Z_1(0)>0$. Hence, one cannot use standard techniques for establishing a minorization condition for elliptic (or hypoelliptic) diffusions. We take a pathwise approach here by analyzing a suitable collection of driving Brownian paths to obtain a sub-density of the form described in Theorem \ref{minorization}. This is done by first `removing the drift' by applying Girsanov's Theorem and analyzing the simpler system given by gaps between $N$ ordered Brownian motions and the local time at zero of the bottom particle. This, along with an appropriate control of the Radon-Nikodym derivative, yields the desired result.

 Let
 \begin{equation}
 	\sn \doteq \frac{1}{128}, \;\;\; \en \doteq \sn + \left(\frac{1}{63}-\frac{1}{64}\right).
 \end{equation} 
\begin{theorem}\label{minorization} 
	Let $C = [0, \frac{g}{128}] \times [\frac{g}{2},g]^N$.
	There exists $D \in \clb(\mathbb{R} \times \mathbb{R}_+^N)$ such that
    $\lambda(D \cap C) > 0$,
and such that for each $(v, \bfz) \in [0, \frac{g}{128}]\times (0, \infty) \times \mathbb{R}_+^{N-1}$, there is a $K_{(v,\bfz)} \in (0,\infty)$ so that
\begin{equation}\label{MLminor}
    \inf_{t \in [\sn, \en]}\mathbb{P}^t((v, \bfz), S) \geq K_{(v,\bfz)}\lambda(S \cap D) \mbox{ for every } S \in \clb(\mathbb{R} \times \mathbb{R}_+^N).
\end{equation}
 Moreover, the map $(v,\bfz) \mapsto K_{(v,\bfz)}$ is measurable and for any  $0\le a_i<b_i<\infty$, 
 $1\le i \le N$, $a_1>0$, with $\bar A = [0, \frac{g}{128}] \times [a_1, b_1]\times \cdots \times [a_N, b_N]$,  
\begin{equation}\label{MLuniform}
    \bar{K}_{\bar A}  \doteq \inf_{(v, \bfz) \in \bar A}K_{(v,\bfz)} > 0.
\end{equation}

\end{theorem}

In proving the above it will be convenient to introduce a probability measure $\tilde \PP^*_{(v, \bfz)}$ that is mutually absolutely continuous to $\PP^*_{(v, \bfz)}$ and which is somewhat simpler to analyze. This measure corresponds to the law of the processes $(\bfB, V,\bfZ)$ given as in \eqref{eq:gapprocb} but with $V$ on the right side of equation for $Z_1$ replaced by the $0$ process.
Recall the path space $(\Om^*, \clf^*)$ and the coordinate processes $(\bfB, V, \bfZ)$ given on this space.  Let $\{\clf^*_t\}_{t \geq 0}$ be the filtration generated by these coordinate processes. For $(v,\bfz) \in \mathbb{R} \times \mathbb{R}_+^N$
 let $\tilde{\mathbb{P}}^*_{(v, \bfz)}$ be the probability measure on $(\Omega^*, \clf^*)$  such that  under 
 $\tilde{\mathbb{P}}^*_{(v, \bfz)}$ the following hold:
\begin{enumerate}[(i)]
    \item  $\bfB$ is the standard $N$-dimensional $\clf^*_t$-Brownian motion.
    \item For each $t \in [0,\infty)$, with  $L(t) =  \Gamma_1(\bfz + A\bfB(\cdot))(t)$,
    \begin{align}
    \begin{split}\label{drifteqn}
        \bfZ(t)  = \bfz + A\bfB(t)  + RL(t), \;\;
        V(t)  = v+ gt - L_1(t).
    \end{split}
    \end{align}
   \end{enumerate}
   
\editc{\subsection{Outline of Proof}

The proof of Theorem \ref{minorization} is organized as follows. In Lemma \ref{Novikov}, we establish a version of Novikov's criterion which allows us to relate $\PP^*_{(v, \bfz)}$ to $\tilde{\PP}^*_{(v, \bfz)}$ via Girsanov's Theorem. In Corollary \ref{girsanov}, we use the preceding lemma to invoke Girsanov's Theorem and make explicit the relation between the two measures.

We next prove a number of technical results in support of Theorem \ref{minorization}. In Lemma \ref{kild}, we establish a minorization condition for a `killed' version of $\bfZ$ under law $\tilde{\PP}^*_{(v, \bfz)}$, when $\bfZ(0)$ lies in a certain compact set $F$. In Lemma \ref{subdens}, we prove the existence of a subdensity for the supremum of Brownian motion over a compact time interval under certain constraints on its infimum and final location. This supremum, in turn, is connected to the local time $L_1$ via the Skorohod map. As under law $\tilde{\PP}^*_{(v, \bfz)}$, existence of a subdensity at a fixed time for $(\bfZ, V)$ is implied by that for $(\bfZ, L_1)$ (see \eqref{drifteqn}), the above two lemmas are crucial in proving Theorem \ref{minorization}. Lemmas \ref{H2sub1} and \ref{H2sub2} provide a version of the `support theorem' where a tractable event in terms of the driving Brownian motions is constructed under which the gap process $\bfZ$ at a prescribed time $\sn/4$ lies in $F$ almost surely under $\tilde{\PP}^*_{(v, \bfz)}$.

Lastly, we prove Theorem \ref{minorization}. Using the strong Markov property, we analyze the process pathwise between appropriately chosen stopping times. We first let $Z_1$ hit zero at time $\tau_1$ after which, under the event on the driving Brownian motions described in Lemma \ref{H2sub2}, the local time $L_1$ lies in a given Borel set and the gaps $\bfZ$ lie in the set $F$ introduced in Lemma \ref{kild} at time $\tau_1 + \sn/4$. Theorem \ref{minorization} now follows upon combining this and the minorization condition on the killed gap process obtained in Lemma \ref{kild}, which is used in analyzing the subsequent process path.


\subsection{Proof of Theorem \ref{minorization}}}
   
In order to relate $\PP^*_{(v, \bfz)}$ with $\tilde{\PP}^*_{(v, \bfz)}$ we establish the following integrability property which will be used to verify a variation of Novikov's criterion. {\cg In the following, $\tilde{\mathbb{E}}^*_{(v,\bfz)}$ denotes the expectation under the probability measure $\tilde\PP^*_{(v, \bfz)}$. Under  $\tilde{\PP}^*_{(v, \bfz)}$, the local times $L_i$, $1\le i \le N$ (and with $L_{N+1}=0$) have the following pathwise representation:
\begin{equation}\label{locrep}
\begin{aligned}
 L_1(t) &= \sup_{s \leq t}(-z_1 + \frac{1}{2}L_2(s) - B_1(s))^+,\\
 L_2(t) &= \sup_{s \leq t}(-z_2 + \frac{1}{2}L_3(s) + L_1(s) + B_1(s) - B_2(s))^+,\\
  L_i(t) &= \sup_{s \leq t}(-z_i + \frac{1}{2}(L_{i+1}(s) + L_{i-1}(s)) + B_{i-1}(s) - B_i(s))^+, \  \  i=3, \ldots, N.
\end{aligned}
\end{equation}

}

\begin{lemma}\label{Novikov}  For every $c\in (0, \infty)$ and $r\in \NN$, there is a $m \in \NN$  such that with $t_k = k/m$, $k=0, 1, \ldots \editc{,} rm-1$,
	for each $(v,\bfz) \in \mathbb{R} \times \mathbb{R}_+^N$,
$${\cg \tilde{\mathbb{E}}^*_{(v,\bfz)}\,e^{\frac{c}{2}\int_{t_{k}}^{t_{k+1}} V(s)^2 ds} < \infty.}$$
\end{lemma}

\begin{proof}
Fix $c\in (0,\infty)$ and $r\in \NN$.
 Also, fix $(v,\bfz) \in \mathbb{R} \times \mathbb{R}_+^N$. All equalities and inequalities  in the proof are almost sure with respect to the measure $\tilde{\mathbb{P}}^*_{(v,\bfz)}$. 
 %
%
%
Note that, for $t\ge 0$, by \eqref{locrep},
\begin{align*}
    L_1(t) & \leq \frac{1}{2}L_2(t) + \sup_{s \leq t}(-B_1(s)),\\
     L_2(t) & \leq \frac{1}{2}L_3(t) + L_1(t) + \sup_{s \leq t}(B_1(s) - B_2(s)),\\
    L_i(t) & \leq \frac{1}{2}(L_{i+1}(t) + L_{i-1}(t)) + \sup_{s \leq t}(B_{i-1}(s) - B_i(s)), \  \  i=3, \ldots, N.
\end{align*}
Define
\begin{equation}\label{bstar}
\begin{aligned}
    B^*_1(t) &=  \sup_{s \leq t}(-B_1(s)) \\
    B^*_i(t) &= \sup_{s \leq t}(B_{i-1}(s) - B_i(s)), \textnormal{ for }\, i = 2, \ldots, N \\
    \bfB^*(t) &= (B_1^*(t),...,B_N^*(t)).
\end{aligned}
\end{equation}
 Recall the matrices $U=I-R$ and  $W = (I-U)^{-1}$. Then, it is easy to verify that
\begin{equation}\label{Umat}
    U =
    \begin{pmatrix}
    0 & \frac{1}{2} & 0 & 0 & 0 & \cdots & 0 & 0 \\
    1 & 0 & \frac{1}{2} & 0 & 0 & \cdots & 0 & 0 \\
    0 & \frac{1}{2} & 0 & \frac{1}{2} & 0 & \cdots & 0 & 0 \\
    \vdots & \vdots & \vdots & \vdots & \vdots & \vdots & \vdots & \vdots \\
    0 & 0 & \cdots & \cdots & \cdots & \cdots & \frac{1}{2} & 0 
    \end{pmatrix}.
\end{equation}
 and so from the above inequalities we can write, for $t\ge 0$,
$$\bfL(t) \leq U\bfL(t) +\bfB^*(t).$$
In particular, recalling that $W$ can be written as an infinite sum of matrices with nonnegative entries, we have that
$$L_1(t) \leq (W\bfB^*(t))_1.$$
Now fix $m\in \NN$ which will be chosen suitably below. Define $t_k = k/m$, $k=0, 1, \ldots \editc{,} rm-1$.
Then, for any $k$ as above,
\begin{multline*}
    \int_{t_k}^{t_{k+1}}V(s)^2ds \leq m^{-1}\sup_{s \in [t_k, t_{k+1}]}V(s)^2 \\
     = m^{-1}\sup_{s \in [t_k, t_{k+1}]}(v+gs-L_1(s))^2 
     \le 2m^{-1}(|v|+rg)^2 + 2m^{-1} (L_1(r))^2.
\end{multline*}	
It then follows
\begin{align*}
\max_{0\le k \le rm-1} e^{\frac{c}{2}\int_{t_k}^{t_{k+1}}V(s)^2ds} & \leq c_1 e^{cm^{-1} (W\bfB^*(r))^2_1}	
\end{align*}
where
$c_1= e^{cm^{-1}(|v|+rg)^2}$.
  The expectation of the right side under ${\mathbb{P}}^*_{(v,\bfz)}$ (which is independent of $(v,\bfz) \in \RR\times \RR_+^N$)
  is finite for sufficiently large $m$. The result follows.
\end{proof}
For $(v, \bfz) \in\mathbb{R} \times \mathbb{R}_+^N$ and $r \in \NN$, with an abuse of notation, denote the projection of $\PP^*_{(v, \bfz)}$  [resp. $\tilde{\PP}^*_{(v, \bfz)}$] on
$\Om^r \doteq \clc([0,r]: \RR^N\times \RR \times \RR_+^N)$ once more as $\PP^*_{(v, \bfz)}$ [resp. $\tilde{\PP}^*_{(v, \bfz)}$]. Denote by $\clf^r$ the Borel $\sigma$-field on $\Om^r$. The coordinate processes $\bfB, V, \bfZ$ on 
$(\Om^r, \clf^r)$ and the canonical filtration $\{\clf^r_t\}_{0\le t \le r}$ are defined in an analogous manner. Denote by $\bfe_1$ the unit vector $(1,0, 0,\ldots, 0)'$ in $\RR^N$.

\begin{corollary}\label{girsanov}
	Fix $r\in \NN$.
Define for $t \in [0,r]$,  real measurable maps $\cle(t)$ on $(\Om^r, \clf^r)$ as
\begin{equation}\label{expmart}
\mathcal{E}(t) \doteq e^{- \sum_{i=1}^N \int_0^t V(s)(A^{-1} \bfe_1)_i dB_i(s) - \frac{|A^{-1}\bfe_1|^2}{2}\int_0^ t V(s)^2ds }.
\end{equation}
Then for every $(v, \bfz) \in \mathbb{R} \times \mathbb{R}_+^N$, $\tilde{\EE}^*_{(v, \bfz)}[\cle(r)] =1$ and for every $F \in \clb (\clc([0,r]: \RR \times \RR_+^N))$ 
\begin{equation}\label{Girs}
{\cg    \mathbb{P}^*_{(v,\bfz)}((V,\bfZ)\in F) = \tilde{\mathbb{E}}^*_{(v,\bfz)}[\textbf{1}_{\{(V,\bfZ) \in F\}}\mathcal{E}(r)].}
\end{equation}

\end{corollary}

\begin{proof}

Fix $(v,\bfz) \in \mathbb{R} \times \mathbb{R}_+^N$ and $r\in \NN$. For $t \in  [0,r]$, define 
\begin{equation*}
    \tilde{\bfB}(t) \doteq \bfB(t) + \int_0^t V(s)A^{-1}\bfe_1 ds .
\end{equation*}
By Lemma \ref{Novikov} with $c = |A^{-1}\bfe_1|^2$ and (a slight modification of) \cite[Corollary 3.5.14]{KarShre}, it follows that $\{\cle(t)\}_{0\le t \le r}$ is a martingale with respect to the filtration $\{\clf^r_t\}_{0 \le t \le r}$ under the probability measure $\tilde{\mathbb{P}}^*_{(v,\bfz)}$. Hence, from Girsanov's theorem, $\{\tilde{\bfB}(s)\}_{0\le s \le r}$ is a Brownian motion under the probability measure ${\mathbb{Q}}^*_{(v,\bfz)}$ defined by $d{\mathbb{Q}}^*_{(v,\bfz)} \doteq \cle(r)d\tilde{\mathbb{P}}^*_{(v,\bfz)}$. Also, under the measure ${\mathbb{Q}}^*_{(v,\bfz)}$ we have
\begin{align*}
    \bfZ(t) & = \Gamma_2(\bfz + A\bfB(\cdot))(t) 
     = \Gamma_2\left(\bfz  + A\left(\tilde{\bfB}(\cdot) - \int_0^{\cdot} V(s)A^{-1}\bfe_1ds\right)\right)(t) \\
    & = \Gamma_2\left(\bfz  + A\tilde{\bfB}(\cdot) - \int_0^{\cdot} V(s)ds\,\bfe_1\right)(t).
\end{align*}
By the unique solvability given in Theorem \ref{thm:wellposed} and the definition of $\mathbb{P}^*_{(v,\bfz)}$
it now follows that the law of $(V,\bfZ)$ under  ${\mathbb{Q}}^*_{(v,\bfz)}$ is same as 
that under $\mathbb{P}^*_{(v,\bfz)}$. The result follows.
\end{proof}

We next prove several  technical estimates that will be needed in the proof of Theorem \ref{minorization}. 

\begin{lemma}\label{kild}
Let

\begin{equation}\label{Fset}
    F  \doteq  [\frac{\gn}{16},\frac{\gn}{4}] \times [\frac{\gn}{10},2\gn] \times [\frac{3\gn}{4},2\gn]^{N-2}.
\end{equation}
Let $G \subset (\mathbb{R}_+^N)^o$ be an open and bounded domain with $\clc^{2}$ boundary such that 
$$F \subset  F_1 \doteq [\frac{\gn}{16},\gn] \times [\frac{\gn}{10},2\gn] \times [\frac{3\gn}{4},2\gn]^{N-2} \subset G.$$ Let $\sigma_F \doteq \inf_{x \in F,\, y \in \partial G}|A^{-1}(x-y)| $ and choose $ \epsilon > 0$ so that  $G_{1} \doteq \{x \in G: \inf_{y \in \partial G}|A^{-1}(x-y)| > \epsilon\}$
satisfies $G\supset G_{1} \supset F_1 \supset F$.
Define on $(\Om^*, \clf^*)$,
 $\tau_G = \inf\{t \geq 0: \bfZ(t) \notin G\}$. 
 Also, fix a `cemetery point' $\partial^* \in (\RR_+^N)^c$ and define the `killed process' $\{\bfZ^*(t)\}$ by
 \begin{align}\label{eq:zstar}
    \bfZ^*(t) \doteq 
	\begin{cases}
		\bfZ(t) &  \mbox{ if } t < \tau_G \\
     \partial^*& \mbox{ if } t \geq \tau_G,
	 \end{cases}
\end{align}
Then, there is a $c_G \in (0,\infty)$ such that  for any $J \in \mathcal{B}(\mathbb{R}_+^N)$, 
\begin{equation}
    \inf_{s \in [\frac{\sn}{4},\en], (v,\bfz) \in \RR\times F}\tilde{\mathbb{P}}^*_{(v,\bfz)}(\bfZ^*(s) \in J) \geq c_G\lambda(J \cap G_{1}).
\end{equation}

\end{lemma}

\begin{proof} Fix $s \in [\frac{\sn}{4}, \en], (v,\bfz) \in \RR\times F$ and $J \in \mathcal{B}(\mathbb{R}_+^N)$ with $\lambda(J\cap G_{1}) > 0$. 
	Since, under $\tilde{\mathbb{P}}^*_{(v,\bfz)}$, $\bfZ(t)= \bfz + A\bfB(t)$ until the first time it has hit the boundary of the positive orthant,
\begin{align*}
	\tilde{\mathbb{P}}^*_{(v,\bfz)}(\bfZ^*(s) \in J) &= \tilde{\mathbb{P}}^*_{(v,\bfz)}(\bfZ^*(s) \in J \cap G)\\
	&= \tilde{\mathbb{P}}^*_{(v,\bfz)}(\bfz + A \bfB(s) \in J\cap G, \bfz + A\bfB(u) \in G \mbox{ for all } u \le s)\\
	&= \tilde{\mathbb{P}}^*_{(v,\bfz)}(A^{-1}\bfz + \bfB(s) \in A^{-1}(J \cap G), \, A^{-1}\bfz + \bfB(u) \in A^{-1}(G),\mbox{ for all } u \le s ).
\end{align*}
Denote the transition probability density at time $t$ of an $N$-dimensional standard Brownian motion in $A^{-1}G$, started from $x$ and killed at the boundary of $A^{-1}G$, by $p_t(x,\cdot)$.
Then from the above identities it follows
\begin{align}\label{eq:denseq}
	\tilde{\mathbb{P}}^*_{(v,\bfz)}(\bfZ^*(s) \in J) &= \int_{A^{-1}(G\cap J)}p_s(A^{-1}\bfz, y)dy.
\end{align}

%
%

%

From \cite[Theorem 1.1]{ZhangHeatKer} we have that there exists $T > 0$ and $c_1, c_2 \in (0,\infty)$  such that for all $x, y \in A^{-1}G$:
\begin{equation}\label{densltT}
    p_t(x,y) \geq \editc{\left(\frac{\rho(x)\rho(y)}{t}\wedge 1\right)}\frac{c_1}{t^{N/2}}e^{-c_2|x-y|^2/t},\,\,\,\,\,t \in [0, T]
\end{equation}
\begin{equation}\label{densgtT}
    p_t(x,y) \geq c_1\rho(x)\rho(y)e^{-c_2 t},\,\,\,\,\,\,t \in (T,\infty),
\end{equation}
where $\rho(x) = \inf_{r \in \partial G}|x-A^{-1}r|$.
  Note that there is a $\eta \in (0,\infty)$ such that
  $$\lambda(A^{-1}(C)) = \eta \lambda(C) \mbox{ for all } C \in \clb(\RR^N).$$
We now estimate the right side of \eqref{eq:denseq}.
First suppose that $s > T$. Then since $\bfz \in F$ and $s \le 1$,
\begin{multline*}
   \int_{A^{-1}(G\cap J)}p_s(A^{-1}\bfz, y)dy 
     \geq \int_{A^{-1}(G\cap J)}c_1\rho(A^{-1}\bfz)\rho(y)e^{-c_2s}dy \\
     \geq  c_1\sigma_Fe^{-c_2}\int_{A^{-1}( G_{1}\cap J)}\rho(y)dy
    \geq c_1\sigma_Fe^{-c_2}\epsilon\eta\lambda( G_{1}\cap J).
\end{multline*}
Letting
$c_{G,1} = c_1\sigma_Fe^{-c_2}\epsilon\eta$,
we have from \eqref{eq:denseq}, when $s> T$,
\begin{align}\label{KPminor}
\tilde{\mathbb{P}}^*_{(v,\bfz)}(\bfZ^*(s) \in J) \geq c_{G,1}\lambda(J\cap G_1).
\end{align}
Now consider the case $s \leq T$. Then, a similar estimate shows,
\begin{align*}
\int_{A^{-1}(J\cap G)}p_s(A^{-1}\bfz,y)dy 
    & \geq \int_{A^{-1}(J\cap G)}\editc{\left(\frac{\rho(A^{-1}\bfz)\rho(y)}{s}\wedge 1\right)}\frac{c_1}{s^{N/2}}e^{-c_2|A^{-1}\bfz-y|^2/s}dy \\
    & \geq c_{G,2}\,\lambda(J\cap G_1) \\
\end{align*}
where
$c_{G,2} = \eta c_1((\epsilon\sigma_F)\wedge 1) e^{-4c_2(\sn)^{-1}
	\sup_{x,y \in A^{-1}G}|x-y|^2}.$
Thus,  once more from \eqref{eq:denseq}, when $s\le T$,
\begin{align}\label{KPminorb}
\tilde{\mathbb{P}}^*_{(v,\bfz)}(\bfZ^*(s) \in J) \geq c_{G,2}\lambda(J\cap G_1).
\end{align}
Setting $c_G = c_{G,1}\wedge c_{G,2}$, we have the result on combining \eqref{KPminor} and \eqref{KPminorb}. 
\end{proof}

For $z_1 \in \mathbb{R}$, let $\mathcal{P}_{z_1}$ denote a probability measure on $\Omega^*$ under which the coordinate process $\{B_1(t)\}$ is a standard Brownian motion starting at $z_1$. We will  use similar notation for the corresponding expectation. 
\begin{lemma}\label{subdens} There exists a  $\mathscr{K} \in (0, \infty)$ such that for every $I \in \mathcal{B}(\mathbb{R})$,
\begin{equation}
 \mathcal{P}_{0} \left(\sup_{0\le u \le \frac{\sn}{4}} B_1(u) \in I,\, \inf_{0\le u \leq \frac{\sn}{4}} B_1(u) > -\frac{6\gn}{10},\,B_1(\frac{\sn}{4}) \in [-\frac{\gn}{8},-\frac{\gn}{16}]\right) \geq \mathscr{K}\lambda(I \cap [0,\frac{\gn}{63}]).
\end{equation}
\end{lemma}
\begin{proof}
Let $I \in \mathcal{B}(\mathbb{R})$.  We assume  without loss of generality that, $I \subset [0,\frac{\gn}{63}]$ and $I$ is of the form $I = [\beta_1, \beta_2] \subset \mathbb{R}_+$ for $0 \leq \beta_1 \leq \beta_2$ (the choice of $\mathscr{K}$ will be independent of $\beta_1, \beta_2$). 
Let $\gamma \doteq \frac{\gn}{2}(-\frac{1}{8}-\frac{1}{16})$ be the midpoint of $ [-\frac{\gn}{8},-\frac{\gn}{16}]$. 
For a level $c \in \RR$, let $\tau_{c} \doteq  \inf\{t \geq 0: B_1(t) = c\}$.
Define $\sigma \doteq \tau_{-6\gn/10}$ and $\tau^{\beta}_i \doteq \tau_{\beta_i}$ for $i=1,2$.
Then
\begin{align*}
    &\mathcal{P}_{0}(\sup_{u \leq \frac{\sn}{4}}(B_1(u)) \in I,\, \inf_{u \leq \frac{\sn}{4}} B_1(u) > -\frac{6\gn}{10},\,B_1(\frac{\sn}{4}) \in [-\frac{\gn}{8},-\frac{\gn}{16}]) \\
    & = \mathcal{P}_{0}(\sup_{u \leq \frac{\sn}{4}}(B_1(u)) \in I,\, \sigma > \frac{\sn}{4},\,B_1(\frac{\sn}{4}) \in [-\frac{\gn}{8},-\frac{\gn}{16}]).
\end{align*}
Using the strong Markov property of the Brownian motion, we obtain,
\begin{align*}
    &\mathcal{P}_{0}(\sup_{u \leq \frac{\sn}{4}}(B_1(u)) \in I,\, \sigma > \frac{\sn}{4},\,B_1(\frac{\sn}{4}) \in [-\frac{\gn}{8},-\frac{\gn}{16}]) \\
    & \geq \mathcal{P}_{0}(\tau_1^{\beta} \leq \frac{\sn}{8} \wedge \sigma,\, \sup_{u \leq \frac{\sn}{4}}(B_1(u)) \in I,\, \sigma > \frac{\sn}{4},\,B_1(\frac{\sn}{4}) \in [-\frac{\gn}{8},-\frac{\gn}{16}]) \\
    & = \mathcal{P}_{0}(\textbf{1}_{\{\tau_1^{\beta} \leq \frac{\sn}{8} \wedge \sigma\,\}}\Theta(\tau_1^{\beta})),
\end{align*}
where, for $t \in [0, \frac{\sn}{8}]$,
\begin{align}
\Theta(t) \doteq \mathcal{P}_{\beta_1}(\sup_{u \leq \frac{\sn}{4}-t}(B_1(u)) \leq \beta_2,\, \sigma > \frac{\sn}{4}-t,\,B_1(\frac{\sn}{4}-t) \in [-\frac{\gn}{8},-\frac{\gn}{16}]).
\end{align}
By another application of the strong Markov property, for $t \in [0,\frac{\sn}{8}]$,
\begin{align*}
    \Theta(t) &\geq \mathcal{P}_{\beta_1}(\tau_{\gamma} \leq  \tau_2^{\beta} \wedge \frac{\sn}{16},\, B_1(s) \in [-\frac{\gn}{8},-\frac{\gn}{16}] \mbox{ for all } s \in [\tau_{\gamma}, \frac{\sn}{4}-t]) \\
    & = \mathcal{P}_{\beta_1}(\textbf{1}_{\{\tau_{\gamma} \leq  \tau_2^{\beta} \wedge \frac{\sn}{16}\}}\,\Theta'(\tau_{\gamma})) 
\end{align*}
where  for $u \in [0, \frac{\sn}{16}]$,
\begin{align}
    \Theta'(u) \doteq \mathcal{P}_{\gamma}(B_1(s) \in [-\frac{\gn}{8},-\frac{\gn}{16}] \mbox{ for all } s \in [0, \frac{\sn}{4}-t-u]).
\end{align}
Thus letting
$$\kappa_1 \doteq \mathcal{P}_{\gamma}( B_1(s) \in [-\frac{\gn}{8},-\frac{\gn}{16}] \mbox{ for all } s \in [0, \frac{\sn}{4}])$$
we have that, for $t \in [0,\frac{\sn}{8}]$,
\begin{equation}
\Theta(t) \ge \kappa_1 \mathcal{P}_{\beta_1}(\tau_{\gamma} \leq  \tau_2^{\beta} \wedge \frac{\sn}{16}).
\end{equation}
Also, by an application of the reflection principle,
\begin{align*}
    \mathcal{P}_{\beta_1}(\tau_{\gamma} \leq  \tau_2^{\beta} \wedge \frac{\sn}{16}) & = \mathcal{P}_{\beta_1}(\tau_{\gamma} \leq \frac{\sn}{16}) - \mathcal{P}_{\beta_1}( \tau_2^{\beta} < \tau_{\gamma} \leq \frac{\sn}{16}) \\
    & = \mathcal{P}_{\beta_1}(\tau_{\gamma} \leq \frac{\sn}{16}) - \mathcal{P}_{\beta_1 + 2(\beta_2-\beta_1)}( \tau_{\gamma} \leq \frac{\sn}{16}).
\end{align*}
From the  definition of the stopping times we see,
\begin{align*}
    \mathcal{P}_{\beta_1}(\tau_{\gamma} \leq \frac{\sn}{16}) & = \mathcal{P}_{0}(\sup_{u \leq \frac{\sn}{16}}(B_1(u)) \geq \beta_1 - \gamma ), \\
    \mathcal{P}_{\beta_1 + 2(\beta_2-\beta_1)}( \tau_{\gamma} \leq \frac{\sn}{16}) & = \mathcal{P}_{0}(\sup_{u \leq \frac{\sn}{16}}(B_1(u)) \geq \beta_1 + 2(\beta_2-\beta_1) - \gamma ).
\end{align*}
Using the explicit form for the probability density for the law of the maximum of a Brownian motion, we then obtain,
\begin{align*}
    \mathcal{P}_{\beta_1}(\tau_{\gamma} \leq \frac{\sn}{16}) - \mathcal{P}_{\beta_1 + 2(\beta_2-\beta_1)}(\tau_{\gamma} \leq \frac{\sn}{16}) & = \int_{\beta_1-\gamma}^{\beta_1 + 2(\beta_2-\beta_1)-\gamma}\frac{4\sqrt{2}}{\sqrt{\pi\sn}}e^{-8z^2/\sn} \\
    & \geq \frac{8\sqrt{2}}{\sqrt{\pi\sn}}\inf_{\beta_1-\gamma\leq z\leq \beta_1+2(\beta_2-\beta_1)-\gamma}e^{-8z^2/\sn}(\beta_2-\beta_1).
\end{align*}
Since
$$\beta_1 + 2(\beta_2-\beta_1) - \gamma \leq 2\beta_2 - \gamma \leq \frac{2\gn}{63} + \frac{3\gn}{32} \le \frac{\gn}{4},$$
we have
$$\inf_{\beta_1-\gamma\leq z\leq 2(\beta_2-\beta_1)-\gamma}e^{-8z^2/\sn} \geq e^{-\frac{\gn^2}{2\sn}}.$$
Thus, for $t \in [0, \sn/8]$,
$$\Theta(t) \ge \kappa_1 \mathcal{P}_{\beta_1}(\tau_{\gamma} \leq  \tau_2^{\beta} \wedge \frac{\sn}{16}) 
\ge \kappa_1 \frac{8\sqrt{2}}{\sqrt{\pi\sn}} e^{-\frac{\gn^2}{2\sn}} (\beta_2-\beta_1).$$
Finally, observe that, as $I \subset [0, \frac{\gn}{63}]$,
\begin{align*}
    \mathcal{P}_{0}(\tau_1^{\beta} \leq \frac{\sn}{8} \wedge \sigma\,) \geq \mathcal{P}_{0}(\sup_{u \leq \frac{\sn}{8}}(B_1(s)) > \frac{\gn}{63},\,\,\inf_{u \leq \frac{\sn}{8}}(B_1(s)) > -\frac{6\gn}{10}) \doteq \kappa_2.
\end{align*}
The result now follows on setting
$
\mathscr{K} = \kappa_1\kappa_2\frac{8\sqrt{2}}{\sqrt{\pi\sn}}e^{-\frac{\gn^2}{2\sn}}.$
\end{proof}

For $0\le s \le 1$ and $(z_2, z_3, \ldots , z_N) \in \RR_+^{N-1}$, define
\begin{align*}
    \hat{L}_1(s) &= \sup_{u \leq s}(-B_1(u)), \;\;\;  \hat{Z}_1(s) = B_1(s) + \hat{L}_1(s) \\
    \hat{L}_i(s) &= \sup_{u \leq s}(-z_i+B_{i-1}(u)-B_i(u))^+, \  i = 2, \ldots , N .\\
  \end{align*}

\begin{lemma}\label{H2sub1} Let $I \in \mathcal{B}(\mathbb{R})$ be such that $I \subset [0,\frac{\gn}{63}]$ and $(v, {\bfz}) \in \RR \times \{0\} \times [\gn,\frac{3\gn}{2}]^{N-1}$. Let $H \in \clf^*$. Then the following are equivalent:
	\begin{enumerate}[(a)]
		\item On $H$, $\tilde{\mathbb{P}}^*_{(v,{\bfz})}$ a.s.,
		\begin{inparaenum}[(i)]
	    \item $L_1(\frac{\sn}{4}) \in I$,
	    \item $L_i(\frac{\sn}{4}) \leq \frac{\gn}{6}, \mbox{ for all } i = 2, \ldots ,N$,
	    \item $\sup_{0 \leq u\leq \frac{\sn}{4}}B_1(u)  < \frac{6\gn}{10}$,
	    \item $\sup_{0 \leq u \leq \frac{\sn}{4}}|B_i(u)|  < \frac{\gn}{8}, \mbox{ for all } i = 2,\ldots, N$,
	    \item $Z_1(\frac{\sn}{4})  \in [\frac{\gn}{16},\frac{\gn}{4}]$.	
		\end{inparaenum}
		\item On $H$, $\tilde{\mathbb{P}}^*_{(v,{\bfz})}$ a.s.,
		\begin{inparaenum}[(i')]
		    \item $\hat L_1(\frac{\sn}{4}) \in I$,
		    \item $\hat L_i(\frac{\sn}{4}) \leq \frac{\gn}{6}, \mbox{ for all } i = 2, \ldots,N$,
		    \item $\sup_{0 \leq u\leq \frac{\sn}{4}}B_1(u)  < \frac{6\gn}{10}$,
		    \item $\sup_{0 \leq u \leq \frac{\sn}{4}}|B_i(u)|  < \frac{\gn}{8}, \mbox{ for all } i = 2,\ldots, N$,
		    \item $\hat Z_1(\frac{\sn}{4})  \in [\frac{\gn}{16},\frac{\gn}{4}]$.
			\end{inparaenum}
	\end{enumerate}
	Furthermore, under these equivalent conditions, on $H$, $\tilde{\mathbb{P}}^*_{(v,{\bfz})}$ a.s., $L_1(\sn/4) = \hat L_1(\sn/4)$ and
	$L_i(\sn/4)=0$ for $i= 2, \ldots, N$.
\end{lemma}

\begin{proof}
Fix $(v,{\bfz}) \in \RR \times \{0\} \times [\gn,\frac{3\gn}{2}]^{N-1}$. 
Noting that $z_i \geq \gn$ for $i = 2, \ldots \editc{,} N$,  we see that,  when conditions $(i) - (v)$ hold, for all $u \leq \frac{\sn}{4}$,
\begin{align*}
    -z_i+B_{i-1}(u)-B_i(u)+\frac{1}{2}(L_{i-1}(u)+L_{i+1}(u)) &\leq -\gn + \frac{\gn}{4} + \frac{\gn}{6}\leq 0, \; i = 3, \ldots, N, \\
        -z_2+B_{1}(u)-B_2(u)+\frac{1}{2}L_{3}(u)+L_{1}(u) &\leq -\gn + \frac{6\gn}{10}+\frac{\gn}{8} + \frac{\gn}{12} + \frac{\gn}{63}  \leq 0.
\end{align*}
 Hence, when conditions $(i) - (v)$ hold on $H$, by \eqref{locrep}, $\tilde{\mathbb{P}}^*_{(v,{\bfz})}$ a.s.,  $L_i(\frac{\sn}{4}) = \hat{L}_i(\frac{\sn}{4}) =0$ for $i = 2, \ldots , N$ which in  turn says that $L_1(\frac{\sn}{4}) = \hat{L}_1(\frac{\sn}{4})$ and 
 $Z_1(\frac{\sn}{4}) = \hat{Z}_1(\frac{\sn}{4})$. Thus in this case  $(i') - (v')$ hold on $H$, $\tilde{\mathbb{P}}^*_{(v,{\bfz})}$ a.s.

On the other hand, suppose that $(i') - (v')$ hold on $H$, $\tilde{\mathbb{P}}^*_{(v,{\bfz})}$ a.s.
Consider the stopping times, 
\begin{align*}
    \nu_2 & = \inf\{t \geq 0: -z_2 + B_{1}(t) - B_2(t) +\frac{1}{2}L_{3}(t)+L_{1}(t) \geq 0\} \\
    \nu_i & = \inf\{t \geq 0: -z_i + B_{i-1}(t) - B_i(t) +\frac{1}{2}(L_{i+1}(t)+L_{i-1}(t)) \geq 0\}, \;\; i = 3, \ldots , N 
   \end{align*}
   and let  $\nu  = \min_{2 \leq i \leq N}\nu_i$.

Then, for $s \leq \nu$, $L_i(s) =\hat L_i(s) = 0,\, \mbox{ for all } i = 2, ..., N$ and so $L_1(s) = \hat{L}_1(s)$ and $Z_1(s) = \hat{Z}_1(s)$. Thus, if $\nu \leq \frac{\sn}{4}$, 
\begin{align*}
    -z_2 + B_1(\nu) - B_2(\nu) +\frac{1}{2} L_3(\nu) +  L_1(\nu) 
    =  -z_2 + B_1(\nu) - B_2(\nu) + \hat{L}_1(\nu) 
    \leq  -\gn +\frac{6\gn}{10} + \frac{\gn}{8} + \frac{\gn}{63} < 0,
\end{align*}
and, for $i = 3, \ldots , N$,
\begin{align*}
    -z_i + B_{i-1}(\nu) - B_i(\nu) + \frac{1}{2}(L_{i+1}(\nu)+L_{i-1}(\nu)) 
    =  -z_i + B_{i-1}(\nu) - B_i(\nu) 
    \leq  -\gn +\frac{\gn}{8} + \frac{\gn}{8} < 0.
\end{align*}
This contradicts the definition of $\nu$ and consequently we must have that $\nu > \frac{\sn}{4}$. Thus 
$(i) - (v)$ hold on $H$, $\tilde{\mathbb{P}}_{(v,{\bfz})}$ a.s., and the result follows.
\end{proof}

Recall the set $F$ introduced in Lemma \ref{kild}.
\begin{lemma}\label{H2sub2} Let $I \in \mathcal{B}(\mathbb{R})$ be such that $I \subset [0,\frac{\gn}{63}]$ and $(v,{\bfz}) \in \RR \times \{0\} \times [\gn,\frac{3\gn}{2}]^{N-1}$. Let $H \in \clf^*$ and suppose the equivalent conditions of Lemma \ref{H2sub1} hold on $H$.
Then, on $H$, $\tilde{\mathbb{P}}^*_{(v,{\bfz})}$ a.s., $\bfZ(\frac{\sn}{4}) \in F$.
\end{lemma}

\begin{proof}
Under assumptions of the lemma, on $H$, $\tilde{\mathbb{P}}_{(v,{\bfz})}$ a.s.,
\begin{align*}
    Z_2(\sn/4) &= z_2 + B_2(\sn/4) - B_1(\sn/4) -\frac{1}{2}L_3(\sn/4)+L_2(\sn/4)-L_1(\sn/4) \\
    & = z_2 + B_2(\sn/4) - B_1(\sn/4) - L_1(\sn/4) 
     \geq \gn -\frac{\gn}{8}-\frac{6\gn}{10}-\frac{\gn}{63} \geq \frac{\gn}{10}, \\
     Z_i(\sn/4) &= z_i + B_i(\sn/4) - B_{i-1}(\sn/4) -\frac{1}{2}(L_{i-1}(\sn/4)+L_{i+1}(\sn/4))+L_i(\sn/4) \\
     & = z_i + B_i(\sn/4) - B_{i-1}(\sn/4) \geq \frac{3\gn}{4},  \;\;i = 3, \ldots, N.
\end{align*}
From Lemma \ref{H2sub1}, under the assumptions of the current lemma, $L_1(\sn/4) = \hat L_1(\sn/4)$ and so
\begin{align*}
    \frac{\gn}{63} & \geq L_1(\sn/4) = \hat{L}_1(\sn/4) 
     = \sup_{u \leq \sn/4}(-B_1(u)) \geq -B_1(\sn/4).
\end{align*}
Thus we have  the upper bound,
\begin{align*}
    Z_2(\sn/4) &= z_2 + B_2(\sn/4) - B_1(\sn/4) - L_1(\sn/4) \leq \frac{3\gn}{2} + \frac{\gn}{8} + \frac{\gn}{63} \leq 2\gn,
\end{align*}
Also, for $i = 3, ..., N$,
\begin{align*}
     Z_i(\sn/4) &= z_i + B_i(\sn/4) - B_{i-1}(\sn/4) \leq \frac{3\gn}{2} + \frac{\gn}{8} + \frac{\gn}{8} \le 2\gn.
\end{align*}
Hence, $(Z_2(\sn/4), \ldots, Z_N(\sn/4)) \in  [\frac{\gn}{10},2\gn] \times [\frac{3\gn}{4},2\gn]^{N-2}$.
Also, under the conditions of the lemma $Z_1 \in [\frac{\gn}{16},\frac{\gn}{4}]$. Thus 
$\bfZ(\sn/4) \in F$ and the lemma is proved.
\end{proof}

We can now complete the proof of Theorem \ref{minorization}.\\ \ \\

\textbf{Proof of Theorem \ref{minorization}.}
Recall $F, G$ and $G_1$ from Lemma \ref{kild}.
We will prove the theorem with $D \doteq D_1\times G_1$ where $D_1 = [0, g/128]$.
Let $(v,\bfz) \in [0,\frac{g}{128}] \times (0,\infty) \times \RR_+^{N-1}$. All equalities and inequalities of random quantities in the proof are under the measure $\tilde{\mathbb{P}}^*_{(v,\bfz)}$.
Let $r \in [\sn,\en]$ be given. It suffices to establish the estimate in \eqref{MLminor} for  $S \in \mathcal{B}(\mathbb{R} \times \mathbb{R}_+^N)$ of the form $S = I \times J,\,\,\,I \in \mathcal{B}(\mathbb{R}),\,\, J \in \mathcal{B}(\mathbb{R}_+^N)$ with {\cg $I \subseteq D_1$ and $J \subseteq G_1$, for a choice of the constant $K_{(v,\bfz)}$ independent of $I,J$.} 
For such an $S$, letting 
    $\tilde{B}(t) \doteq \sum_{i=1}^N (A^{-1})_{i,1}B_i(t)$,
by Corollary \ref{girsanov},
\begin{align}
\begin{split}\label{minorineq1}
  \mathbb{P}^r((v, \bfz), S)   & = \tilde{\mathbb{E}}^*_{(v,\bfz)} \textbf{1}_{\{V(r) \in I, \bfZ(r) \in J\}}\mathcal{E}(1 )  = \tilde{\mathbb{E}}^*_{(v,\bfz)} \textbf{1}_{\{V(r) \in I, \bfZ(r) \in J\}}\mathcal{E}(r)\\
     & = \,\,\tilde{\mathbb{E}}^*_{(v,\bfz)} \textbf{1}_{\{V(r) \in I, \bfZ(r) \in J\}}e^{ -\int_0^rV(s)d\tilde{B}(s) - \frac{|A^{-1}\bfe_1|^2}{2}\int_0^rV(s)^2ds }.
\end{split}
\end{align}
On the set $\{V(r) \in I\}$,  $L_1(r) = gr - V(r) + v \leq gr + v \leq gr + \frac{g}{128},$ so that by monotonicity, $L_1(s) \leq g(r+\frac{1}{128})$ for all $s \leq r$. This implies that, on this set, for $s \leq r,$ $-2g \leq -g(r + \frac{1}{128}) \leq V(s) \leq gr + v \leq 2g,$ i.e., $|V(s)| \leq 2g$.
Using this estimate in \eqref{minorineq1} we get
\begin{equation}\label{eq:mainest}
 \mathbb{P}^r((v, \bfz), S)  \ge e^{-2|A^{-1}\bfe_1|^2g^2}\tilde{\mathbb{E}}^*_{(v,\bfz)} \textbf{1}_{\{V(r) \in I, \bfZ(r) \in J\}}e^{ -\int_0^rV(s)d\tilde{B}(s) }.
\end{equation}
By It\^{o}'s formula,
\begin{align*}
    -\int_0^rV(s)d\tilde{B}(s) & = \int_0^r\tilde{B}(s)dV(s)  - V(r)\tilde{B}(r) \\
    & = -\int_0^r\tilde{B}(s)dL_1(s) + g\int_0^r\tilde{B}(s)ds - V(r)\tilde{B}(r).
\end{align*}
Define the stopping time
\begin{equation*}
    \tau_1 \doteq \inf\{t \geq 0: Z_1(t) = 0\},
\end{equation*}
and let
\editc{
\begin{align*}
    H \doteq & \left\lbrace \tau_1 \leq \frac{\sn}{4}, (Z_2(\tau_1), \ldots , Z_N(\tau_1)) \in [\gn,\frac{3\gn}{2}]^{N-1}, L_1(\tau_1 + \frac{\sn}{4}) \in v+gr - I,\right.\\
    &\left.\quad  \bfZ(\tau_1+\frac{\sn}{4}) \in F,  \bfZ(s) > 0, \mbox{ for all } s \in [\tau_1+\frac{\sn}{4}, r],
	\,\, \bfZ(r) \in J \right\rbrace.
\end{align*}
}
Note that
\begin{equation}\label{hcontain}
H \subset \{(V(r),\bfZ(r)) \in I \times J\}.
\end{equation}
On $H$ we have, 
\begin{align*}
    - V(r)\tilde{B}(r) & \geq- 2g|\tilde{B}(r)| 
     \geq  - 2g|\tilde{B}(r) - \tilde{B}(\tau_1 + \frac{\sn}{4})| 
     -2g|\tilde{B}(\tau_1+\frac{\sn}{4}) - \tilde{B}(\tau_1)| - 2g|\tilde{B}(\tau_1)|.
\end{align*}
In addition, on $H$,
\begin{align*}
    g\int_0^r\tilde{B}(s)ds & = g\int_0^{\tau_1}\tilde{B}(s)ds + g\int_{\tau_1}^{\tau_1+\frac{\sn}{4}}(\tilde{B}(s) - \tilde{B}(\tau_1))ds + \frac{g\sn}{4}\tilde{B}(\tau_1) \\
    &  +  g\int_{\tau_1+\frac{\sn}{4}}^r(\tilde{B}(s) - \tilde{B}(\tau_1+\frac{\sn}{4}))ds + g\tilde{B}(\tau_1+\frac{\sn}{4})(r-(\tau_1+\frac{\sn}{4})) \\
    & = g\int_0^{\tau_1}\tilde{B}(s)ds + g\int_{\tau_1}^{\tau_1+\frac{\sn}{4}}(\tilde{B}(s) - \tilde{B}(\tau_1))ds + g\tilde{B}(\tau_1)(r-\tau_1) \\
    &  +  g\int_{\tau_1+\frac{\sn}{4}}^r(\tilde{B}(s) - \tilde{B}(\tau_1+\frac{\sn}{4}))ds + g(\tilde{B}(\tau_1+\frac{\sn}{4}) - \tilde{B}(\tau_1))(r-(\tau_1+\frac{\sn}{4})).
\end{align*}
Also, by the definition of $\tau_1$, on  $H$,

\begin{align*}
    -\int_0^r \tilde{B}(s)dL_1(s) & = -\int_{\tau_1}^{\tau_1+\frac{\sn}{4}} \tilde{B}(s)dL_1(s) \\
    & = -\int_{\tau_1}^{\tau_1+\frac{\sn}{4}} (\tilde{B}(s)-\tilde{B}(\tau_1))dL_1(s) - \tilde{B}(\tau_1)(L_1(\tau_1+\frac{\sn}{4})-L_1(\tau_1)) \\
    & {\geq  -\int_{\tau_1}^{\tau_1+\frac{\sn}{4}} (\tilde{B}(s)-\tilde{B}(\tau_1))dL_1(s) - g |\tilde{B}(\tau_1)|.}
\end{align*}

Now let
\begin{align*}
    U_1 & \doteq  g\int_{\tau_1+\frac{\sn}{4}}^r(\tilde{B}(s) - \tilde{B}(\tau_1+\frac{\sn}{4}))ds -2g|\tilde{B}(r)-\tilde{B}(\tau_1 + \frac{\sn}{4})|,\\
    U_2 & \doteq g(\tilde{B}(\tau_1+\frac{\sn}{4}) - \tilde{B}(\tau_1))(r-(\tau_1+\frac{\sn}{4}))  -2g|\tilde{B}(\tau_1+\frac{\sn}{4}) - \tilde{B}(\tau_1)|\\
    &\,\,\, - \int_{\tau_1}^{\tau_1+\frac{\sn}{4}} (\tilde{B}(s)-\tilde{B}(\tau_1))dL_1(s) + g\int_{\tau_1}^{\tau_1+\frac{\sn}{4}}(\tilde{B}(s) - \tilde{B}(\tau_1))ds,\\
    U_3 & \doteq {- 3g|\tilde{B}(\tau_1)|} + g\int_0^{\tau_1}\tilde{B}(s)ds +  g\tilde{B}(\tau_1)(r-\tau_1).
\end{align*}
Then, by \eqref{hcontain}, we have the lower bound
\begin{align}
    \tilde{\mathbb{E}}^*_{(v,\bfz)} \textbf{1}_{\{V(r) \in I, \bfZ(r) \in J\}}e^{ -\int_0^rV(s)d\tilde{B}(s) }  \geq \tilde{\mathbb{E}}^*_{(v,\bfz)} \textbf{1}_He^{ -\int_0^rV(s)d\tilde{B}(s) } 
     { \geq } \tilde{\mathbb{E}}^*_{(v,\bfz)} \textbf{1}_He^{ U_1 + U_2 + U_3}. \label{eq:406}
\end{align}
Recall the killed process $\bfZ^*$ from \eqref{eq:zstar}. Define the sets
\begin{align*}
    H_1(s) & = \{\bfZ^*(s) \in J\}, \; 0\le s  \le 1, \\
    H_2(v) &= \left\lbrace L_1(\frac{\sn}{4}) \in gr+v-I, L_i(\frac{\sn}{4}) \leq \frac{\gn}{6} \mbox{ for } 2\le i \le N,\,\, \sup_{u \leq \frac{\sn}{4}} B_1(u) < \frac{6\gn}{10},\,\,\right. \\
    &\left. \,\,\,\,\,\,\,\,\,\,\,\,\,\sup_{u \leq \frac{\sn}{4}}|B_i(u)| < \frac{\gn}{8}\,\,\,  i = 2, ..., N,\,\, Z_1(\frac{\sn}{4}) \in [\frac{\gn}{16},\frac{\gn}{4}] \right\rbrace, \; \mbox{ where } v\in [0, g/128]. \\
    H_3 & = \left\lbrace \tau_1 \leq \frac{\sn}{4}, (Z_2(\tau_1), \ldots, Z_N(\tau_1)) \in [\gn,\frac{3\gn}{2}]^{N-1}\right\rbrace.
\end{align*}
Also, set
\begin{align*}
    U_1'(t) & \doteq  g\int_0^{t}\tilde{B}(s)ds -2g|\tilde{B}(t)|,\; 0 \le t \le 1,\\
    U_2' & \doteq   -3g|\tilde{B}(\frac{\sn}{4})|
   - \int_{0}^{\frac{\sn}{4}} \tilde{B}(s)dL_1(s) + g\int_{0}^{\frac{\sn}{4}}\tilde{B}(s)ds.
\end{align*}
Applying the Strong Markov Property at $\tau_1+\frac{\sn}{4}$ and then $\tau_1$, we have
\begin{multline}
    \tilde{\mathbb{E}}^*_{(v,\bfz)} \textbf{1}_He^{ U_1 + U_2 + U_3}  \geq \inf_{(\tilde v,\tilde \bfz) \in \RR\times F,\, \frac{\sn}{4} \leq s \leq r }\tilde{\mathbb{E}}^*_{(\tilde v,\tilde \bfz)} \textbf{1}_{H_1(s)}e^{U'_1(s)} \\
    \,\,\,\,\,\,\,\times \inf_{(\hat v,\hat{\bfz}) \in \RR\times [\gn,\frac{3\gn}{2}]^{N-1}}\tilde{\mathbb{E}}^*_{(\hat v,(0,\hat\bfz))} \textbf{1}_{\{L_1(\frac{\sn}{4}) \in gr+v-I, \bfZ(\frac{\sn}{4}) \in F\}}e^{U_2'} \;\;\times \tilde{\mathbb{E}}^*_{(v,\bfz)} \textbf{1}_{H_3}e^{U_3}. \label{eq409}
\end{multline}

Note that since by assumption $I \subseteq [0, g/128]$, $r \in [\sn,\en]$, and $v \in [0, g/128]$, 
\begin{equation}\label{eq:inclu}
	\tilde I \doteq gr+v-I \subseteq [0, g/63].
\end{equation}
Thus, using Lemma \ref{H2sub2},  we see
\begin{multline*}
    \tilde{\mathbb{E}}^*_{(v,\bfz)} \textbf{1}_He^{ U_1 + U_2 + U_3}  \geq 
\inf_{(\tilde v,\tilde \bfz) \in \RR\times F,\, \frac{\sn}{4} \leq s \leq r }\tilde{\mathbb{E}}^*_{(\tilde v,\tilde \bfz)} \textbf{1}_{H_1(s)}e^{U'_1(s)} \\
 \times \inf_{(\hat v,\hat{\bfz}) \in \RR\times [\gn,\frac{3\gn}{2}]^{N-1}}\tilde{\mathbb{E}}^*_{(\hat v,(0,\hat{\bfz}))} \textbf{1}_{H_2(v)}e^{U'_2} \times \tilde{\mathbb{E}}^*_{(v,\bfz)} \textbf{1}_{H_3}e^{U_3}.
\end{multline*}
For the final term, note that, on $H_3$,
{\cg $U_3 \ge -5g\sup_{0\le s \le \en} |\tilde B(s)|$.}
Now, for $M' > 0$, define

\editc{
\begin{align*}
    H_3'(M') &= \left\lbrace\sup_{0 \leq s \leq \en}|\tilde B(s)| < M',\,\,\, Z_1(s) > 0 \,\,\,\mbox{ for all } s \leq \frac{\sn}{8},\,\, \inf_{\frac{\sn}{8} \leq s \leq \frac{\sn}{4}}Z_1(s) = 0, \,\,\right. \\
     & \left. \,\,\,\,\,\,\,\,\,\, (Z_2(s), \ldots \editc{,} Z_N(s)) \in [\gn,\frac{3\gn}{2}]^{N-1} \,\,\,\mbox{ for all } s \in [\frac{\sn}{8}, \frac{\sn}{4}]\right\rbrace.
\end{align*}
}
{\cg
Clearly $H_3'(M') \subset H_3$.
For any $(v, \bfz) \in [0,\frac{g}{128}] \times (0,\infty) \times \RR_+^{N-1}$, one can construct suitable Brownian paths to obtain a measurable choice of $M' = M'(v,\bfz)$ such that
$$
\kappa_{(v,\bfz)} \doteq \tilde{\mathbb{P}}^*_{(v,\bfz)}(H_3'(M'(v,\bfz)))>0.
$$
This definition readily implies the measurability of $(v,\bfz) \mapsto \kappa_{(v,\bfz)}$ through the measurability of the maps $(v,\bfz) \mapsto M'(v,\bfz)$ and $(v,\bfz) \mapsto \tilde{\mathbb{P}}^*_{(v,\bfz)}(A^\circ)$ for any $A^\circ \in \mathcal{F}$.
}
{\cg Recall the set $\bar A$ from the statement of Theorem \ref{minorization}. Using continuity properties of the transition kernel of Brownian motion in its starting point, the choice of $M'(v,\bfz)$ can be made such that
\begin{equation}\label{eq:unifonc}
\sup_{(v,\bfz) \in \bar A} M'(v,\bfz) < \infty, \ \ \ \inf_{(v,\bfz) \in \bar A} \kappa_{(v,\bfz)} >0.
\end{equation}
It now follows that,
\begin{equation} \tilde{\mathbb{E}}^*_{(v,\bfz)} \textbf{1}_{H_3} e^{U_3} \ge 
\tilde{\mathbb{E}}^*_{(v,\bfz)} \textbf{1}_{H_3'(M'(v,\bfz))} e^{U_3} \ge 
e^{-5g M'(v,\bfz)}  \kappa_{(v,\bfz)}. \label{eq:411}\end{equation}}
 Now consider the term involving $H_1(s)$.
Note that, on the set $H_1(s)$, $A\bfB(u) +z \in G$ for all $u \le s$. Since $G$ is bounded, we have that for some $\kappa_G \in (0, \infty)$, under $\tilde{\mathbb{P}}^*_{(v,\bfz)}$, for all $(v,\bfz) \in \RR\times F$
$$\sup_{0\le u \le s} |\tilde B(u)| \le \kappa_G \mbox{ on } H_1(s), \mbox{ for all } s \in [\sn/4, r] \mbox{ and } r \in [\sn,\en].$$
Thus, from Lemma \ref{kild},
\begin{equation}
\inf_{(v,\bfz) \in \RR\times F,\, \frac{\sn}{4} \leq s \leq r }\tilde{\mathbb{E}}^*_{(v,\bfz)} \textbf{1}_{H_1(s)}e^{U'_1(s)}  \ge e^{-3g\kappa_G} \inf_{(v,\bfz) \in \RR\times F,\, \frac{\sn}{4} \leq s \leq r }\tilde{\mathbb{P}}^*_{(v,\bfz)} (\bfZ^*(s)\in J) \ge e^{-3g\kappa_G}  c_G \lambda(J\cap G_1).	
\label{eq:413}
\end{equation}

Consider finally the term involving $H_2(v)$. From Lemma \ref{H2sub1}  (and recalling \eqref{eq:inclu})  it follows that, on $H_2(v)$, for $ v \in [0, g/128]$ and $0\le s \le \sn/4$,
$$-B_1(s) \le \sup_{u\le \sn/4} (-B_1(u)) = L_1(\sn/4) \le g/63.$$
Using this and other properties of the set $H_2(v)$, we see that with $c_A  \doteq  \frac{6\gn}{10} \sum_{i=1}^N |(A^{-1})_{i1}|$, on $H_2(v)$,
\begin{equation}
	\sup_{0\le s \le \sn/4}|\tilde B(s)| \le c_A.
\end{equation}
It then follows that, on $H_2(v)$,
\begin{equation}
	U'_2 \ge -3g c_A - \frac{g \sn}{4} c_A - c_A L_1(\sn/4) \ge -4gc_A.
\end{equation}
Thus, we have
\begin{equation}\label{eq:424}
\inf_{(\hat v,\hat{\bfz}) \in \RR\times [\gn,\frac{3\gn}{2}]^{N-1}}\tilde{\mathbb{E}}^*_{(\hat v,(0,\hat{\bfz}))} \textbf{1}_{H_2(v)}e^{U'_2}  \ge e^{-4gc_A} \inf_{(\hat v,\hat{\bfz}) \in \RR\times [\gn,\frac{3\gn}{2}]^{N-1}}\tilde{\mathbb{P}}^*_{(\hat v,(0,\hat{\bfz}))}(H_2( v)).
\end{equation}
%
Note that the conditions
$\sup_{u\le \sn/4} (-B_2(u)) \le -13g/30 + \hat z_2$ and $\sup_{u \leq \frac{\sn}{4}} B_1(u) <  6\gn/10$ imply that $\hat L_2(\sn/4) \le g/6$.
Thus from  Lemma \ref{H2sub1}, and using \eqref{eq:inclu} again,
\begin{align}
   &\tilde{\mathbb{P}}^*_{(\hat v,(0,\hat{\bfz}))}(H_2(v))\nonumber\\
     &\quad= \tilde{\mathbb{P}}^*_{(\hat v,(0,\hat{\bfz}))}\left( \hat{L}_1(\frac{\sn}{4}) \in \tilde I, \hat{L}_i(\frac{\sn}{4}) \leq \frac{\gn}{6}, \sup_{u \leq \frac{\sn}{4}}|B_i(u)| < \frac{\gn}{8} \mbox{ for } 2\le i \le N,\,\, \sup_{u \leq \frac{\sn}{4}} B_1(u) < \frac{6\gn}{10} , \hat{Z}_1(\frac{\sn}{4}) \in [\frac{\gn}{16},\frac{\gn}{4}] \right) \nonumber\\
     &\quad\geq \tilde{\mathbb{P}}^*_{(\hat v,(0,\hat{\bfz}))}\left(\hat{L}_1(\frac{\sn}{4}) \in \tilde I,\, \sup_{u \leq \frac{\sn}{4}}(-B_2(u)) \leq -\frac{13\gn}{30}+\hat z_2, \hat{L}_3(\frac{\sn}{4}) \leq \frac{\gn}{6}, ..., \hat{L}_N(\frac{\sn}{4}) \leq \frac{\gn}{6}\right.,\,\, \,\, \\
 & \left.  \,\,\,\,\,\,\,\,\,\,\,\,\,\,\,\,\,\,\,\,\,\,\,\,\,\,\,\,\,\,\,\,\,\,\,\,\,\,\sup_{u \leq \frac{\sn}{4}} B_1(u) < \frac{6\gn}{10}, \, \sup_{u \leq \frac{\sn}{4}}|B_i(u)| < \frac{\gn}{8} \mbox{ for } 2\le i \le N,\,\,\hat{Z}_1(\frac{\sn}{4}) \in [\frac{\gn}{16},\frac{\gn}{4}] \right) \nonumber\\
     &\quad= K_{\hat{\bfz}}\,\tilde{\mathbb{P}}^*_{(\hat v,(0,\hat{\bfz}))}\left(\hat{L}_1(\frac{\sn}{4}) \in \tilde I,\, \sup_{u \leq \frac{\sn}{4}} B_1(u) < \frac{6\gn}{10},\,\hat{Z}_1(\frac{\sn}{4}) \in [\frac{\gn}{16},\frac{\gn}{4}]\right), \label{eq:416}
\end{align}
where 
\begin{align}
K_{\hat{\bfz}}=\tilde{\mathbb{P}}^*_{(\hat v,(0,\hat{\bfz}))}\left(\sup_{u \leq \frac{\sn}{4}}(-B_2(u)) \leq -\frac{13\gn}{30}+\hat z_2, \hat{L}_3(\frac{\sn}{4}) \leq \frac{\gn}{6}, ..., \hat{L}_N(\frac{\sn}{4}) \leq \frac{\gn}{6}, \, \sup_{u \leq \frac{\sn}{4}}|B_i(u)| < \frac{\gn}{8} \mbox{ for } 2\le i \le N\,\,\right),
\end{align}
and in the last step we have used the independence of $B_1$ and $(B_2, \ldots , B_N)$.

Note that $\hat{Z}_1(\frac{\sn}{4}) = B_1(\frac{\sn}{4}) + \hat{L}_1(\frac{\sn}{4})$, so if $\hat{L}_1(\frac{\sn}{4}) \in [0, \frac{\gn}{63}]$ and $\,B_1(\frac{\sn}{4}) \in [\frac{\gn}{16}, \frac{\gn}{8}]$, then $\hat{Z}_1(\frac{\sn}{4}) \in [\frac{\gn}{16},\frac{\gn}{4}]$. Consequently,
\begin{align*}
    &\tilde{\mathbb{P}}^*_{(\hat v,(0,\hat{\bfz}))}(\hat{L}_1(\frac{\sn}{4}) \in \tilde I,\, \sup_{u \leq \frac{\sn}{4}} B_1(u) < \frac{6\gn}{10},\,\hat{Z}_1(\frac{\sn}{4}) \in [\frac{\gn}{16},\frac{\gn}{4}]) \\
    & \geq \tilde{\mathbb{P}}^*_{(\hat v,(0,\hat{\bfz}))}(\sup_{u \leq \frac{\sn}{4}}(-B_1(u)) \in \tilde I,\, \sup_{u \leq \frac{\sn}{4}} B_1(u) < \frac{6\gn}{10},\,B_1(\frac{\sn}{4}) \in [\frac{\gn}{16},\frac{\gn}{8}])\\
&=\, \tilde{\mathbb{P}}^*_{(\hat v,(0,\hat{\bfz}))}(\sup_{u \leq \frac{\sn}{4}}(B_1(u)) \in \tilde I,\, \inf_{u \leq \frac{\sn}{4}} B_1(u) > -\frac{6\gn}{10},\,B_1(\frac{\sn}{4}) \in [-\frac{\gn}{8},-\frac{\gn}{16}]),
\end{align*}
where in the last line we have used the fact that
 $\{B(s)\}_{s \leq \frac{\sn}{4}}$ is equal in distribution to $\{-B(s)\}_{s\leq \frac{\sn}{4}}$.
Applying Lemma \ref{subdens} we have 
\begin{multline}
\tilde{\mathbb{P}}^*_{(\hat v,(0,\hat{\bfz}))}(\sup_{u \leq \frac{\sn}{4}}(B_1(u)) \in \tilde I,\, \inf_{u \leq \frac{\sn}{4}} B_1(u) > -\frac{6\gn}{10},\,B_1(\frac{\sn}{4}) \in [-\frac{\gn}{8},-\frac{\gn}{16}]) \ge \mathscr{K}\lambda(\tilde I \cap [0,\frac{\gn}{63}])\\
= \mathscr{K}\lambda(\tilde I ) = \mathscr{K}\lambda( I ) = \mathscr{K}\lambda( I \cap D_1), \label{eq:425}
\end{multline}
where for the last equality we have used that $I \subseteq [0, g/128]=D_1$.
Thus, letting
$$\hat K \doteq \inf_{\hat{\bfz} \in [\gn,\frac{3\gn}{2}]^{N-1}} K_{\hat{\bfz}},$$
we have on combining estimates in  \eqref{eq:mainest}, \eqref{eq:406}, \eqref{eq409}, \eqref{eq:411}, \eqref{eq:413}, \eqref{eq:424}, \eqref{eq:416}, \eqref{eq:425},
\begin{align*}
\mathbb{P}^r((v, \bfz), S)  &\ge e^{-2|A^{-1}\bfe_1|^2g^2} e^{-5gM'(v,\bfz)} \kappa_{(v,\bfz)} e^{-3g\kappa_G}  c_G e^{-4gc_A} \hat K \mathscr{K}\lambda(J\cap G_1)\lambda( I \cap D_1)\\
 &= K_{(v,\bfz)}\lambda((I\times J)\cap D)
 \end{align*}
{\cg where
 $$
K_{(v,\bfz)} = e^{-2|A^{-1}\bfe_1|^2g^2}e^{-5gM'(v,\bfz)} \kappa_{(v,\bfz)} e^{-3g\kappa_G}  c_G e^{-4gc_A} \hat K \mathscr{K}.$$
This proves the first statement in the theorem. The second statement is immediate from the measurability of $(v, \bfz) \mapsto \kappa_{(v,\bfz)}$ indicated earlier in the proof and \eqref{eq:unifonc}.}
\qedsymbol 

\section{Stationary Distribution: Uniqueness}\label{sec:exisuniq}
In this section, we establish uniqueness of the stationary distribution by using the minorization estimate in Theorem \ref{minorization} in conjunction with the following lemma. This lemma also plays a crucial role in establishing exponential ergodicity of the system.
\begin{lemma}\label{uniqlem}
For each $(v,\bfz) \in \mathbb{R} \times \mathbb{R}_+^{N}$, there exists $r_0 \doteq r_0(v,\bfz) \in \mathbb{N}$ such that
\begin{equation*}
    \mathbb{P}^{r_0}((v,\bfz), R) > 0,
\end{equation*}
where
\begin{equation}\label{Rset}
R \doteq (0,\frac{g}{128}) \times (0, \infty) \times \mathbb{R}_+^{N-1}.
\end{equation}
Furthermore, if $v \ge g/128$, we can take $r_0 = 1$.
\end{lemma}

\begin{proof}
Let $(v,\bfz) \in \mathbb{R} \times \mathbb{R}_+^{N}$ be given. 
In view of Corollary \ref{girsanov} it suffices to show that for some $r_0 \in \NN$
$$ \tilde{\mathbb{P}}^*_{(v,\bfz)}((V(r_0),\bfZ(r_0))\in R) >0.$$

Consider first the case where $v < g/128$. 
Define 
$$
v_1 \doteq
\begin{cases}
\frac{g}{256} - v, & v < \frac{g}{256},\\
\frac{1}{2}(\frac{g}{128}-v), & v \in [\frac{g}{256}, \frac{g}{128}).
\end{cases}
$$
Set $v_2 \doteq v+v_1$. 
Write $v_1 = g(k+t_1)$ with $k \in \NN_0$ and $t_1\in [0,1)$. Let $r_0 \doteq k+1$ and $t_2 \doteq (k+t_1)/2$.
Let $v_3 \doteq gr_0 - v_1$.
Fix $\delta \in (0, v_3)$ such that $[v_2-\delta, v_2+\delta] \subset (0, g/128)$.
Consider the set $A_1 \in \clf^*$ defined as
$$A_1 \doteq  \{L_1(t_2) \in [v_3-\delta, v_3+\delta], Z_1(t)>0 \mbox{ for all } t \in (t_2, r_0]\}.$$
Then on $A_1$, $\bfZ(r_0) \in (0, \infty)\times \RR_+^{N-1}$ and
\begin{align*}
	V(r_0) &= v+ (k+1)g - L_1(r_0) = v+ (k+1)g - L_1(t_2) \in v+ (k+1)g  - [v_3-\delta, v_3+\delta]\\
	&=  [v+ (k+1)g  -v_3-\delta, v+ (k+1)g  -v_3+\delta].
\end{align*}
Also
$$v+ (k+1)g  -v_3 = v+ (k+1)g  -(g(k+1)-v_1) = v+v_1 = v_2$$
Thus on $A_1$, $V(r_0) \in [v_2-\delta, v_2+ \delta] \subset (0, g/128)$ and consequently $A_1 \subset \{(V(r_0), \bfZ(r_0)) \in R\}$.
It is easily verified that $\tilde{\mathbb{P}}^*_{(v,\bfz)}(A_1) >0$ which proves the result for the case 
$v < g/128$.

Now consider the case $v\ge g/128$.
Let $v_1 \doteq v+g- g/256$ and fix $\delta \in (0, g/256)$.
Consider the set $A_2 \in \clf^*$ defined as
$$A_2 \doteq  \{L_1(1/2) \in [v_1-\delta, v_1+\delta], Z_1(t)>0 \mbox{ for all } t \in (1/2, 1]\}.$$
Then, with $r_0=1$, we see, on $A_2$, $\bfZ(r_0) \in (0, \infty)\times \RR_+^{N-1}$ and
\begin{align*}
	V(r_0) &= v+ g - L_1(r_0) = v+ g - L_1(1/2) \in v+ g  - [v_1-\delta, v_1+\delta]\\
	&=  [v+ g  -v_1-\delta, v+ g  -v_1+\delta] = [g/256-\delta, g/256+\delta] = [v_2-\delta, v_2+\delta]\subset (0, g/128).
\end{align*}
Thus  $A_2 \subset \{(V(r_0), \bfZ(r_0)) \in R\}$.
Once again, it is easily verified that $\tilde{\mathbb{P}}^*_{(v,\bfz)}(A_2) >0$ proving the result for the case 
$v \ge g/128$ with $r_0=1$.

\end{proof}
\begin{theorem}\label{atmostonesd}
The  Markov family $\{\PP_{(v,\bfz)}\}_{(v,\bfz) \in \RR \times \RR_+^N}$ has at most one stationary distribution.
\end{theorem}
\begin{proof}
By Birkhoff's ergodic theorem, if there are multiple  stationary distributions, then we can find two that are mutually singular \cite[Chapter 4, Theorem 4.4 and Lemma 4.6]{einsiedlerergodic}. Suppose that $\pi, \pi'$ are mutually singular stationary distributions. Then there is a  $A \in \clb(\mathbb{R} \times \mathbb{R}_+^N)$ such that $\pi(A) = \pi'(A^c) = 0$. Recall 
the set $D$ from Theorem \ref{minorization}.  Since $\lambda(D) > 0$, it follows that either $\lambda(D \cap A) > 0$ or $\lambda(D\cap A^c) > 0$. For specificity, suppose $\lambda(D \cap A) > 0$. We will now argue that $\pi(A)>0$, arriving at a contradiction.
 By Theorem \ref{minorization}, with $R$ as in Lemma \ref{uniqlem}, for every $(v, \bfz) \in R$, there is a $K_{(v,\bfz)}>0$ such that
\begin{equation}\label{uniqineq}
    \mathbb{P}^{\sn}((v,\bfz), A) \geq K_{(v,\bfz)}\lambda(A\cap D).
\end{equation}
Define the transition probability kernel  $\QQ$ on $\RR \times \RR_+^N$ as
\begin{equation*}
    \QQ((\tilde v,\tilde \bfz), S) \doteq \sum_{i=1}^\infty \frac{1}{2^i}\mathbb{P}^{i+\sn}((\tilde v,\tilde \bfz), S), \; (\tilde v,\tilde \bfz) \in \RR \times \RR_+^N, \; S \in \clb(\RR \times \RR_+^N).
\end{equation*}
Since $\pi$ is a stationary distribution, we have
\begin{equation}\label{eq:statstep}
	   \pi(A)  = \int_{\mathbb{R} \times \mathbb{R}_+^N} \QQ((\tilde v,\tilde \bfz), A) d\pi(\tilde v,\tilde \bfz).
\end{equation}
Also, for any $(\tilde v,\tilde \bfz) \in \RR \times \RR_+^N$ and with $r_0 = r_0(\tilde v,\tilde \bfz) \in \NN$ as in Lemma \ref{uniqlem},
\begin{align*}
	\QQ((\tilde v,\tilde \bfz), A) &\ge 2^{-r_0} \mathbb{P}^{r_0+\sn}((\tilde v,\tilde \bfz), A)
	\ge 2^{-r_0} \int_{R} \mathbb{P}^{r_0}((\tilde v,\tilde \bfz), (d v, d\bfz))
	\mathbb{P}^{\sn}(( v,\bfz), A)\\
	&\ge 2^{-r_0} \lambda(A\cap D)\int_{R} \mathbb{P}^{r_0}((\tilde v,\tilde \bfz), (d v, d\bfz))
	K_{(v,\bfz)} >0.
\end{align*}
From \eqref{eq:statstep} it now follows that $\pi(A)>0$ which gives a contradiction and proves the result.
\end{proof}

\section{Product form of stationary density}\label{sec:prodform}

In this section, we prove Theorem \ref{thm:prodform}. The proof relies on `guessing' a product form for the stationary joint density and showing that it satisfies the partial differential equations (along with appropriate boundary conditions) that characterize such stationary densities. This guess is inspired by \cite{banerjee2019gravitation}, where a product form joint density was obtained for the velocity and gap processes of the system comprising one inert particle and one Brownian particle.

\begin{proof}[Proof of Theorem \ref{thm:prodform}]
The generator of the process $(V,\mathbf{Z})$ given by \eqref{eq:gapproc} acts on any $f: \mathbb{R} \times \mathbb{R}_+^N \rightarrow \mathbb{R}$ that is continuously differentiable in $v$ and twice continuously differentiable in $(z_1,\dots, z_N)$, and compactly supported in the interior of $\mathbb{R} \times \mathbb{R}_+^N$, by
$$
\mathcal{L} f(v,\bfz) = \frac{1}{2}\sum_{1 \le i,j\le N} h_{ij}\frac{\partial f}{\partial z_i \partial z_j}(v,\bfz) +g \frac{\partial f}{\partial v}(v,\bfz) - v\frac{\partial f}{\partial z_1}(v,\bfz), \ (v,\bfz) \in \mathbb{R} \times (0,\infty)^N,
$$
where $h_{11} = 1$, $h_{ii} = 2$ for $2 \le i \le N$, $h_{ij} = -1$ for $|i-j| = 1$, and $h_{ij}=0$ otherwise.
Moreover, from the pathwise existence and uniqueness (Theorem \ref{thm:wellposed}) it readily follows that the associated submartingale problem \cite[Definition 2.1]{KangRam} for our process is well-posed. For $c_0,c_1,\ldots,c_N, { \phi} >0$, consider the function
\begin{equation}
\pi(v, \bfz) = c_{\pi}e^{-c_0(v + \phi)^2}\prod_{i=1}^Ne^{-c_iz_i}, \ (v,\bfz) \in \mathbb{R} \times \mathbb{R}_+^N, \label{eq:920}
\end{equation}
where $c_{\pi}$ is the normalization constant ensuring $\int_{\mathbb{R} \times \mathbb{R}_+^N} \pi(v,\bfz) dv d\bfz =1$. Translating the conditions (1)-(3) of \cite[Theorem 3]{KangRam}, $\pi$ is the density of a stationary distribution if $\pi$ satisfies the interior condition
\begin{align}
\mathcal{L}^* \pi(v,\bfz) \doteq \frac{1}{2}\sum_{1 \le i,j \le N} h_{ij}\frac{\partial \pi}{\partial z_i \partial z_j}(v,\bfz) -g \frac{\partial \pi}{\partial v}(v,\bfz) + \frac{\partial (v \pi)}{\partial z_1}(v,\bfz) = 0, \ (v,\bfz) \in \mathbb{R} \times (0,\infty)^N, \label{s1}
\end{align}
and boundary conditions
\begin{align}
 2v\pi(v,\bfz)+\frac{\partial \pi}{\partial z_1}(v,\bfz)-\frac{\partial \pi}{\partial z_2}(v,\bfz)+\frac{\partial \pi}{\partial v}(v,\bfz) &= 0 \ \mbox{ if } z_1=0, \label{s2}\\
 -\frac{\partial \pi}{\partial z_{i-1}}(v,\bfz) +2\frac{\partial \pi}{\partial z_i}(v,\bfz) - \frac{\partial \pi}{\partial z_{i+1}}(v,\bfz)& = 0 \ \mbox{ if } z_i=0, \mbox{ for some } 2 \le i \le N-1, \label{s3}\\
 -\frac{\partial \pi}{\partial z_{N-1}}(v,\bfz) + 2 \frac{\partial \pi}{\partial z_{N}}(v,\bfz)& = 0 \ \mbox{ if } z_N=0. \label{s4}
 \end{align}
 We will solve for the constants $c_0,c_1,\ldots,c_N, c_{\pi}$ to obtain a $\pi$ satisfying the above conditions.
 
The conditions \eqref{s3} and \eqref{s4} applied to \eqref{eq:920} yield that
\begin{align}
c_{i-1} - 2c_i + c_{i+1} &= 0, \,\,\,\,\,i = 2, ..., N-1,\label{s5}\\
   c_{N-1} -2c_N &= 0.
 \end{align}
From these identities, we obtain that
\begin{align*}
    c_{N-1} &= 2c_N \\
    c_{N-2} &= 2c_{N-1} - c_N = 3c_N.
\end{align*}
Fix $j \in \{2,...,N-1\}$. Suppose that we have $c_i = (N-i+1)c_N$ for all $ j \leq i \leq N$. Then, from \eqref{s5},
\begin{align*}
    c_{j-1} = 2c_j - c_{j+1} &= (2(N-j+1) - (N-j))c_N \\
    & = (N-(j-1)+1)c_N.
\end{align*}
Hence, we have by induction that $c_i = (N-i+1)c_N$ for $i = 1,...,N$. Substituting this into \eqref{s2}, we  see that
$$2v-c_1+c_2-2c_0(v+\phi)=0, \,\,\, \mbox{ for all } \ v \in \mathbb{R}.$$
Since this holds for all $v \in \mathbb{R}$, we must have $c_0 = 1$, and so
\begin{align*}
    2\phi = c_2 - c_1 = (N-1)c_N - Nc_N = -c_N 
\end{align*}
and thus $c_N = -2\phi$. Next substituting \eqref{eq:920} in \eqref{s1},
\begin{equation}\label{icond}
\frac{1}{2}\sum_{1 \leq i,j \leq N}h_{ij}c_ic_j + 2gc_0(v+\phi)-c_1v = 0, \,\,\,\mbox{ for all } \ v \in \mathbb{R}.
\end{equation}
Again, since this holds for all $v \in \mathbb{R}$, we must have,
$$2g = c_1 = Nc_N.$$
From the above relations, we obtain
\begin{align}
	\label{statdistconst}
    c_0 = 1 ,\;\;
    c_i = 2\left(\frac{N-i+1}{N}\right)g,\,\,\, \  i = 1,...,N,\;\;
    \phi = -\frac{g}{N}.
\end{align}
To show that this choice of constants yields a valid density for some stationary distribution, it remains only to demonstrate that \eqref{s1} holds for all $(v,\bfz) \in \mathbb{R} \times (0,\infty)^N$, or equivalently, from \eqref{icond} and \eqref{statdistconst},
\begin{align}\label{fi1}
    \frac{1}{2}\sum_{1 \leq i,j \leq N}h_{ij}c_ic_j - 2\frac{g^2}{N} = 0.
\end{align}
To see this, note that
\begin{align*}
    \frac{1}{2}\sum_{1 \leq i,j \leq N}h_{ij}c_ic_j & = \frac{1}{2}\sum_{1 \leq i,j \leq N}h_{ij}\left(\frac{N-i+1}{N}\right)\left(\frac{N-j+1}{N}\right)4g^2 \\
    & = \frac{2g^2}{N^2}\sum_{1 \leq i,j \leq N}h_{ij}(N-i+1)(N-j+1) \\
    & = \frac{2g^2}{N^2}\sum_{i=1}^N(N-i+1)\sum_{j=1}^N h_{ij}(N-j+1).
\end{align*}
From the formulae of $\{h_{ij}\}_{1 \le i,j \le N}$, it follows that
$$
\sum_{j=1}^N h_{ij}(N-j+1) = \delta_{1,i},\,\,\,\mbox{ for all } \ i = 1,...,N,
$$
where $\delta_{1,i}$ is the Kronecker delta function. Hence,
\begin{align*}
  \frac{1}{2}\sum_{1 \leq i,j \leq N}h_{ij}c_ic_j
   &= \frac{2g^2}{N^2}\sum_{i=1}^N(N-i+1)\delta_{1,i}
    = \frac{2g^2}{N},
\end{align*}
which proves \eqref{fi1}.
We have therefore shown that $\pi$ with constants as in \eqref{statdistconst} is indeed the density for a stationary distribution of the process $(V,\mathbf{Z})$. Uniqueness follows from Theorem \ref{thm:exisuniq}.
\end{proof}

\section{Exponential Ergodicity}\label{sec:geomerg}

In this section we will prove Theorem \ref{thm:geomerg}. Since the main source of stability
in our system is the local time interactions between particles, standard PDE techniques for constructing
Lyapunov functions (\cite{khas},\cite{mattingly2002ergodicity},\cite{eberle2019couplings}) for hypoelliptic diffusions are hard to implement.
Furthermore, the singular nature of the dynamics arising from the motion of the inert particle, and the spatial dependence of the drift (which contains a $V$ term), make it challenging to adapt the Lyapunov function constructions  for reflected Brownian motions, which proceed via an analysis of the associated noiseless system \cite{dupuis1994lyapunov},\cite{atar2001positive}.

\editc{\subsection{Outline of Approach}}
Here, we take a different approach to exponential ergodicity by analyzing excursions of the process between appropriately chosen stopping times (see \eqref{eq:sigmatimes}) as the velocity of the inert particle `toggles' between two levels. Control on the exponential moments of these stopping times is established in Sections \ref{highlev} and \ref{sec:lowlev}. In Section \ref{singdrift}, it is shown that the intersection local time between the bottom two particles creates a `singular' drift that results in a reduction of the function $\bar{Z}_2(t) \doteq \sum_{i=1}^{N-1}iZ_{N-i+1}(t)$ of the gaps when observed at successive stopping times. These estimates are combined in Section \ref{comphit} to show that the distribution of return times of the process to an appropriately chosen compact set $C^*$ has exponentially decaying tails. Finally, in Section \ref{ee}, the exponential moments of this return time are used to construct a suitable Lyapunov function. The minorization estimate in Theorem \ref{minorization} is utilized to show that $C^*$ is a `petite' (or small) set in the language of \cite{DowMeyTwe}. These facts together imply exponential ergodicity using the machinery developed in \cite{DowMeyTwe} (see Theorem 6.2 there). \editc{Proofs of some
technical lemmas are deferred to Section \ref{sec:techlem} in order to make it easier to see the overall idea.}

\editc{We note here that the connection between finiteness of exponential moments of certain hitting times, Lyapunov functions and exponential ergodicity is not new \cite{meyn2012markov,DowMeyTwe}. The main work in this section is in establishing that exponential moments of associated hitting times are finite through a detailed pathwise analysis of the process. A general treatment of the above connection in the context of diffusion processes has been undertaken in \cite{cattiaux2013poincare,monmarche20192}, among others. However, typically the diffusion processes are assumed to be hypoelliptic and/or reversible with respect to the stationary measure, neither of which  apply to our setting.}

\subsection{An Inequality for the Local Time}\label{loctimeineq}
In this section we 
establish an estimate on local times which will be used several times. Recall the matrix $W$ from Section \ref{sec:mainres} and the process $\bfB^*$ from
\eqref{bstar}.
\begin{lemma}\label{locin}
For any $(v,\bfz) \in \mathbb{R} \times \mathbb{R}_+^{N}$ and $t \geq 0$, the following inequality holds, $\mathbb{P}^*_{(v,\bfz)}-\textnormal{a.s.}$, for all $i = 1, 2, ..., N$, 
\begin{equation}\label{LTineq}
    L_i(t) \leq W_{i,1}t\sup_{0 \leq s \leq t}(V(s))^+ + \sum_{j=1}^{N}W_{i,j}B^*_j(t).
\end{equation}
Moreover, with $\bar Y(t) \doteq \sum_{i=2}^N (N-i+1) B_i^*(t) $,
\begin{equation}\label{eql2l1}
	L_2(t) \le \frac{2(N-1)}{N}L_1(t) + \frac{2}{N} \bar Y(t), \; t \ge 0.
\end{equation}

\end{lemma}

\begin{proof}
Let $(v,\bfz) \in \mathbb{R} \times \mathbb{R}_+^{N}$ be given. 
All inequalities in the proof are a.s. under $\mathbb{P}^*_{(v,\bfz)}$.
For $1\le i \le N$, the local times $L_i$ are given as
\begin{equation}\label{eq:loctimes}
\begin{aligned}
    L_1(t) &= \sup_{s \leq t}(-z_1+\frac{1}{2}L_2(s) + \int_0^sV(u)du - B_1(s))^+ \\
    L_2(t) &= \sup_{s \leq t}(-z_2 + \frac{1}{2}L_3(s) + L_1(s) + B_1(s) - B_2(s))^+ \\
    L_i(t) &= \sup_{s \leq t}(-z_i + \frac{1}{2}(L_{i-1}(s)+L_{i+1}(s))+B_{i-1}(s)-B_{i}(s))^+, \;\; i = 3, \ldots , N. 
\end{aligned}
\end{equation}
Using these identities we see that
\begin{equation}\label{eq:lbds}
\begin{aligned}
    L_1(t) 
    &\leq \sup_{0 \leq s \leq t}(V(s))^{+}t + \frac{1}{2}L_2(t) + \sup_{s \leq t}(-B_1(s)) \\
    L_2(t) 
    &\leq \frac{1}{2}L_3(t) + L_1(t) + \sup_{s \leq t}(B_1(s) - B_2(s)) \\
    L_i(t) 
    &\leq \frac{1}{2}(L_{i+1}(t) + L_{i-1}(t)) + \sup_{s \leq t}(B_{i-1}(s) - B_i(s)), \;\;   i = 3, \ldots , N. 
\end{aligned}
\end{equation}
Recalling the matrix $U$ from \eqref{Umat} the above inequalities can be written as
$$\bfL(t) \leq \sup_{0 \leq s \leq t}(V(s))^{+}t\bfe_1 + U\bfL(t) + \bfB^*(t),\,\,\,\,\,\, t \geq 0.$$
Using the fact that $W=(I-U)^{-1}$ is a matrix with nonnegative entries, we have,
$$\bfL(t) \leq \sup_{0 \leq s \leq t}(V(s))^{+}tW\bfe_1 + W\bfB^*(t),\,\,\,\,\,\, t \geq 0.$$
This proves the first statement in the lemma. For the second inequality, note that from \eqref{eq:lbds} we have
$$\sum_{i=3}^N (N-i+1) (L_i(t) - \frac{1}{2} L_{i+1}(t) - \frac{1}{2} L_{i-1}(t)) + (N-1) (L_2(t) - L_1(t) - \frac{1}{2} L_3(t)) \le \sum_{i=2}^N (N-i+1) B_i^*(t) = \bar Y(t).$$
Simplifying the left side, we see,
$$\frac{N}{2} L_2(t) - (N-1)L_1(t) \le \bar Y(t)$$
which proves the second statement.
\end{proof}

\subsection{Hitting Time of an Upper Velocity Level}\label{highlev}
For $c \in \RR$, let $\hat\tau_c \doteq \inf\{t \geq 0 : V(t) = c\}$.
 The main result of this section is the following control on exponential moments of this hitting time.

\begin{proposition}\label{sigma1} There exists a $\lala \in (0,\infty)$ such that
$$\sup_{ (v,\bfz) \in [\frac{g}{2N},2g]\times \RR_+^N}\,\mathbb{E}^*_{(v,\bfz)}\,e^{\lala\,\hat\tau_{4g}} < \infty.$$
\end{proposition}
Proof of the proposition relies on the three lemmas given below. 
\editc{Proofs of these lemmas are given in Section \ref{props1}.
The proposition is proved after the statements of these lemmas.}

\begin{lemma}\label{sec1prop1}   There exists a $\beta \in (0,\infty)$ so that 
$$\sup_{{\bfz} \in \RR_+^{N}}\,\mathbb{E}^*_{(0,{\bfz})}\,e^{\beta\,\hat\tau_{g/(2N)}} < \infty.$$
\end{lemma}

\begin{lemma}\label{sec1prop2} We have
\begin{equation}
\inf_{(v, \bfz) \in [\frac{g}{2N},2g]\times \RR_+^N}\mathbb{P}^*_{(v,\bfz)}(\hat \tau_{4g} < \hat\tau_{0}) \doteq p >0.
\end{equation}	
\end{lemma}

\begin{lemma}\label{sec1prop3} There exists $\lala_1 > 0$ so that 
$$\sup_{(v,\bfz) \in [0,4g] \times \RR_+^N}\,\mathbb{E}^*_{(v,\bfz)}\,e^{\lala_1\,(\hat\tau_{4g}\wedge\hat\tau_{0})} < \infty.$$
\end{lemma}

 We now prove  the main result of the section.\\

\noindent \textbf{Proof of Proposition \ref{sigma1}.}
Define $\tau_{-1} = \tau_0 \doteq  0$ and for $i \in \NN_0$, define
\begin{align*}
    \tau_{2i+1}  \doteq \inf\{t \geq \tau_{2i}: V(t) = 4g\,\,\, or\,\,\, 0\} ,\;\;
    \tau_{2i+2}  \doteq \inf\{t \geq \tau_{2i+1}: V(t) = \frac{g}{2N}\,\,\, or\,\,\, 4g\}. 
\end{align*}
Define
$$\mathscr{N} = \inf\{k \geq 0: V(\tau_{2k+1})=4g\}.$$
From Lemma \ref{sec1prop2} it follows that
\begin{equation}
\sup_{(v,\bfz) \in [\frac{g}{2N}, 2g]\times \RR_+^N}\mathbb{P}^*_{(v,\bfz)}(\mathscr{N} = k) \le (1-p)^{k-1}, \ \ k \ge 0.	
\end{equation}
Note that
\begin{align}
    \hat\tau_{4g} &\leq \sum_{i=0}^{\mathscr{N}}(\tau_{2i+1}-\tau_{2i-1})
     \leq \sum_{i=1}^{\mathscr{N}+1}(\tau_{2i}-\tau_{2i-2}). \label{eq:535}
\end{align}
By Lemmas \ref{sec1prop1} and \ref{sec1prop3} there are $c_1, c_2 \in (0,\infty)$ such that
$$\sup_{(v,\bfz) \in [\frac{g}{2N}, 2g]\times \RR_+^N}\mathbb{P}^*_{(v,\bfz)}(\tau_2 \geq t) \leq c_1e^{-c_2t}.$$
It then follows that, for $0 < \alpha < c_2$,
\begin{align*}
  \sup_{(v,\bfz) \in [\frac{g}{2N}, 2g]\times \RR_+^N}\mathbb{E}^*_{(v,\bfz)}\,e^{\alpha\,\tau_2} & \le \int_{-\infty}^{\infty}\alpha e^{\alpha s}\sup_{(v,\bfz) \in [\frac{g}{2N}, 2g]\times \RR_+^N}\mathbb{P}^*_{(v,\bfz)}(\tau_2 \geq s)ds \\
     &\leq 1 + \alpha c_1\int_0^{\infty}e^{(\alpha - c_2)s}ds 
     = 1 + \frac{\alpha c_1}{c_2-\alpha}.
\end{align*}
Choose $\delta \in (0,1)$ such that $(1+\delta)(1-p)\doteq \kappa <1$. 
Now choose $\alpha>0$ sufficiently small such that
$$\sup_{(v,\bfz) \in [\frac{g}{2N}, 2g]\times \RR_+^N}\mathbb{E}^*_{(v,\bfz)}\,e^{2\alpha\,\tau_2} \le (1+\delta).$$
%
%
%
Applying Cauchy-Schwarz and the Strong Markov property we now see that,
for any $(v,\bfz) \in [\frac{g}{2N}, 2g]\times \RR_+^N$,
\begin{align*}
    \mathbb{E}^*_{(v,\bfz)}e^{\alpha\sum_{i=1}^{\mathscr{N}+1}(\tau_{2i}-\tau_{2i-2})} & = \sum_{k=0}^{\infty}\mathbb{E}^*_{(v,\bfz)}e^{\alpha\sum_{i=1}^{k+1}(\tau_{2i}-\tau_{2i-2})}\textbf{1}_{\{\mathscr{N} = k\}} \\
    & \leq \sum_{k=0}^{\infty}(\mathbb{E}^*_{(v,\bfz)}e^{2\alpha\sum_{i=1}^{k+1}(\tau_{2i}-\tau_{2i-2})})^{\frac{1}{2}}(\mathbb{P}^*_{(v,\bfz)}(\mathscr{N} = k))^{\frac{1}{2}} \\
    & \leq \frac{2}{(1-p)^{1/2}}\sum_{k=0}^{\infty}(1+\delta)^{k/2}(1-p)^{\frac{k}{2}} 
     \leq \sum_{k=0}^{\infty}\kappa^{k/2}  < \infty.
\end{align*}
The result now follows on combining the above estimate with \eqref{eq:535}.
\qedsymbol

\subsection{Hitting Time of a Lower Velocity Level}
\label{sec:lowlev}
Let $\sigma_1 \doteq \hat \tau_{4g} = \inf\{t \geq 0: V(t) = 4g\}$ and set
 $\sigma_2 \doteq \inf\{t \geq \sigma_1: V(t) = 2g\}$. 
 The main result of the section is the following.
\begin{proposition}\label{sigma2} There is $\lala_2 > 0$ such that
$$\sup_{\hat{\bfz} \in  \RR_+^{N-1}}\,\,\mathbb{E}^*_{(2g,0,\hat{\bfz})}\,e^{\lala_2\sigma_2} < \infty.$$
\end{proposition}
The proof relies on the following three lemmas. \editc{Proofs of these lemmas are given in Section \ref{sec:pfsigma2}.
The proposition is proved after the statements of the lemmas.}
\begin{lemma}\label{sec2prop1} There is a $\lala_3 > 0$ such that 
$$\sup_{\hat{\bfz} \in  \RR_+^{N-1}}\mathbb{E}^*_{(2g,0,\hat{\bfz})}e^{\lala_3\, Z_1(\sigma_1)^{1/2}} < \infty.$$
\end{lemma}

\begin{lemma}\label{sec2prop2} Define $\tau_0^{Z_1} \doteq \inf\{t \geq 0: Z_1(t) = 0\}$. There is a $\lala_4 > 0$ and 
	$\kappa_1, \kappa_2 \in (0,\infty)$ such that for any $z_1 \in \RR_+$ and $\lala \in (0, \lala_4]$
		%
\begin{equation*}
    \sup_{\hat{\bfz} \in \mathbb{R}_+^{N-1}}\mathbb{E}^*_{(4g,z_1,\hat{\bfz})}e^{\lala \tau_0^{Z_1}} \leq \kappa_1 e^{\kappa_2\lala z_1^{1/2}}.
\end{equation*}
\end{lemma}

\begin{lemma}\label{sec2prop3}  There exists a $\lala_5 > 0$ and $\kappa_1', \kappa_2' \in (0, \infty)$ such that for all 
	$\lala \in (0, \lala_5)$ and 
	${ v \in [2g, \infty)}$, 
$$\sup_{\hat{\bfz} \in \RR_+^{N-1}}\mathbb{E}^*_{(v, 0,\hat{\bfz})}e^{\lala \,\hat \tau_{2g}} \leq \kappa_1' e^{\kappa_2' \lala v}.$$
\end{lemma}

We now prove the main result of this section.
$\,$ \\

\noindent \textbf{Proof of Proposition \ref{sigma2}.}
Let $\alpha \in (0,1)$ be such that
\begin{equation}\label{eq:alphcon}
	\alpha < \gamma_5, \;\; \alpha (1+\kappa_2'g) \le \gamma_4, \;\; 2\alpha (1+\kappa_2'g)\kappa_2  \le \gamma_3, \;\; 2\alpha(1+\kappa_2'g)  \le \gamma,
\end{equation}
where $\gamma_5$  and $\kappa_2'$ are as in Lemma \ref{sec2prop3}, $\kappa_2$ and $\gamma_4$ are as in Lemma \ref{sec2prop2},
$\gamma_3$ is as in Lemma \ref{sec2prop1} and $\lala$ is as in Proposition \ref{sigma1}.

Fix $\hat{\bfz} \in \RR_+^{N-1}$. Define stopping time
\begin{align}\label{eq:etastop}
	\eta_1 &\doteq \inf\{t \ge \sigma_1: Z_1(t)=0\}.
\end{align}
Note that $\sigma_2 = \inf\{t \ge \eta_1:  V(t) = 2g\}$.
From the strong Markov property, Lemma \ref{sec2prop3},  and recalling the first condition on $\alpha$ from \eqref{eq:alphcon}, 
\begin{align*}
\mathbb{E}^*_{(2g,0,\hat{\bfz})}\,e^{\alpha\sigma_2} 
 &= 
\mathbb{E}^*_{(2g,0,\hat{\bfz})} \left[ \mathbb{E}^*_{(2g,0,\hat{\bfz})}(e^{\alpha\sigma_2} \mid \clf^*_{\eta_1})\right] \le \kappa_1'\mathbb{E}^*_{(2g,0,\hat{\bfz})}e^{\kappa_2'\alpha V(\eta_1) + \alpha \eta_1}\\
& \le \kappa_1'\mathbb{E}^*_{(2g,0,\hat{\bfz})}e^{\kappa_2'\alpha (4g+ g\eta_1) + \alpha \eta_1}
\le \kappa_1'e^{4\kappa_2'\alpha g}\mathbb{E}^*_{(2g,0,\hat{\bfz})}e^{\alpha(1+\kappa_2'g) \eta_1}.
\end{align*}
Thus, with $d_1 = \kappa_1'e^{4\kappa_2' \alpha g}$ and $d_2= (1+\kappa_2'g)$,
\begin{equation}
\mathbb{E}^*_{(2g,0,\hat{\bfz})}\,e^{\alpha\sigma_2} \le  d_1 \mathbb{E}^*_{(2g,0,\hat{\bfz})}e^{\alpha d_2\eta_1}. \label{eq:541}
\end{equation}
Using the strong Markov property again,
\begin{align*}
\mathbb{E}^*_{(2g,0,\hat{\bfz})}e^{\alpha d_2\eta_1} & =
\mathbb{E}^*_{(2g,0,\hat{\bfz})} \left[ \mathbb{E}^*_{(2g,0,\hat{\bfz})}(e^{\alpha d_2\eta_1} \mid \clf^*_{\sigma_1})\right]\\
&\le \kappa_1\mathbb{E}^*_{(2g,0,\hat{\bfz})} e^{\kappa_2 \alpha  d_2Z_1(\sigma_1)^{1/2} + \alpha d_2\sigma_1}\\
&\le \kappa_1\left(\mathbb{E}^*_{(2g,0,\hat{\bfz})} e^{2\kappa_2 \alpha  d_2Z_1(\sigma_1)^{1/2} }\right)^{1/2}
\left(\mathbb{E}^*_{(2g,0,\hat{\bfz})} e^{2 \alpha d_2\sigma_1}\right)^{1/2},
\end{align*}
where the second inequality is from Lemma \ref{sec2prop2} and on  recalling the second condition on $\alpha$ from \eqref{eq:alphcon}, and the last line is from Cauchy-Schwarz inequality.
Next, applying Lemma \ref{sec2prop1}, and recalling the third condition on $\alpha$ from \eqref{eq:alphcon}, 
$$\sup_{\hat{\bfz} \in \RR_+^{N-1}}\mathbb{E}^*_{(2g,0,\hat{\bfz})} e^{2\kappa_2 \alpha d_2  Z_1(\sigma_1)^{1/2}}
\doteq d_3 <\infty.$$
Finally, applying Proposition \ref{sigma1} and recalling the fourth condition on $\alpha$ from \eqref{eq:alphcon} 
we have
$$\sup_{\hat{\bfz} \in \RR_+^{N-1}}\mathbb{E}^*_{(2g,0,\hat{\bfz})}e^{2 \alpha d_2\sigma_1} \doteq d_4<\infty.$$
Combining the above estimates we have
$$
\sup_{\hat{\bfz} \in \RR_+^{N-1}}\mathbb{E}^*_{(2g,0,\hat{\bfz})}\,e^{\alpha\sigma_2} \le d_1 \kappa_1 d_3^{1/2} d_4^{1/2} <\infty.$$
The result follows.
\qedsymbol

\subsection{A Negative Singular Drift Property}\label{singdrift}

For $\bfz \in \RR_+^N$,  define
$\bar{z}_2 \doteq \sum_{i=2}^N(N-i+1)z_i$. Similarly, for $\RR_+^N$ valued process $\{\bfZ(t)\}$, we define
 for $t\ge 0$,
\begin{equation}\label{barz2}
\bar{Z}_2(t) \doteq \sum_{i=2}^N(N-i+1)Z_i(t) = \sum_{i=1}^{N-1}iZ_{N-i+1}(t).
\end{equation}

The main result of this section is Propositon \ref{neginexp}, where we will show that if $\bar{z}_2$ is large, then the process $\bar{Z}_2(\cdot)$ decreases in expectation in the course of an appropriately large number of excursions of the velocity process between the levels $2g$ and $4g$ (see \eqref{eq:sigmatimes}).  

The following lemma gives a key algebraic representation of $\bar{Z}_2(\cdot)$ in terms of $L_1$, $L_{k+1}$ for $k \in \{1,\ldots,N-1\}$, and additional error terms. If $\bar{z}_2$ is large, then there exists $k \in \{1,\ldots,N-1\}$ such that $Z_{k+
1}(0) = z_{k+1}$ is large. Thus, it takes a long time for this gap to hit zero. Before this time, the lowest (inert) particle `pushes' the bottom $k+1$ particles up and thereby reduces $\bar{Z}_2(\cdot)$, as captured by the $L_1$ term in the lemma. This `singular' drift through local times results in stability and, in turn, exponential ergodicity, of the system.
\begin{lemma}\label{zless}
Let $Y^{(1)}_{1}(t) = 0$ and $Y^{(1)}_{k}(t) \doteq \sum_{i=2}^{k}(k-i+1) B_i^*(t), \, t \ge 0, \, 2 \le k \le N$. Also define $M(t) \doteq \sum_{i=2}^NB_i(t) - (N-1)B_1(t), \, t \ge 0$. Then
for all $(v,\bfz) \in \RR\times \RR_+^N$,  $\PP^*_{(v,\bfz)}$ a.s., 
\begin{equation}\label{zlesseq}
 \bar{Z}_2(t) - \bar{z}_2  { \,\,\le\,\, }  M(t) + \frac{N}{k}Y^{(1)}_k(t) + \frac{N}{2k}L_{k+1}(t) - \frac{(N-k)}{k}L_1(t), \ t \ge 0, \ 1 \le k \le N,
\end{equation}
{ with equality for $k = 1$.}
\end{lemma}

\begin{proof}
Note that, for $(v,\bfz) \in \RR\times \RR_+^N$, under $\PP^*_{(v,\bfz)}$, for $t\ge 0$,
 \begin{align*}
    \bar{Z}_2(t) & = -\frac{(N-1)}{2}L_1(t) + \sum_{i=2}^N(N-i+1)\left[(B_i(t)-B_{i-1}(t)-\frac{1}{2}(L_{i+1}(t)+L_{i-1}(t))+L_i(t) + z_i\right] \\
    & = -\frac{(N-1)}{2}L_1(t) + \sum_{i=1}^{N-1}i\left[(B_{N-i+1}(t)-B_{N-i}(t)-\frac{1}{2}(L_{N-i+2}(t)+L_{N-i}(t))+L_{N-i+1}(t) + z_{N-i+1}\right] \\
    & = \bar z_2 - \frac{(N-1)}{2}L_1(t) + \sum_{i=1}^{N-1}iB_{N-i+1}(t) - \sum_{i=1}^{N-1}iB_{N-i}(t) +\sum_{i=1}^{N-1}i(L_{N-i+1}(t)-\frac{1}{2}(L_{N-i+2}(t)+L_{N-i}(t))). \\
\end{align*}
Also,
\begin{align*}
    \sum_{i=1}^{N-1}iB_{N-i+1}(t) - \sum_{i=1}^{N-1}iB_{N-i}(t) & = \sum_{i=0}^{N-2}(i+1)B_{N-i}(t) - \sum_{i=1}^{N-1}iB_{N-i}(t) 
     = \sum_{i=2}^NB_i(t) - (N-1)B_1(t).
\end{align*}
Moreover,
\begin{multline*}
    \sum_{i=1}^{N-1}i(L_{N-i+1}(t)-\frac{1}{2}(L_{N-i+2}(t)+L_{N-i}(t)))\\ 
	 = -\frac{1}{2}\sum_{i=1}^{N-1}i(L_{N-i+2}(t)-L_{N-i+1}(t)) 
    +\frac{1}{2}\sum_{i=1}^{N-1}i(L_{N-i+1}(t)-L_{N-i}(t))) \\
     = -\frac{1}{2}\sum_{i=0}^{N-2}(i+1)(L_{N-i+1}(t)-L_{N-i}(t))
   +\frac{1}{2}\sum_{i=1}^{N-1}i(L_{N-i+1}(t)-L_{N-i}(t)) \\
   = \frac{(N-1)}{2}(L_2(t)-L_1(t))+\frac{1}{2}L_2(t).
\end{multline*}
Hence, 
\begin{align*}
    \bar{Z}_2(t) & = \bar{z}_2 - \frac{(N-1)}{2}L_1(t) + \sum_{i=2}^NB_i(t) - (N-1)B_1 (t)+ \frac{(N-1)}{2}(L_2(t)-L_1(t))+\frac{1}{2}L_2(t) \\
    & = \bar{z}_2 + \sum_{i=2}^NB_i(t) - (N-1)B_1(t) + \frac{N}{2}L_2(t) - (N-1)L_1(t).
\end{align*}
Consider the martingale
$M(t) \doteq \sum_{i=2}^NB_i(t) - (N-1)B_1(t)$. Then
\begin{equation}\label{eq:semmart}
	\bar{Z}_2(t)  =  \bar{z}_2 + M(t) + \frac{N}{2}L_2(t) - (N-1)L_1(t).
	\end{equation}
This proves \eqref{zlesseq} for $k=1$. Now, we will use this along with some local time inequalities to prove \eqref{zlesseq} for $k \ge 2$. Note that, from \eqref{eq:lbds},
\begin{align}
    L_2(t) &\leq L_1(t) + B^*_2(t) +\frac{1}{2}L_3(t) \label{eq:eq147} \\
    L_i & \leq B^*_i(t) + \frac{1}{2}(L_{i+1}(t) + L_{i-1}(t)),\,\, i = 3, \ldots , N.\nonumber
\end{align}
From these identities it follows that, for $k \in \{3, \ldots, N\},$
\begin{align*}
    \sum_{i=3}^k (k-i+1)\left[(L_i(t) - \frac{1}{2}(L_{i+1}(t) + L_{i-1}(t)))\right] + &(k-1)(L_2(t)-L_1(t)-\frac{1}{2}L_3(t)) \\
    &\leq \sum_{i = 2}^k (k-i+1)B^*_i(t) \doteq  Y^{(1)}_k(t).
\end{align*}
On the other hand,
\begin{align*}
    \sum_{i=3}^k (k-i+1)(L_i(t) - \frac{1}{2}(L_{i+1}(t) + L_{i-1}(t)))& + (k-1)(L_2(t)-L_1(t)-\frac{1}{2}L_3(t)) \\
    &= \frac{k}{2}L_2(t) - (k-1)L_1(t) -\frac{1}{2}L_{k+1}(t).
\end{align*}
Combining the last two displays and multiplying through by $\frac{2}{k}$,
\begin{align}\label{eq:1055}
    L_2(t) \leq \frac{2(k-1)}{k}L_1(t) + \frac{1}{k}L_{k+1}(t) + \frac{2}{k}Y^{(1)}_k(t), \, k = 3, \ldots , N.
\end{align}
The last display holds trivially for $k=1$  and also for $k=2$,  as can be seen from \eqref{eq:eq147}.

Hence, for all $1\le k \le N$, using \eqref{eq:semmart},
\begin{align}
    \bar{Z}_2(t) - \bar{z}_2 & = M(t) + \frac{N}{2}L_2(t) - (N-1)L_1(t) \nonumber\\
    & \leq M(t) + \frac{N}{2}(\frac{2}{k}Y^{(1)}_k(t) + \frac{2(k-1)}{k}L_1(t) + \frac{1}{k}L_{k+1}(t)) - (N-1)L_1(t) \nonumber\\
    & =  M(t) + \frac{N}{k}Y^{(1)}_k(t) + \frac{N}{2k}L_{k+1}(t) - \frac{(N-k)}{k}L_1(t).\label{eq:507}
\end{align}
This proves the lemma.	
\end{proof}

Define the sequence of stopping times $\{\sigma_m\}_{m\ge 0}$ as $\sigma_0=0$, and for $i \ge 0$,
\begin{equation}
	\label{eq:sigmatimes}
	\sigma_{2i+1} \doteq \inf\{t \geq \sigma_{2i}: V(t) = 4g\},\;\;
 \sigma_{2i+2} \doteq \inf\{t \geq \sigma_{2i+1}: V(t) = 2g\}.
 \end{equation}
For $\hat\bfz \in \RR_+^{N-1}$, abusing notation, write $\sum_{i=2}^N(N-i+1)\hat z_i$ as $\bar z_2$.
\begin{proposition}\label{neginexp}
There exists $\Delta_0 > 0$ so that, for every $\Delta \geq \Delta_0,$ there is a $l \in \NN$ such that
\begin{equation}
    \sup_{\hat\bfz \in \RR_+^{N-1}: \bar z_2 \geq \Delta}\mathbb{E}_{(2g,0,\hat{\bfz})}(\bar{Z}_2(\sigma_{2l}) - \bar z_2) < 0.
\end{equation}
\end{proposition}

\editc{This proposition will be proven using the following two lemmas. Proofs of the lemmas are given in Section \ref{sec:pfneginexp}.}

\begin{lemma}\label{zlesslem1}
There exists an $l_0 \in \mathbb{N}$ and $c_2 > 0$, such that for all $1\le k<N$, $\hat\bfz \in \RR_+^{N-1}$, and  $l \ge l_0$,
\begin{equation}\mathbb{E}^*_{(2g, 0,\hat\bfz)}(\bar{Z}_2(\sigma_{2l}) - \bar{z}_2) \leq -c_2l +\frac{N}{2k}\mathbb{E}^*_{(2g, 0,\hat\bfz)}L_{k+1}(\sigma_{2l}).\label{eq:1140}\end{equation}
\end{lemma}

To complete the proof of Proposition \ref{neginexp} we will  estimate, in the next lemma,  the second term in the bound \eqref{eq:1140}. 

Take $\Delta > 0$ and suppose $\hat\bfz \in \cls_{\Delta} \doteq \{\hat\bfz \in \RR_+^{N-1}: \bar{z}_2 \geq \Delta\}$.
Then there is a $ k \in \{1, \ldots, N-1\}$ so that 
\begin{equation}z_{k+1}\geq \frac{\Delta}{N^2}.\label{eq:1132}
\end{equation}
We will work with this $k$ in the following.

\begin{lemma}\label{zlesslem2}
For $\Delta>0$ and $\hat\bfz \in \cls_{\Delta}$, let $k = k(\Delta)$ satisfy \eqref{eq:1132}. There exist positive constants $\Delta_1, D_1, D_2,D_3$ such that for any $\Delta \ge \Delta_1$ and $l \in \mathbb{N}$,
\begin{align}\label{lkbd}
    \mathbb{E}^*_{(2g, 0,\hat\bfz)}L_{k+1}(\sigma_{2l}) 
	\le D_1l^{5/2}\editc{\left(\sqrt{l}e^{-D_2\sqrt{\Delta}/l}+e^{-D_3\Delta^{3/2}}\right)}.
\end{align}
\end{lemma}

\noindent {\bf Proof of Proposition \ref{neginexp}.}
With $l_0$ as in Lemma \ref{zlesslem1} and $\Delta_1$ as in Lemma \ref{zlesslem2},
let $\Delta_0' \doteq \max\{\Delta_1, l_0^4\}$. Setting $l = l(\Delta) =\lfloor\Delta^{1/4}\rfloor + 1$, we use Lemma \ref{zlesslem1} and Lemma \ref{zlesslem2} to obtain positive constants $c_2', D_1', D_2'$ such that for all $\Delta \ge \Delta_0'$ and $\hat\bfz \in \cls_{\Delta} $,
$$\mathbb{E}^*_{(2g,0,\hat\bfz)}(\bar{Z}_{2}(\sigma_{2l}) - \bar{z}_2) \leq -c_2'\Delta^{1/4}+\frac{ND_1'}{4k}\Delta^{3/4}e^{-D_2'\Delta^{1/4}}.$$
The result now follows upon taking $\Delta_0 \geq \Delta_0'$ such that the above bound is negative for all $\Delta \ge \Delta_0$.
\hfill \qed

The next proposition shows that $|\bar{Z}_2(\sigma_2)-\bar{Z}_2(0)|$ has a finite exponential moment.
\begin{proposition}\label{z2expmoment} There exists $\lala_6 > 0$ so that
$$\sup_{\hat{\bfz} \in  \RR_+^{N-1}}\mathbb{E}^*_{(2g,0,\hat{\bfz})}e^{\lala_6 |\bar{Z}_2(\sigma_2) - \bar{z}_2|} < \infty.$$
\end{proposition}
\begin{proof}
From \eqref{eq:semmart} and  \eqref{eql2l1}, under $\mathbb{P}^*_{(2g,0,\hat{\bfz})}$,
\begin{align*}
|\bar{Z}_2(\sigma_2)  -\bar{z}_2| &\le  |M(\sigma_2)| + \frac{N}{2}L_2(\sigma_2) + (N-1)L_1(\sigma_2)\\
&\le |M(\sigma_2)|  +  |\bar Y(\sigma_2)| + 2(N-1)L_1(\sigma_2).
\end{align*}
Also note that
$$L_1(\sigma_2) =2g+ g\sigma_2 - V(\sigma_2) = 2g+ g\sigma_2-2g = g\sigma_2.$$
Thus
$$
|\bar{Z}_2(\sigma_2)  -\bar{z}_2|  \le  2(N-1)g \sigma_2+  |M(\sigma_2)|+ |\bar Y(\sigma_2)|.$$ Hence, writing $Y^\circ(t) := |M(t)| + |\bar Y(t)|$, for any $\gamma >0$, using Cauchy Schwarz inequalty,
\begin{equation}\label{pp}
\mathbb{E}^*_{(2g, 0 ,\hat{\bfz})} e^{\gamma |\bar Z_2(\sigma_2) - \bar z_2|} \le \left(\mathbb{E}^*_{(2g,0,\hat{\bfz})}\,e^{4(N-1)g\gamma \sigma_2}\right)^{1/2} \left(\mathbb{E}^*_{(2g,0, \hat{\bfz})} e^{2\gamma |Y^\circ(\sigma_2)|}\right)^{1/2}. 
\end{equation}
Recall $\gamma_2$ from Proposition \ref{sigma2}, and write $D := \sup_{\hat{\bfz} \in  \RR_+^{N-1}}\,\,\mathbb{E}^*_{(2g,0,\hat{\bfz})}\,e^{\lala_2\sigma_2} < \infty$. Proceeding as in the proof of Lemma \ref{sec2prop1} \editc{(see Section \ref{pf:sec2prop1})}, observe using \eqref{eq:elemconc}, Proposition \ref{sigma2} and Markov's inequality that there exist $c, c'>0$ such that for any $\gamma \in (0,\gamma_2/2)$ and any $\hat{\bfz} \in \RR_+^{N-1}$,
\begin{align*}
\mathbb{E}^*_{(2g,0, \hat{\bfz})} e^{2\gamma |Y^\circ(\sigma_2)|} \le \sum_{k=0}^{\infty}\left(\mathbb{E}^*_{(2g,0, \hat{\bfz})} e^{4\gamma\sup_{0 \le s \le k+1} |Y^\circ(s)|}\right)^{1/2}(\mathbb{P}^*_{(2g,0, \hat{\bfz})}(\sigma_2 \geq k))^{1/2} \le c\sqrt{D}\sum_{k=0}^{\infty}e^{c'\gamma^2(k+1) - \gamma k}.
\end{align*}
The proposition follows from the above bound, \eqref{pp} and Proposition \ref{sigma2} upon choosing $\gamma \in (0, \min\{\gamma_2/(4(N-1)g), \gamma_2/2\})$ small enough so that the sum on the right side in the above display is finite.
\end{proof}

\subsection{Hitting Time of a Compact Set}\label{comphit}
Recall the sequence of stopping times $\{\sigma_j\}_{j\in \NN_0}$ introduced in \eqref{eq:sigmatimes} and the process $\bar Z_2$ defined in \eqref{barz2}. Also fix $\Delta \ge \Delta_0$ where $\Delta_0$ is as in
Proposition \ref{neginexp}.
Define 
\begin{align}\label{eq928}
    \Gamma'  \doteq \inf\{\sigma_{2k} \geq 0: k \in \mathbb{N}, \bar{Z}_2(\sigma_{2k}) \leq \Delta\}, \;\;
    \Gamma  \doteq \inf\{t \geq 0: Z_1(t) = 0, \bar{Z}_2(t) \leq \Delta, V(t) = 2g\}.
\end{align}
Recall that for $\hat\bfz \in \RR_+^{N-1}$, we write $\sum_{i=2}^N(N-i+1)\hat z_i$ as $\bar z_2$.

\begin{proposition}\label{gammabound}
    There exist $\lala_7 > 0$ and $c, c' > 0$ such that for any $\hat\bfz \in \RR_+^{N-1}$ with $\bar z_2 \ge \Delta$ and any $t \ge c' \bar z_2$,
$$\mathbb{P}^*_{(2g,0,\hat{\bfz})}(\Gamma > t) \leq ce^{-\lala_7 t}.$$
\end{proposition}
\begin{proof}
From the definition of the stopping times $\{\sigma_j\}_{j\in \NN_0}$ it follows that, for each $k \in \mathbb{N}$, $\sigma_{2k}$ is a point of decrease of the velocity process and, consequently, $Z_1(\sigma_{2k})=0$. Indeed, if this is not the case, one can produce an open interval containing $\sigma_{2l}$ where the velocity is strictly increasing, leading to a contradiction to the definition of $\sigma_{2k}$. 
Since $Z_1(\sigma_{2k})=0$, it follows that $\Gamma' \ge \Gamma$. It therefore suffices to show the result with $\Gamma$ replaced by $\Gamma'$.

By Proposition \ref{neginexp}, we can obtain $l_* \in \mathbb{N}$ and $\mu_*>0$ such that for any $j \in \mathbb{N}$, 
\begin{equation}\label{nm1}
 \sup_{\hat{\bfz} \in  \RR_+^{N-1} }  \left\lbrace \mathbb{E}_{(2g,0,\hat{\bfz})}(\bar{Z}_2(\sigma_{2jl_*}) - \bar{Z}_2(\sigma_{2(j-1)l_*}) \, \vert \, \clf^*_{\sigma_{2(j-1)l_*}}) + \mu_*\right\rbrace \textbf{1}_{\{\Gamma' > \sigma_{2(j-1)l^*}\}} \le 0,
\end{equation}
where $\clf^*_{\sigma_{2(j-1)l_*}}$ denotes the filtration associated with the process stopped at time $\sigma_{2(j-1)l_*}$.

Write $\mathcal{X}_j := \bar{Z}_2(\sigma_{2jl_*}) - \bar{Z}_2(\sigma_{2(j-1)l_*}), \, \tilde{\mathcal{F}}_{j-1} := \clf^*_{\sigma_{2(j-1)l_*}}, \, j \in \mathbb{N}$.
By Proposition \ref{z2expmoment} and the strong Markov property, 
$$\sup_{\hat{\bfz} \in  \RR_+^{N-1}}\sup_{j \in \mathbb{N}}\mathbb{E}^*_{(2g,0,\hat{\bfz})}\left(e^{\gamma_6|\mathcal{X}_j|} \, \vert \, \tilde{\mathcal{F}}_{j-1}\right)  < \infty.$$
This in particular says that $\sup_{\hat{\bfz} \in  \RR_+^{N-1}}\sup_{j \in \mathbb{N}}\mathbb{E}^*_{(2g,0,\hat{\bfz})}\left(|\mathcal{X}_j| \, \Big| \, \tilde{\mathcal{F}}_{j-1}\right) < \infty$. From these observations and Markov's inequality, we conclude that there exist positive constants $c_1, c_2$ such that for any $j \in \mathbb{N}$,
$$
\sup_{\hat{\bfz} \in  \RR_+^{N-1}}\mathbb{P}^*_{(2g,0,\hat{\bfz})}\left(|\mathcal{X}_j - \mathbb{E}(\mathcal{X}_j \, \vert \, \tilde{\mathcal{F}}_{j-1})| \ge x \, \vert \, \tilde{\mathcal{F}}_{j-1}\right) \le c_1 e^{-c_2 x}, \ x \ge 0.
$$
Hence, by \cite[Theorem 2.2]{wainwright2019high} and its proof, there exist non-negative numbers $(\nu,b)$ such that for any $j \in \mathbb{N}$,
$$
\sup_{\hat{\bfz} \in  \RR_+^{N-1}}\mathbb{E}^*_{(2g,0,\hat{\bfz})}\left(e^{\lambda (\mathcal{X}_j - \mathbb{E}(\mathcal{X}_j \, \vert \, \tilde{\mathcal{F}}_{j-1}))} \, \vert \, \tilde{\mathcal{F}}_{j-1} \right) \le e^{\nu^2\lambda^2/2}, \ \text{for all} \ |\lambda| < 1/b.
$$ 
Therefore, by \cite[Theorem 2.3]{wainwright2019high}, there exist positive constants $c_3, c_4$ such that for any $\hat{\bfz} \in  \RR_+^{N-1}$ with $\bar z_2 \ge \Delta$ and any $t > \left(\frac{2}{\mu_*} 
+ \frac{2}{\Delta}\right)\bar z_2$,
\begin{align}\label{5.15.1}
\mathbb{P}^*_{(2g,0,\hat{\bfz})}\left(\Gamma' > \sigma_{2l_*\lfloor t \rfloor}\right) &= \mathbb{P}^*_{(2g,0,\hat{\bfz})}\left(\sum_{j=1}^{\lfloor t \rfloor}\mathcal{X}_j > \Delta - \bar z_2, \, \Gamma' > \sigma_{2l_*\lfloor t \rfloor}\right)\notag\\
&\le \mathbb{P}^*_{(2g,0,\hat{\bfz})}\left(\sum_{j=1}^{\lfloor t \rfloor}(\mathcal{X}_j - \mathbb{E}(\mathcal{X}_j \, \vert \, \tilde{\mathcal{F}}_{j-1})) > \Delta - \bar z_2 + \mu_*\lfloor t \rfloor, \, \Gamma' > \sigma_{2l_*\lfloor t \rfloor}\right)\notag\\
&\le \mathbb{P}^*_{(2g,0,\hat{\bfz})}\left(\sum_{j=1}^{\lfloor t \rfloor}(\mathcal{X}_j - \mathbb{E}(\mathcal{X}_j \, \vert \, \tilde{\mathcal{F}}_{j-1})) > \Delta + \mu_*t/4 \right)\notag\\
& \le c_3 e^{-c_4 t}.
\end{align}
In the above display the first inequality is from \eqref{nm1} while the second inequality is from the facts that due to our condition on $\bar z_2$ and $t$ we have that  $\mu_*(t-1) > 2 \bar z_2$ and $t>2$, which says that
$$\mu_*\lfloor t \rfloor - \bar z_2 =\frac{1}{2}\mu_*\lfloor t \rfloor - \bar z_2 + \frac{1}{2}\mu_*\lfloor t \rfloor \ge \frac{1}{2}(\mu_*\lfloor t \rfloor- 2 \bar z_2) + \frac{1}{4}\mu_*t \ge \frac{1}{4}\mu_*t.
$$
Now, Proposition \ref{sigma2} and the strong Markov property imply that there exists $A \ge 1$ such that
$$
\sup_{\hat{\bfz} \in  \RR_+^{N-1}}\mathbb{E}^*_{(2g,0,\hat{\bfz})}\left(e^{\gamma_2 \sigma_{2l_*\lfloor t \rfloor }}\right) \le A^{\lfloor t \rfloor}, \ t > 0.
$$
Hence, taking $a>0$ such that $e^{\gamma_2a} >A$, we obtain positive constants $c_3', c_4'$ such that for any $t >0$,
\begin{equation}\label{5.15.2}
\sup_{\hat{\bfz} \in  \RR_+^{N-1}}\mathbb{P}^*_{(2g,0,\hat{\bfz})}\left(\sigma_{2l_*\lfloor t \rfloor } > a t\right)  \le c_3' e^{-c_4' t}.
\end{equation}
Using \eqref{5.15.1} and \eqref{5.15.2}, we conclude that there exist positive constants $c_5, c_6$ such that for any $\hat{\bfz} \in  \RR_+^{N-1}$ with $\bar z_2 \ge \Delta$ and any $t > \left(\frac{2}{\mu_*} 
+ \frac{2}{\Delta}\right)\bar z_2$,
$$
\mathbb{P}^*_{(2g,0,\hat{\bfz})}(\Gamma' > at) \le \mathbb{P}^*_{(2g,0,\hat{\bfz})}\left(\Gamma' > \sigma_{2l_*\lfloor t \rfloor}\right) + \sup_{\hat{\bfz} \in  \RR_+^{N-1}}\mathbb{P}^*_{(2g,0,\hat{\bfz})}\left(\sigma_{2l_*\lfloor t \rfloor } > a t\right) \le c_5 e^{-c_6 t}.
$$
The result follows upon taking $c= c_5, \gamma_7 =c_6$ and $c'=a\left(\frac{2}{\mu_*} 
+ \frac{2}{\Delta}\right)$.
\end{proof}

\subsection{Completing the Proof of Exponential Ergodicity}\label{ee}

In this section, we will complete the proof of Theorem \ref{thm:geomerg}.
We begin with the following proposition the proof of which will be completed in Section \ref{sec:driftcon}.
Fix $\Delta \ge \Delta_0$ where $\Delta_0$ is as in
Proposition \ref{neginexp}.
Define
\begin{equation}\label{eq:cstar}
	C^* \doteq \{(v,\bfz) \in \RR \times \RR_+^N: v=2g,\;  z_1= 0, \; \bar z_2\le \Delta\}.\end{equation}
Let
$\tau_{C^*}(1) \doteq  \inf\{t \geq 1: (V(t),\bfZ(t)) \in C^*\}$.
\begin{proposition}\label{driftcondn}$\;$
\begin{enumerate}
    \item There exists $\eta > 0$ such that
\begin{equation}\label{v0}
            \tilde{V}_0(v,\bfz) \doteq \mathbb{E}^*_{(v,\bfz)}e^{\eta\tau_{C^*}(1)} < \infty, \,\,\,\,\,\, \mbox{ for all }  (v,\bfz) \in \RR \times \RR_+^N.
        \end{equation}
		Furthermore,
$$\sup_{(v,\bfz) \in C^*}\tilde V_0(v,\bfz) \doteq  M <\infty.$$
 \item There exists a non-zero measure $\nu$  on $\clb(\RR\times \RR_+^N)$ and $r_1 \in (0, \infty)$ such that, for all $(v,\bfz) \in C^*$,
        $$\mathbb{P}^{r_1}((v,\bfz), A) \geq \nu(A) \mbox{ for all } A \in \clb(\RR\times \RR_+^N).$$
   \end{enumerate}
\end{proposition}
    Part (2) of the above proposition shows that, in the terminology of Down, Meyn and Tweedie (cf. \cite[Section 3]{DowMeyTwe}, the set $C^*$ 
        is $\nu-$petite (or small) for the Markov family $\{\PP_{(v,\bfz)}\}_{(v,\bfz) \in \RR \times \RR_+^N}$.
Together with part (1) of the proposition, this shows that
 the conditions of \cite[Theorem 6.2]{DowMeyTwe} are satisfied and consequently, the function $V_0$ defined as
 \begin{equation}\label{lyapunov}
    V_0(v,\bfz) \doteq 1 - \frac{1}{\eta} + \frac{1}{\eta}\tilde{V}_0(v,\bfz), \; (v, \bfz) \in \RR \times \RR_+^{N},
\end{equation}
satisfies the drift condition $(\cld_T)$ in \cite[Section 5]{DowMeyTwe}. 
We will now like to apply \cite[Theorem 5.2]{DowMeyTwe} to conclude the proof of exponential ergodicity. For this we show in the next two results that the Markov process $\{\PP_{(v,\bfz)}\}_{(v,\bfz) \in \RR \times \RR_+^N}$ is irreducible and aperiodic.

Recall the set $D$ from Theorem \ref{minorization}.
\begin{proposition}\label{irred}
	Define the measure $\psi$ on $\clb(\RR\times \RR_+^N)$ as $\psi(A)\doteq \lambda(A\cap D)$, $A \in \clb(\RR\times \RR_+^N)$. Then the Markov process $\{\PP_{(v,\bfz)}\}_{(v,\bfz) \in \RR \times \RR_+^N}$ is
	$\psi$-irreducible.
\end{proposition}

\begin{proof}
Fix $(v,\bfz) \in \mathbb{R} \times \mathbb{R}_+^N $. Let $B \in \mathcal{B}(\mathbb{R}\times \mathbb{R}_+^N)$ be such that $\lambda(B\cap D) > 0$. To establish $\psi$-irreducibility it suffices to show
\begin{equation}
	\mathbb{E}^*_{(v,\bfz)} \int_0^{\infty} \textbf{1}_{\{(V(t),\bfZ(t)) \in B\}}dt >0.
\end{equation}
From Theorem \ref{minorization}, for each $t \in [\sn, \en]$ and $(v', \bfz') \in R= (0,\frac{g}{128}) \times (0, \infty) \times \mathbb{R}_+^{N-1}$, 
\begin{equation*}
    \mathbb{P}^t((v',\bfz'), B) \ge K_{(v',\bfz')}\lambda(B \cap D).
\end{equation*}
Also, from Lemma \ref{uniqlem}, for any $(v, \bfz) \in \RR \times \RR_+^{N}$, there exists $r_0 \doteq r_0(v,\bfz) \in \mathbb{N}$ such that 
\begin{equation}
    \mathbb{P}^{r_0}((v,\bfz), R) > 0. \label{eq:irredpos}
\end{equation}
Observe that for $t \in [r_0 + \sn, r_0 + \en],$
\begin{align*}
    \mathbb{P}^t((v,\bfz), B) & = \int_{\mathbb{R}\times\mathbb{R}_+^{N}} \mathbb{P}^{t-r_0}((v',\bfz'), B) d\mathbb{P}^{r_0}((v,\bfz), dv',d\bfz') 
     \geq \lambda(B\cap D)\int_R K_{(v',\bfz')} d\mathbb{P}^{r_0}((v,\bfz), dv',d\bfz'). 
\end{align*}
The latter expression is strictly positive in view of \eqref{eq:irredpos}, the  positivity of $K_{(v,\bfz)}$ for
$(v,\bfz) \in R$ and our assumption concerning $B$. 
Finally note that
\begin{align*}
    \mathbb{E}^*_{(v,\bfz)}\int_0^{\infty}\textbf{1}_{\{(V(t),\bfZ(t)) \in B\}}dt &= \int_0^{\infty}\mathbb{P}^t ((v,\bfz), B)dt 
     \geq \int_{r_0+\sn}^{r_0+\en}\mathbb{P}^t ((v,\bfz), B)dt > 0.
\end{align*}
The result follows.
\end{proof}

\begin{proposition}\label{aper}

The Markov process $\{\PP_{(v,\bfz)}\}_{(v,\bfz) \in \RR \times \RR_+^N}$ is aperiodic.

\end{proposition}

\begin{proof}
Recall the set $C$ and the constant $\bar{K}_{\bar A}$ from
Theorem \ref{minorization} and let $\bar{K} \doteq \bar{K}_C$.
Define the measure $\nu$ on $\clb(\RR \times \RR_+^N)$ as
$\nu(B) \doteq \bar{K}\lambda(B \cap D)$, for $B \in \clb(\RR \times \RR_+^N)$. 
From Theorem \ref{minorization} it follows that the set $C$ in the statement of the theorem is 
 $\nu$-small. Hence, for aperiodicity, it suffices to show that, for some $t_0>0$
\begin{equation}\label{apersuf}
	\mathbb{P}^{t}((v,\bfz),C) > 0, \,\,\,\,\,\mbox{ for all } t \geq t_0, \mbox{ and } (v,\bfz) \in C.
\end{equation}
Since $\lambda(C \cap D) > 0$, we have that \eqref{apersuf} holds for $t \in [\sn, \en]$  and all
$(v,\bfz) \in C$. Let $\delta = \en-\sn$,
We now claim that, for all $m \in \NN$,
$$\mathbb{P}^t((v,\bfz),C) > 0, \,\,\,\,\,\mbox{ for all } t \in [m\sn, m\sn + m\delta]  \mbox{ and } (v,\bfz) \in C.$$
Indeed, clearly the result is true  with $m=1$, and if the result is true with $m =k$ then it is also true for $m=k+1$ since any $t\in [(k+1)\sn, (k+1)\sn + (k+1)\delta]$ can be written as $t_1+t_2$ with
$t_1 \in [k\sn, k\sn + k\delta]$ and $t_2 \in [\sn, \sn+\delta]$, and
$$\mathbb{P}^t((v,\bfz),C) \ge \int_C  P^{t_1}((v,\bfz), (d\tilde v, d\tilde \bfz)) P^{t_2}((\tilde v,\tilde \bfz),C) >0 \mbox{ for all } (v,\bfz) \in C. $$
Now choose $k_0\in \NN$ such that $k_0\delta \ge \sn$. Then
$\mathbb{P}^t((v,\bfz),C) > 0$ for all $(v,\bfz) \in C$ and $t \in [k\sn, (k+1)\sn]$ for all $k \ge k_0$. We conclude that $\mathbb{P}^t((v,\bfz),C) > 0$ for all $(v,\bfz) \in C$ and for all $t\ge k_0\sn$. The result follows.
\end{proof}

We can now complete the proof of exponential ergodicity.\\

\noindent {\bf Proof of Theorem \ref{thm:geomerg}.}
As noted previously, Proposition \ref{driftcondn}  shows that
 the conditions of \cite[Theorem 6.2]{DowMeyTwe} are satisfied and consequently, the function $V_0$ defined in 
 \eqref{lyapunov}
satisfies the drift condition $(\cld_T)$ in \cite[Section 5]{DowMeyTwe}. Also from Propositions \ref{irred}
and \ref{aper} the Markov process is $\psi$-irreducible and aperiodic. The result is now immediate from 
\cite[Theorem 5.2]{DowMeyTwe}.

%
%
%
%
%
%
%
%
%
%


\section{Proofs of Some Results from Section  \ref{sec:geomerg}.}

\label{sec:techlem}

In this section we present proofs of some technical results stated without proof in Section \ref{sec:geomerg}.

\subsection{Proofs of Lemmas for Proposition \ref{sigma1}.}\label{props1}

In this section we provide the proofs of Lemmas \ref{sec1prop1}, \ref{sec1prop2}, and \ref{sec1prop3} stated in Section \ref{highlev} that were used in the proof of Proposition \ref{sigma1}.

\subsubsection{Proof of Lemma \ref{sec1prop1}.}
\label{sec:pfsec1prop1}

Fix ${\bfz} \in \RR_+^{N}$. All inequalities in the proof will be a.s. under $\mathbb{P}^*_{(0,{\bfz})}$.
 Using \eqref{LTineq}, we have that for $t \leq \hat\tau_{g/(2N)}$,
$$L_1(t) \leq \sum_{i=1}^NW_{1,i}B_i^*(t) + \frac{gW_{1,1}t}{2N}.$$
It can be verified that
\begin{equation}
	W_{1,1} = N, \mbox{ and }  W_{i,1} = 2N - 2(i - 1), \,\, i = 2,...,N.
\end{equation}
Using this and since $V(t) = gt - L_1(t)$, it follows that
\begin{align*}
    V(t) &\geq -\sum_{i=1}^NW_{1,i}B_i^*(t) + g(1 - \frac{W_{1,1}}{2N})t 
     = -\sum_{i=1}^NW_{1,i}B_i^*(t) + \frac{gt}{2} \doteq Q(t).
\end{align*}
%
%
Define $\hat \sigma_{g/(2N)} \doteq \inf\{t\ge 0: Q(t) = g/(2N)\}$.  Then the  above inequality implies that $\hat\sigma_{g/(2N)}\geq \hat\tau_{g/(2N)}$. 
By a standard concentration bound (see \eqref{eq:elemconc})
 it follows that there are $\vrr_1, \vrr_2 \in (0, \infty)$ such that 
\begin{equation}\label{eq:concbm}
	\mathbb{E}^*_{(0,{\bfz})}e^{\theta\sum_{i=1}^NW_{1,i}B_i^*(s)} \le \vrr_1 e^{\vrr_2\theta^2 s} \mbox{ for all } s \ge 0 \mbox{ and } \theta \in (0, \infty).
\end{equation}
Then, for an arbitrary $\theta, \beta > 0$, we have
\begin{align*}
    \mathbb{E}^*_{(0,{\bfz})}e^{\beta \hat\tau_{g/(2N)}} & = \int_0^\infty \mathbb{P}^*_{(0,{\bfz})}(\hat\tau_{g/(2N)} > \frac{\ln(s)}{\beta})\,\,ds \\
    & \leq \int_0^\infty\mathbb{P}^*_{(0,{\bfz})}(\hat\sigma_{g/(2N)} > \frac{\ln(s)}{\beta})\,\,ds \\
    & \leq \int_0^\infty\mathbb{P}^*_{(0,{\bfz})}(Q(\frac{\ln(s)}{\beta}) < \frac{g}{2N})\,\,ds \\
    & \le 1 + \int_1^\infty\mathbb{P}^*_{(0,{\bfz})}(\frac{g\ln(s)}{2\beta} < \frac{g}{2N}+\sum_{i=1}^NW_{1,i}B_i^*(\frac{\ln(s)}{\beta}))\,\,ds \\
    & \leq 1 + e^{\theta g/(2N)}\int_1^\infty e^{-\theta g\ln(s)/2\beta}\mathbb{E}^*_{(0,{\bfz})}e^{\theta\sum_{i=1}^NW_{1,i}B_i^*(\frac{\ln(s)}{\beta})}\,ds \\
    & \leq 1 + \vrr_1e^{\theta g/(2N)}\int_1^\infty s^{-\theta g/2\beta}s^{\theta^2\vrr_2/\beta}\,ds.
\end{align*}
Now take
$$\theta \doteq \frac{g}{4\vrr_2}, \;\; \beta \doteq \frac{\theta g}{8}.$$
Then
$$-\theta g/2\beta + \theta^2\vrr_2/\beta = -2.$$
The result follows. \hfill \qed

\subsubsection{Proof of Lemma \ref{sec1prop2}.}
We will first show that
\begin{equation}\label{step1bb}
	\inf_{(v, \bfz) \in [\frac{g}{4N},4g]\times [1, \infty) \times \RR_+^{N-1}}\mathbb{P}^*_{(v,\bfz)}(\hat \tau_{4g} < \hat\tau_{0}) \doteq  p_1 >0.\end{equation}
Note that, for $t>0$, on the set $\{\hat \tau_{4g} >t\}$, for $(v, \bfz) \in [\frac{g}{4N},4g]\times\RR_+^N$,
under $\mathbb{P}^*_{(v,\bfz)}$,
\begin{align*}
L_1(t) &\le \sup_{0\leq s\leq t}(-z_1-B_1(s)+\frac{1}{2}L_2(s)+4gs)^+	\le \sup_{0\leq s\leq t}(-z_1-B_1(s)+4gs)^+ + \frac{1}{2}L_2(t)\\
&\le \sup_{0\leq s\leq t}(-z_1-B_1(s)+4gs)^+ + \frac{(N-1)}{N} L_1(t) + \frac{1}{N} \bar Y(t),
\end{align*}
where the last inequality uses \eqref{eql2l1}.
Thus
\begin{equation}\label{eq:1033}
	L_1(t) \le N \sup_{0\leq s\leq t}(-z_1-B_1(s)+4gs)^+ + \bar Y(t).\end{equation}

Consider the set $A_1 \in \clf^*$ defined as
$$A_1 \doteq \{ -B_1(s)+ 4gs -1 < 0 \mbox{ for all } s \in [0,8] \mbox{ and } \bar Y(8)< g/8N\}.$$
Note that
$$\inf_{(v,\bfz)\in \RR \times \RR_+^N} \PP_{(v,\bfz)}^*(A_1) \doteq p_1'>0.$$
Also, for $(v, \bfz) \in [\frac{g}{4N},4g]\times [1, \infty) \times \RR_+^{N-1}$, under $\PP_{(v,\bfz)}^*$, on $A_1$,
$$L_1(8\wedge \hat \tau_{4g}) \le N \sup_{0\leq s\leq 8\wedge \hat \tau_{4g}}(-z_1-B_1(s)+4gs )^+ + \bar Y(8) = \bar Y(8) < g/8N < 4g.$$
So, in particular,
$$V(8\wedge \hat \tau_{4g}) = V(\hat \tau_{4g})1_{\{\hat \tau_{4g}\le 8\}} + V(8)1_{\{\hat \tau_{4g}> 8\}}
\ge 4g 1_{\{\hat \tau_{4g}\le 8\}}  + (8g-4g)1_{\{\hat \tau_{4g}> 8\}}
= 4g$$
and consequently $\hat \tau_{4g}\le 8$.
Also, under the same conditions, for $s<8$,
$$V(s\wedge \hat \tau_{4g}) \ge v- L_1(s\wedge \hat \tau_{4g}) \ge v - L_1(8 \wedge \hat \tau_{4g}) > \frac{g}{4N} - \frac{g}{8N} >0.$$
Thus we have
$$p_1 = \inf_{(v, \bfz) \in [\frac{g}{4N},4g]\times [1, \infty) \times \RR_+^{N-1}}\mathbb{P}^*_{(v,\bfz)}(\hat \tau_{4g} < \hat\tau_{0}) \ge \inf_{(v,\bfz)\in \RR \times \RR_+^N} \PP_{(v,\bfz)}^*(A_1) = p_1'>0.$$
This proves \eqref{step1bb}.

Let $\nu_1 \doteq \inf\{t \ge 0: Z_1(t)\ge 1\}$.
In order to complete the proof, from the strong Markov property, it suffices to show that
\begin{equation}\label{eq:p2bd}
	\inf_{(v, \bfz) \in [\frac{g}{2N},2g] \times \RR_+^N} \mathbb{P}^*_{(v,\bfz)}(\nu_1 \wedge \hat \tau_{4g} < \hat\tau_{g/4N}) \doteq p_2>0.
\end{equation}
Fix $\delta \in (0,1)$ such that
$$2gN\delta + \frac{1}{2} gN\delta^2 \le \frac{g}{16N}.$$
Define $A_2 \in \clf^*$ as 
$$A_2 \doteq \{B_1(\delta) \ge 1 + 4gN\delta + gN\delta^2 + \frac{3g}{16N}, \; \bar Y(\delta) + N B_1^*(\delta) \le \frac{g}{16N}\}.$$
It is easy to check that
$$\inf_{(v,\bfz)\in [\frac{g}{2N},2g]\times \RR_+^N} \PP_{(v,\bfz)}^*(A_2) \doteq p_2'>0.$$
Furthermore, as in \eqref{eq:1033}, for $(v, \bfz) \in [\frac{g}{2N},2g]\times \RR_+^N$, under $\PP_{(v,\bfz)}^*$, on $A_2$,
$$
L_1(\delta) \le N B_1^*(\delta) + \bar Y(\delta) + N \int_0^{\delta} V^+(s) ds
\le N B_1^*(\delta) + \bar Y(\delta) + 2gN\delta + \frac{1}{2} gN\delta^2.$$
Also, under the same conditions, from \eqref{eql2l1},
$$
L_2(\delta) \le 2 L_1(\delta) + \frac{2}{N} \bar Y(\delta) \le
2N B_1^*(\delta) + 2\bar Y(\delta) + 4gN\delta +  gN\delta^2 + \frac{2}{N} \bar Y(\delta).$$
Thus
\begin{align*}
Z_1(\delta) &= z_1 + B_1(\delta) + L_1(\delta) - \frac{1}{2} L_2(\delta) - \int_0^{\delta} V(s) ds	\\
&\ge 1 + 4gN\delta + gN\delta^2 + \frac{3g}{16N}  - N B_1^*(\delta)  - \bar Y(\delta) - 2gN\delta - \frac{1}{2} gN\delta^2 - \frac{1}{N} \bar Y(\delta) - 2g\delta - \frac{1}{2} g\delta^2\\
&\ge 1.
\end{align*}
Again, under the same conditions, for $0 \le s \le \delta$,
\begin{align*}
	V(s) &\ge \frac{g}{2N} - L_1(\delta) \ge \frac{g}{2N} 
	 -N B_1^*(\delta) - \bar Y(\delta) - 2gN\delta - \frac{1}{2} gN\delta^2\\
	 &\ge \frac{g}{2N} - \frac{g}{16N}  - 2gN\delta - \frac{1}{2} gN\delta^2 \ge \frac{g}{2N} - \frac{g}{16N} - \frac{g}{16N} > \frac{g}{4N}.
\end{align*}
It then follows
\begin{align*}
	p_2 &= \inf_{(v, \bfz) \in [\frac{g}{2N},2g] \times \RR_+^N} \mathbb{P}^*_{(v,\bfz)}(\nu_1 \wedge \hat \tau_{4g} < \hat\tau_{g/4N})\\
	&\ge \inf_{(v, \bfz) \in [\frac{g}{2N},2g] \times \RR_+^N} \mathbb{P}^*_{(v,\bfz)}(\nu_1  < \hat\tau_{g/4N})\\
	&\ge \inf_{(v, \bfz) \in [\frac{g}{2N},2g] \times \RR_+^N} \mathbb{P}^*_{(v,\bfz)}(A_2) = p_2' >0.
\end{align*}
This proves \eqref{eq:p2bd} and completes the proof of the lemma.
\hfill \qed

\subsubsection{Proof of Lemma \ref{sec1prop3}.}
By the strong Markov property, it suffices to show that for some $m \in \NN$
\begin{equation}\label{eq:mpos}
	\inf_{(v,\bfz) \in [0,4g] \times \RR_+^N}\mathbb{P}^*_{(v,\bfz)}(\hat\tau_{4g}\wedge\hat\tau_{0} \le  m)  >0.
\end{equation}
We will prove \eqref{eq:mpos} with  $m=5$.	
We consider two cases:\\
{\bf Case 1:} $z_1\ge 1$.  
Define $A_1 \in \clf^*$ as
$$A_1 \doteq \{-B_1(s) +4gs -1 \le 0 \mbox{ for all } 0 \le s \le 5, \; \bar Y(5) < g\}.$$
It is easily seen that
$$\inf_{(v,\bfz) \in [0,4g] \times \RR_+^N}\mathbb{P}^*_{(v,\bfz)}(A_1) \doteq \kappa_1 >0.$$
From \eqref{eql2l1} and \eqref{eq:1033} it follows that, for $(v,\bfz) \in [0,4g] { \times [1,\infty) \times \RR_+^{N-1}}$, under $\mathbb{P}^*_{(v,\bfz)}$, on
$A_1 \cap \{\hat\tau_{4g}\wedge\hat\tau_{0} > 5\}$,
$$L_1(5) \le N \sup_{s\le 5} (-1+4gs - B_1(s))^+ + \bar Y(5) < g,$$
and consequently
$$V(5) \ge 5g- L_1(5) > 5g-g = 4g.$$
This says that $A_1 \cap \{\hat\tau_{4g}\wedge\hat\tau_{0} > 5\}$ is $\mathbb{P}^*_{(v,\bfz)}$ trivial and so
$$\inf_{(v,\bfz) \in [0,4g] { \times [1,\infty) \times \RR_+^{N-1}}}\mathbb{P}^*_{(v,\bfz)}(\hat\tau_{4g}\wedge\hat\tau_{0} \le 5) \ge \inf_{(v,\bfz) \in [0,4g] \times \RR_+^N}\mathbb{P}^*_{(v,\bfz)}(A_1) = \kappa_1 >0.$$
This proves \eqref{eq:mpos} when $z_1 \ge 1$.\\

\noindent{\bf Case 2:} $z_1<1$.
Define $A_2 \in \clf^*$ as
$$A_2 \doteq \{B_1(5)<-1-9g\}.$$
Clearly 
$$\inf_{(v,\bfz) \in [0,4g] \times \RR_+^N}\mathbb{P}^*_{(v,\bfz)}(A_2) \doteq \kappa_2 >0.$$
Also, for $(v,\bfz) \in [0,4g] { \times [0,1) \times \RR_+^{N-1}}$, under $\mathbb{P}^*_{(v,\bfz)}$, on
$A_2 \cap \{\hat\tau_{4g}\wedge\hat\tau_{0} > 5\}$,
$$
L_1(5) = \sup_{0\le s \le 5} (-z_1 + \frac{1}{2} L_2(s) + \int_0^s V(u) du - B_1(s))^+\ge 
\sup_{0\le s \le 5}(-1 - B_1(s))^+ > 9g$$
and consequently
$$V(5) \le 4g+ 5g - L_1(5) <0.$$
This shows that  $A_2 \cap \{\hat\tau_{4g}\wedge\hat\tau_{0} > 5\}$ is $\mathbb{P}^*_{(v,\bfz)}$ trivial and so
 $$\inf_{(v,\bfz) \in [0,4g] { \times [0,1) \times \RR_+^{N-1}}}\mathbb{P}^*_{(v,\bfz)}(\hat\tau_{4g}\wedge\hat\tau_{0} \le 5 ) \ge \inf_{(v,\bfz) \in [0,4g] \times \RR_+^N}\mathbb{P}^*_{(v,\bfz)}(A_2) = \kappa_2 >0.$$
This completes the proof of \eqref{eq:mpos} when $z_1 < 1$.
The result follows.
\hfill \qed

\subsection{Proofs of Lemmas for Proposition \ref{sigma2}.}
\label{sec:pfsigma2}

In this section we provide the proofs of Lemmas \ref{sec2prop1}, \ref{sec2prop2}, and \ref{sec2prop3}  stated in Section \ref{sec:lowlev} that were used in the proof of Proposition \ref{sigma2}.

\subsubsection{Proof of Lemma \ref{sec2prop1}.}\label{pf:sec2prop1}
Fix $\hat{\bfz} \in \mathbb{R}^{N-1}$. All inequalities in this proof are $\mathbb{P}^*_{(2g,0,\hat{\bfz})}$-a.s. Observe that $4g = V(\sigma_1) = g\sigma_1 - L_1(\sigma_1) + 2g$, so that $L_1(\sigma_1) = g\sigma_1 - 2g$. Then
\begin{align*}
    Z_1(\sigma_1) &= B_1(\sigma_1) - \frac{1}{2}L_2(\sigma_1) + L_1(\sigma_1) - \int_0^{\sigma_1}V(s)ds \\
    &\leq \sup_{0\leq s \leq \sigma_1}(B_1(s)) + g\sigma_1 - \frac{g\sigma_1^2}{2} - 2g\sigma_1 + \int_0^{\sigma_1}L_1(s)ds \\
    &\leq \sup_{0\leq s \leq \sigma_1}(B_1(s)) + L_1(\sigma_1)\sigma_1 \\
    &\leq \sup_{0\leq s \leq \sigma_1}(B_1(s)) + g \sigma_1^2.
\end{align*}
Thus, for $\beta > 0$,
\begin{align}
    e^{\beta(Z_1(\sigma_1))^{1/2}} &\leq e^{\beta(\sup_{0\leq s \leq \sigma_1}(B_1(s)) + g \sigma_1^2)^{1/2}} \\
     &\leq e^{\beta(\sup_{0\leq s \leq \sigma_1}B_1(s))^{1/2} + \beta \sqrt{g}\sigma_1}
    \leq \frac{1}{2}e^{2\beta((\sup_{0\leq s \leq \sigma_1}B_1(s))^{1/2}} + \frac{1}{2}e^{2\beta\sqrt{g}\sigma_1},
	\label{eq:118}
\end{align}
where in the final step we use Young's inequality. We now estimate each of the terms in \eqref{eq:118}. 
We begin by recalling that from Proposition  \ref{sigma1}, we can find $\beta_0 \in (0, 1/2)$ such that
\begin{equation}
	\sup_{\hat{\bfz} \in \mathbb{R}_+^{N-1} }\mathbb{E}^*_{(2g,0,\hat{\bfz})}e^{\beta_0\sigma_1} \doteq c(\beta_0)<\infty.
\end{equation}
Hence, taking $\beta \in (0, \beta_0/(2\sqrt{g})]$, the second term in \eqref{eq:118} is bounded as
\begin{equation}
\sup_{\hat{\bfz} \in \mathbb{R}_+^{N-1} }\mathbb{E}^*_{(2g,0,\hat{\bfz})}e^{2\beta\sqrt{g}\sigma_1} \le c(\beta_0).
\end{equation}
With $\beta \in (0, \beta_0]$ for the first term in \eqref{eq:118}, we have,
\begin{align}
    \mathbb{E}^*_{(2g,0,\hat{\bfz})}e^{2\beta(\sup_{0\leq s \leq \sigma_1}B_1(s))^{1/2}} 
    &\le e^{2\beta} + \mathbb{E}^*_{(2g,0,\hat{\bfz})}\,e^{2\beta\sup_{0\leq s \leq \sigma_1}B_1(s)}\textbf{1}_{\{\sup_{0\leq s \leq \sigma_1}B_1(s) > 1\}}\label{eq:eq145} \\
    & \leq  e^{2\beta} + \sum_{k=0}^{\infty}\mathbb{E}^*_{(2g,0,\hat{\bfz})}\,e^{2\beta\sup_{0\leq s \leq \sigma_1}B_1(s)}\textbf{1}_{\{k \leq \sigma_1 < k+1\}}\nonumber \\
    & \leq  e^{2\beta} + \sum_{k=0}^{\infty}\mathbb{E}^*_{(2g,0,\hat{\bfz})}\,e^{2\beta\sup_{0\leq s \leq k+1}B_1(s)}\textbf{1}_{\{k \leq \sigma_1 < k+1\}}.\nonumber
\end{align}
Using Cauchy-Schwarz inequality,
\begin{align}
    \mathbb{E}^*_{(2g,0,\hat{\bfz})}e^{2\beta(\sup_{0\leq s \leq \sigma_1}B_1(s))^{1/2}} 
    & \leq e^{2\beta} + \sum_{k=0}^{\infty}(\mathbb{E}^*_{(2g,0,\hat{\bfz})}\,e^{4\beta\sup_{0\leq s \leq k+1}B_1(s)})^{1/2}(\mathbb{P}^*_{(2g,0,\hat{\bfz})}(\sigma_1 \geq k))^{1/2} \nonumber\\
    & \leq e^{2\beta} + \vrr_1\sum_{k=0}^{\infty}e^{8\beta^2\vrr_2(k+1)}(\mathbb{P}^*_{(2g,0,\hat{\bfz})}(\sigma_1 \geq k))^{1/2} \nonumber\\
    & \leq e^{2\beta} + \vrr_1 c(\beta_0)^{1/2} e^{8\beta^2\vrr_2} \sum_{k=0}^{\infty}e^{8\beta^2\vrr_2k -\frac{\beta k}{2}} \doteq c_1(\beta) <\infty, \label{eq:907}
\end{align}
where the finiteness follows on choosing $\beta \in (0, \beta_1]$ for sufficiently small $\beta_1 \in (0, \beta_0]$. The second line above follows from a standard concentration inequality (see \eqref{eq:elemconc}) and the last line from Markov's inequality. 
Thus for any $\beta \in (0,\beta_1]$,
\begin{equation*}
    \sup_{\hat{\bfz}  \in \mathbb{R}_+^{N-1}}\mathbb{E}^*_{(2g,0,\hat{\bfz})}\,e^{2\beta(\sup_{0\leq s \leq \sigma_1}B_1(s))^{1/2}}
	\doteq c_1(\beta) < \infty.
\end{equation*}
The result now follows on setting $\gamma_3 = \min\{\beta_0/(2\sqrt{g}),\beta_1\}$.
\hfill \qed
\subsubsection{Proof of Lemma \ref{sec2prop2}.}

Let $(z_1,\hat{\bfz}) \in (0,\infty) \times \mathbb{R}_+^{N-1}$. All inequalities of random quantities in this proof are $\mathbb{P}^*_{(4g,z_1,\hat{\bfz})}$-almost sure. For $ t \leq \tau_0^{Z_1}$, we have
\begin{align*}
    Z_1(t) & = z_1 + B_1(t) - \frac{1}{2}L_2(t) - \int_0^tV(s)ds \\
    & = z_1 + B_1(t) - \frac{1}{2}L_2(t) - \int_0^t(gs + 4g)ds \\
    & \leq z_1 + \sup_{0 \leq s \leq t}B_1(s) - \frac{gt^2}{2} \doteq H(t).
\end{align*}
Consequently, $Z_1(t)$ must hit zero before $H(t)$, and so $\tau_0^H \doteq \inf\{t\ge0: H(t) =0\} \geq \tau_0^{Z_1}$. Thus,
for arbitrary $\lala>0$,
\begin{align*}
    \mathbb{E}^*_{(4g,z_1,\hat{\bfz})}e^{\lala\tau_0^{Z_1}} &= 1 + \int_1^{\infty}\mathbb{P}^*_{(4g,z_1,\hat{\bfz})}(\tau_0^{Z_1} > \frac{\ln(s)}{\lala}) ds \\
    & \leq 1 + \int_1^{\infty}\mathbb{P}^*_{(4g,z_1,\hat{\bfz})}(\tau_0^{H} > \frac{\ln(s)}{\lala}) ds\\
    & \leq 1 + \int_1^{\infty}\mathbb{P}^*_{(4g,z_1,\hat{\bfz})}(H(\frac{\ln(s)}{\lala}) > 0)  ds.
\end{align*}
Thus, using Markov's inequality, for $\theta >0$,	
\begin{align*}
\mathbb{E}^*_{(4g,z_1,\hat{\bfz})}e^{\lala \tau_0^{Z_1}}     & \le 1 + \int_1^{\infty}\mathbb{P}^*_{(4g,z_1,\hat{\bfz})}(z_1 + \sup_{0 \leq u \leq \frac{\ln(s)}{\lala}}B_1(s) > \frac{g(\frac{\ln(s)}{\lala})^2}{2} )ds \\
    & \leq 1 + \int_1^{\infty}\mathbb{P}^*_{(4g,z_1,\hat{\bfz})}(z_1^{1/2} +(\sup_{0 \leq u \leq \frac{\ln(s)}{\lala}}B_1(s))^{1/2} > \sqrt{\frac{g}{2}}\frac{\ln(s)}{\lala})ds \\
    & \leq 1 + e^{\theta z_1^{1/2}}\int_1^{\infty}s^{-\sqrt{\frac{g}{2}}\frac{\theta}{\lala}} \mathbb{E}^*_{(4g,z_1,\hat{\bfz})}e^{\theta(\sup_{0 \leq u \leq \frac{\ln(s)}{\lala}}B_1(s))^{1/2}}ds \\
    & \leq 1 + \vrr_1e^{\theta z_1^{1/2}}\int_1^{\infty}s^{-\sqrt{\frac{g}{2}}\frac{\theta}{\lala}}s^{\frac{\vrr_2\theta^2}{\lala}}ds,
\end{align*}
where in the last line we have used a standard concentration inequality (see \eqref{eq:elemconc}).
Now take $\lala_4 \doteq g/(16\vrr_2)$ and for fixed $\lala \in (0, \lala_4]$, take $\theta = 4\sqrt{2}\lala/\sqrt{g}$.
Then it follows that 
$$-\sqrt{\frac{g}{2}}\frac{\theta}{\lala} + \frac{\vrr_2\theta^2}{\lala} { \leq} -2.$$
Thus
$$
\sup_{\hat{\bfz} \in \mathbb{R}_+^{N-1}}\mathbb{E}^*_{(4g,z_1,\hat{\bfz})}e^{\lala \tau_0^{Z_1}}
\le  1 + \vrr_1e^{4\sqrt{2} \lala z_1^{1/2}/\sqrt{g}}.$$
%
%
%
%
The result follows. \hfill \qed

\subsubsection{Proof of Lemma \ref{sec2prop3}.}
Let $v \in [2g, \infty),\,\hat{\bfz} \in \mathbb{R}_+^{N-1}$, and $\lala > 0$. All inequalities of random quantities in this proof are $\mathbb{P}^*_{(v,0,\hat{\bfz})}$-almost sure. For $t \leq \hat \tau_{2g}$, $V(t) \geq 2g$, so
\begin{align*}
    0 \leq Z_1(t) & = B_1(t) + L_1(t) - \frac{1}{2}L_2(t) - \int_0^tV(s)ds 
     \leq \sup_{0 \leq s \leq t}B_1(s) + L_1(t) - 2gt,
\end{align*}
from which it follows that 
$
    - L_1(t) \leq \sup_{0 \leq s \leq t}B_1(s) - 2gt.
$
Hence,
$$V(t) = gt - L_1(t) + v \leq \sup_{0 \leq s \leq t}B_1(s) - gt + v \doteq Q(t).$$
From this inequality we  see that $\tau^Q_{2g} \doteq \inf\{t\ge 0: Q(t)=2g\} $ satisfies $\tau^Q_{2g} \geq \hat \tau_{2g}$. Then, for any $\theta > 0$,
\begin{align*}
    \mathbb{E}^*_{(v,0,\hat{\bfz})}e^{\lala\hat \tau_{2g}} &= \int_0^\infty \mathbb{P}^*_{(v,0,\hat{\bfz})}(\hat \tau_{2g} \geq \frac{1}{\lala}\ln(s))ds 
     \leq \int_0^\infty \mathbb{P}^*_{(v,0,\hat{\bfz})}(\tau^Q_{2g} \geq \frac{1}{\lala}\ln(s))ds \\
    & \leq 1 + \int_1^{\infty}\mathbb{P}^*_{(v,0,\hat{\bfz})}(Q(\frac{1}{\lala}\ln(s)) > 2g)ds .
\end{align*}
Thus by Markov's inequality,
\begin{align*}	
   \mathbb{E}^*_{(v,0,\hat{\bfz})}e^{\lala\hat \tau_{2g}}    & \leq 1 + e^{-2g\theta}\int_1^{\infty}\mathbb{E}^*_{(v,0,\hat{\bfz})}e^{\theta Q(\frac{1}{\lala}\ln(s))}ds \\
    & = 1 + e^{\theta(v-2g)}\int_1^{\infty}e^{-\frac{g\theta}{\lala}\ln(s)}\mathbb{E}^*_{(v,0,\hat{\bfz})}e^{\theta\sup_{0 \leq t \leq \frac{1}{\lala}\ln(s)}B_1(t)}ds \\
    & \leq 1 + \vrr_1e^{\theta(v-2g)}\int_1^{\infty}s^{-\frac{g\theta}{\lala}+\frac{\vrr_2\theta^2}{\lala}}ds, \\
\end{align*}
where we have once again used \eqref{eq:elemconc}.
Now let $\lala_5 \doteq g^2/(8\vrr_2)$ and for fixed $\lala \in (0, \lala_5)$, take $\theta = 4 \lala/g$.
Then, for any $\lala \in (0, \lala_5)$,
$$-\frac{g\theta}{\lala}+\frac{\vrr_2\theta^2}{\lala} { \leq} -2.$$
It then follows, for $\lala \in (0, \lala_5)$,
$$
\sup_{\hat{\bfz} \in \RR_+^{N-1}}\mathbb{E}^*_{(v, 0,\hat{\bfz})}e^{\lala \,\hat \tau_{2g}} \le 
 1 + \vrr_1e^{4 \lala(v-2g)/g}.
 $$
 The result follows. \hfill \qed
%
%
%
%
%
%
%
%
%

\subsection{Proofs of Lemmas for Proposition \ref{neginexp}.}
\label{sec:pfneginexp}
In this section we provide the proofs of Lemmas \ref{zlesslem1} and \ref{zlesslem2}  stated in Section \ref{singdrift} that were used in the proof of Proposition \ref{neginexp}.

\subsubsection{Proof of Lemma \ref{zlesslem1}.}
Fix $(2g,0, \hat\bfz)\in \RR\times \RR_+^N$. Since $M(t) = \sum_{i=2}^NB_i(t) - (N-1)B_1(t)$, from Proposition \ref{sigma2} (which implies $\mathbb{E}^*_{(2g, 0,\hat\bfz)}\sigma_{2l} < \infty$ for any $l \in \mathbb{N}$) and optional sampling theorem (cf. \cite[Section 1.3.C]{KarShre}),
 we have from Lemma \ref{zless}, for $l\in \NN$,
\begin{align}
    \mathbb{E}^*_{(2g, 0,\hat\bfz)}(\bar{Z}_2(\sigma_{2l}) - \bar{z}_2) &\leq \mathbb{E}^*_{(2g, 0,\hat\bfz)}\left(M(\sigma_{2l}) + \frac{N}{k}Y^{(1)}_k(\sigma_{2l}) - \frac{(N-k)}{k}L_1(\sigma_{2l}) + \frac{N}{2k}L_{k+1}(\sigma_{2l})\right) \nonumber\\
    & = \frac{N}{k}\mathbb{E}^*_{(2g, 0,\hat\bfz)}\,Y^{(1)}_k(\sigma_{2l}) - \frac{(N-k)}{k}\mathbb{E}^*_{(2g, 0,\hat\bfz)}\,L_1(\sigma_{2l}) + \frac{N}{2k}\mathbb{E}^*_{(2g, 0,\hat\bfz)}\,L_{k+1}(\sigma_{2l}). &\ &&\ \label{eq:340}
\end{align}
Using standard martingale maximal inequalities we have
\begin{align*}
    \mathbb{E}^*_{(2g, 0,\hat\bfz)}B_i^*(\sigma_{2l}) &\leq 
 c_0 \sqrt{\mathbb{E}^*_{(2g, 0,\hat\bfz)}\sigma_{2l}} 
     =  c_0\left(\sum_{i=1}^{l}\mathbb{E}^*_{(2g, 0,\hat\bfz)}(\sigma_{2i}-\sigma_{2(i-1)})\right)^{1/2} \le  c_0' \sqrt{l},
\end{align*}
where $ c_0, c_0'\in (0, \infty)$ are independent of $\hat\bfz$ and $l$, and the last inequality once more uses Proposition \ref{sigma2}.
Thus, for some $ c_1 \in (0,\infty)$, for all $k = 1, \ldots, N$, $l\in \NN$, 
\begin{equation}\label{eq:338}
    \sup_{\hat\bfz \in \RR_+^{N-1}}\mathbb{E}^*_{(2g,0,\bfz)}Y_k^{(1)}(\sigma_{2l}) \leq c_1 l^{1/2}.
\end{equation}
Next note that
\begin{align}
    \mathbb{E}^*_{(2g, 0,\hat\bfz)}\,L_1(\sigma_{2l}) & = \sum_{i=1}^{l}\mathbb{E}^*_{(2g, 0,\hat\bfz)}(L_1(\sigma_{2i}) - L_1(\sigma_{2i-2})) \;\;\ge l\inf_{\tilde\bfz \in \RR_+^{N-1}}  \mathbb{E}^*_{(2g, 0,\tilde\bfz)}(L_1(\sigma_2))\nonumber\\
	&\ge gl\inf_{\tilde\bfz \in \RR_+^{N-1}}  \mathbb{P}^*_{(2g,0,\tilde\bfz)}(L_1(\sigma_2) > g)\nonumber\\
	&= gl\inf_{\tilde\bfz \in \RR_+^{N-1}}  \mathbb{P}^*_{(2g,0,\tilde\bfz)}(g\sigma_2 - V(\sigma_2) + 2g  > g)
	= gl\inf_{\tilde\bfz \in \RR_+^{N-1}}  \mathbb{P}^*_{(2g,0,\tilde\bfz)}(\sigma_2  > 1) = gl,
&\ \label{eq:345}	\end{align}
	where the last equality follows on observing that, under $\mathbb{P}^*_{(2g,0,\tilde\bfz)}$, $\sigma_2> \sigma_1>1$ a.s.
The result follows from \eqref{eq:340}, \eqref{eq:338} and \eqref{eq:345}. \hfill \qed

\subsubsection{Proof of Lemma \ref{zlesslem2}.}

Fix $\Delta>0$ and $(2g,0, \hat\bfz)\in \RR\times \RR_+^N$ such that $\hat\bfz \in \cls_{\Delta}$. All inequalities will be a.s. under $\PP^*_{(2g,0,\hat\bfz)}$. Let $k=k(\Delta)$ satisfy \eqref{eq:1132}.
Define
$$\theta_k = \inf\{t \geq 0: Z_{k+1}(t) = 0\}.$$
Then for $t \leq \theta_k$, 
\begin{align}
    Z_{k+1}(t) &= z_{k+1}+B_{k+1}(t)-B_k(t) - \frac{1}{2}(L_k(t) + L_{k+2}(t)) \nonumber\\
    & \geq \frac{\Delta}{N^2}+B_{k+1}(t)-B_k(t) - \frac{1}{2}(L_k(t) + L_{k+2}(t)).\label{eq:551}
\end{align}
To bound $\mathbb{E}^*_{(2g, 0,\hat\bfz)}L_{k+1}(\sigma_{2l})$, we will obtain an upper bound on the probability that $L_{k+1}(\sigma_{2l}) > 0$, or equivalently, the probability that $Z_{k+1}(\cdot)$ hits zero before time $\sigma_{2l}$, using \eqref{eq:551}. Next, we will estimate $\mathbb{E}^*_{(2g, 0,\hat\bfz)}(L_{k+1}(\sigma_{2l})^2)$. These two will be combined using a Cauchy-Schwarz inequality to obtain an upper bound for $\mathbb{E}^*_{(2g, 0,\hat\bfz)}L_{k+1}(\sigma_{2l})$.

We will first obtain an upper bound for $L_k(t)$ for $k<N$ and $t \le \theta_k$. When $3 \le k \le N-1$, from \eqref{eq:lbds},
for $t \le \theta_k$,
\begin{align}
    L_{k}(t) &\leq B^*_k(t) + \frac{1}{2}L_{k-1}(t) \nonumber\\
    L_i(t) &\leq B^*_i(t) + \frac{1}{2}(L_{i-1}(t)+L_{i+1}(t)),\,\, 3 \leq i \leq k-1 \ \mbox{ if } \ k \ge 4, \nonumber\\
    L_2(t) &\leq L_1(t) + B^*_2(t) + \frac{1}{2}L_3(t).\label{lll}
\end{align}
Thus,
\begin{align*}
    \sum_{i=3}^{k-1}&(i-1)(L_i(t)-\frac{1}{2}(L_{i+1}(t)+L_{i-1}(t))) \\
    &+ (L_2(t)-L_1(t)-\frac{1}{2}L_3(t)) + (k-1)(L_k(t)-\frac{1}{2}L_{k-1}(t)) \leq \sum_{i=2}^k(i-1)B^*_i(t) \doteq Y^{(2)}_k(t),
\end{align*}
where the first sum is taken to be zero if $k=3$.
The left side in the above inequality equals
$
   \frac{k}{2}L_k(t) - L_1(t)
$
and so we have, whenever $N>k\ge3$, $t \leq \theta_k$,
\begin{equation}\label{eq:549}
	L_k(t) \leq \frac{2}{k}L_1(t) + \frac{2}{k}Y^{(2)}_k(t).\end{equation}
Note that the above inequality holds trivially if $k=1$, and by \eqref{lll} if $k=2$, and so in fact the above holds under $\mathbb{P}^*_{(2g, 0,\hat\bfz)}$,
with $\hat \bfz \in \cls_{\Delta}$, for $k$ satisfying \eqref{eq:1132} and $t \le \theta_k$ .

We now obtain a similar upper bound on $L_{k+2}(t)$ when $k<N$ and $t \le \theta_k$. 
From \eqref{eq:lbds},
when $k<N-1$, for $t \leq \theta_k$,
\begin{align*}
    &(N-k-1)(L_{k+2}(t)-\frac{1}{2}L_{k+3}(t)) 
     + \sum_{i=k+2}^{N-1}(N-i)(L_{i+1}(t)-\frac{1}{2}(L_{i}(t) + L_{i+2}(t))) \\
    &\leq \sum_{i=k+1}^{N-1}(N-i)B^*_{i+1}(t) \doteq Y^{(3)}_k(t).
\end{align*}
The left side equals $\frac{N-k}{2}L_{k+2}(t)$ and so we have, when $k<N-1$,
\begin{equation}\label{eq:550}
	L_{k+2}(t) \leq \frac{2}{N-k}Y^{(3)}_k(t),\,\,\mbox{ for all } t \leq \theta_k.\end{equation}
	Note that when $k=N-1$ the inequality is trivially true.
Using \eqref{eq:549} and \eqref{eq:550} in \eqref{eq:551}, we have for $t \leq \theta_k$,
under $\mathbb{P}^*_{(2g, 0,\hat\bfz)}$
\begin{align*}
	Z_{k+1}(t) 
	&\ge \frac{\Delta}{N^2} - \sum_{i=1}^N B_{i}^*(t) - \frac{2}{k}L_1(t) - \frac{2}{k}Y^{(2)}_k(t)
	-\frac{2}{N-k}Y^{(3)}_k(t)\\
	&\ge \frac{\Delta}{N^2} - Y^{(4)}_k(t) - \frac{2}{k}L_1(t) ,
\end{align*}
where
$Y^{(4)}_k(t) = \sum_{i=1}^N B_{i}^*(t) + \frac{2}{k}Y^{(2)}_k(t) + \frac{2}{N-k}Y^{(3)}_k(t)$.
Note that if $L_{k+1}(\sigma_{2l}) > 0$, then  $\inf_{0 \leq s \leq \sigma_{2l}}Z_{k+1}(s) = 0$, which in turn implies that  $\theta_k \in [0, \sigma_{2l}]$ and so from the above display
$$\frac{\Delta}{N^2} \le  Y^{(4)}_k(\theta_k) + \frac{2}{k}L_1(\theta_k) \leq  Y^{(4)}_k(\sigma_{2l}) + \frac{2}{k}L_1(\sigma_{2l}). $$
As a consequence, 
\begin{align}
    \mathbb{P}^*_{(2g, 0,\hat\bfz)}(L_{k+1}(\sigma_{2l}) > 0) & \leq \mathbb{P}^*_{(2g, 0,\hat\bfz)}(\frac{\Delta}{N^2} \leq Y^{(4)}_k(\sigma_{2l}) + \frac{2}{k}L_1(\sigma_{2l})) \\
    & \leq \mathbb{P}^*_{(2g, 0,\hat\bfz)}(\frac{\Delta}{2N^2} \leq Y^{(4)}_k(\sigma_{2l}))
     + \mathbb{P}^*_{(2g, 0,\hat\bfz)}(\frac{\Delta}{2N^2} \leq \frac{2}{k}L_1(\sigma_{2l})).\label{eq:631}
\end{align}
Consider now the first term on the right side. Then, for  $T \ge 1$,
\begin{align}
    \mathbb{P}^*_{(2g, 0,\hat\bfz)} (\frac{\Delta}{2N^2} \leq Y^{(4)}_k(\sigma_{2l})) 
    & \leq \mathbb{P}^*_{(2g, 0,\hat\bfz)}(\sigma_{2l} > T) + \mathbb{P}(\frac{\Delta}{2N^2} \leq Y^{(4)}_k(T)) \nonumber\\
    & \leq l { \sup_{\tilde \bfz \in \RR_+^{N-1}}}\mathbb{P}^*_{(2g,0, \tilde\bfz)}(\sigma_{2} > T/l) + c_1e^{-c_2\Delta^2/T} \nonumber\\
    & \leq c_3le^{-c_4T/l} +c_1e^{-c_2\Delta^2/T}, \label{eq:1020}
\end{align}
where $c_i \in (0,\infty)$ are constants  that do not depend on $\Delta$ or $\hat\bfz \in \cls_{\Delta}$, the second inequality uses the strong Markov property and a standard concentration estimate (see \eqref{eq:elemconc}) and the last inequality is a consequence of Proposition \ref{sigma2}.
 Now consider the second term on the right side of \eqref{eq:631}. For $T\ge 1$,
 \begin{align}
 \mathbb{P}^*_{(2g, 0,\hat\bfz)}(\frac{\Delta}{2N^2} \leq \frac{2}{k}L_1(\sigma_{2l})) &\le
 \mathbb{P}^*_{(2g, 0,\hat\bfz)}(\sigma_{2l} > T) + \mathbb{P}^*_{(2g, 0,\hat\bfz)}(\frac{\Delta}{2N^2} \leq \frac{2}{k}L_1(T)) \nonumber\\
 &\le c_3le^{-c_4T/l} +\mathbb{P}^*_{(2g, 0,\hat\bfz)}(L_1(T)\ge \frac{\Delta k}{4N^2}). \label{eq:1019}
 \end{align} 
 Note that under $\PP^*_{(2g, 0,\hat\bfz)}$,
$$\sup_{0 \leq s \leq T}V(s) < 2g + gT.$$
Thus, using \eqref{eql2l1} and \eqref{eq:loctimes}, for any $T\ge 1$,
\begin{align*}
	L_1(T) &\le \frac{1}{2}L_2(T) + (2g+gT)T + B_1^*(T)\le \frac{(N-1)}{N}L_1(T) + \frac{1}{N} \bar Y(T) + (2g+gT)T + B_1^*(T)
\end{align*}
and thus, with $\tilde Y(T) = \bar Y(T) + NB_1^*(T)$ and $c_5 = 3gN$
\begin{align}\label{eq:1114}
	L_1(T) &\le   \bar Y(T) + (2g+gT)TN + NB_1^*(T) { \leq} \tilde Y(T) + c_5T^2.
\end{align}

Take $T = T(\Delta) \doteq \frac{1}{2N} \sqrt{\frac{\Delta}{2c_5}}$ and observe that $c_5T^2 \le \Delta k/(8N^2)$ for all $k \in \{1,\ldots,N\}$. Choose $\Delta_1>0$ such that $T(\Delta_1) \ge 1$.
Then, it follows by a concentration estimate that, for $\Delta \ge \Delta_1$,
\begin{align}
	\mathbb{P}^*_{(2g, 0,\hat\bfz)}(L_1(T)&\ge \frac{\Delta k}{4N^2}) \le
	\mathbb{P}^*_{(2g, 0,\hat\bfz)}(\tilde Y(T)  \ge \frac{\Delta k}{4N^2} - c_5T^2)\\
	&\le \mathbb{P}^*_{(2g, 0,\hat\bfz)}(\tilde Y(T)  \ge c_5T^2) \le c_6 e^{-c_7T^3}.\label{eq:1017}
\end{align}
Then, using \eqref{eq:1017},  \eqref{eq:1019}
and \eqref{eq:1020} in \eqref{eq:631}, we obtain constants $c_2', c_7', c_8, c_9>0$ such that for all $\Delta \ge \Delta_1$,
\begin{equation}\mathbb{P}^*_{(2g, 0,\hat\bfz)}(L_{k+1}(\sigma_{2l}) > 0) \leq 
2c_3le^{-c_4\nunu \sqrt{\Delta}/l} +c_1e^{-c_2'\Delta^{3/2}} + c_6 e^{-c_7'\Delta^{3/2}}
{ \leq} 2c_3le^{-c_4\nunu \sqrt{\Delta}/l} + c_8e^{-c_9\Delta^{3/2}}.
\label{eq:1135}
\end{equation}

We will now obtain an upper bound for $\mathbb{E}_{{(2g,0,\hat \bfz)}}(L_{k+1}(\sigma_{2l})^2)$.
From \eqref{eq:loctimes} we have that, for $m\ge 2$ and $t\ge 0$,
\begin{align*}
	 &(L_N(t) - \frac{1}{2}L_{N-1}(t)) + \sum_{j=m}^{N-2} (N-j)(L_{j+1}(t) -\frac{1}{2}(L_{j+2}(t)+L_{j}(t))) \\ 
    &\leq \sum_{j=m}^{N-1}(N-j)B_{j+1}^*(t) \doteq Y^{(5)}_{m}(t).
\end{align*}
The left side above equals $\frac{N-m+1}{2} L_{m+1}(t) - \frac{N-m}{2}L_m(t) $ and so
$$\frac{N-m+1}{2} L_{m+1}(t) - \frac{N-m}{2}L_m(t) \le Y^{(5)}_{m}(t).$$
Dividing by $(N-m)(N-m+1)/2$ throughout, we have
$$\frac{1}{N-m} L_{m+1}(t)-  \frac{1}{N-m+1} L_m(t) \le \frac{2}{(N-m)(N-m+1)}Y^{(5)}_{m}(t),\;\; 2\le m \le N-1.$$
Summing over $m$ from $2$ to $k$, the above yields
$$\frac{1}{N-k}L_{k+1}(t)-\frac{1}{N-1}L_2(t) \leq \sum_{m=2}^k\frac{2Y^{(5)}_{m}(t)}{(N-m)(N-m+1)} \doteq Y^{(6)}_k(t).$$
and thus
$$L_{k+1}(t) \leq \frac{N-k}{N-1}L_2(t)+(N-k)Y^{(6)}_k(t).$$
From \eqref{eq:1055} (recall it holds for any value of $k\ge 1$) we have
$$L_{k+1}(t) \leq \frac{N-k}{N-1}\left(\frac{2(k-1)}{k}L_1(t) + \frac{2}{k}Y^{(1)}_k(t) + \frac{1}{k}L_{k+1}(t)\right)+(N-k)Y^{(6)}_k(t).$$
Thus
$$\frac{N(k-1)}{k(N-1)} L_{k+1}(t) \le \frac{2(k-1)(N-k)}{k(N-1)} L_1(t) + \frac{2(N-k)}{k(N-1)}Y^{(1)}_k(t)
+ (N-k)Y^{(6)}_k(t)$$
and consequently, when $k>1$,
\begin{align}
    L_{k+1}(t) &\leq \frac{2(N-k)}{N}L_1(t)+\frac{2(N-k)}{N(k-1)}Y^{(1)}_k(t)+\frac{k(N-k)(N-1)}{N(k-1)}Y^{(6)}_k(t) \\
    & = \frac{2(N-k)}{N}L_1(t)+Y^{(7)}_k(t),\label{eq:1122}
\end{align}
where $Y_k^{(7)}(t) = \frac{2(N-k)}{N(k-1)}Y^{(1)}_k(t)+\frac{k(N-k)(N-1)}{N(k-1)}Y^{(6)}_k(t)$.
Recalling the inequality \eqref{eql2l1} we see that \eqref{eq:1122} also holds with $k=1$
and $Y^{(7)}_1(t) \doteq \frac{2}{N} \bar Y(t)$.

Using this, we obtain that
\begin{align*}
    \mathbb{E}^*_{(2g, 0,\hat\bfz)}(L_{k+1}(\sigma_{2l})^2) & \le  2\mathbb{E}^*_{(2g, 0,\hat\bfz)}(Y^{(7)}_k(\sigma_{2l}))^2 + \frac{8(N-k)^2}{N^2}\mathbb{E}^*_{(2g, 0,\hat\bfz)}L_1^2(\sigma_{2l})\\
	& \le 2\mathbb{E}^*_{(2g, 0,\hat\bfz)}(Y^{(7)}_k(\sigma_{2l}))^2
	+ \frac{16(N-k)^2}{N^2}\mathbb{E}^*_{(2g, 0,\hat\bfz)}(\tilde Y(\sigma_{2l}))^2 + \frac{16(N-k)^2c_5^2}{N^2}\mathbb{E}^*_{(2g, 0,\hat\bfz)}(\sigma_{2l})^4.
\end{align*}
where the last line is from \eqref{eq:1114}.
Thus, using the above bound, the strong Markov property, and Proposition \ref{sigma2}, there is a $b_1 \in (0,\infty)$ such that, for all $l \in \NN$, $\hat \bfz \in \cls_\Delta$, and $k$ satisfying \eqref{eq:1132},
\begin{align*}
    \mathbb{E}^*_{(2g, 0,\hat\bfz)}(L_{k+1}(\sigma_{2l})^2) \le b_1 l^5.
\end{align*}
 
Applying Cauchy-Schwarz inequality and using \eqref{eq:1135}, we obtain positive constants $D_1, D_2, D_3$ such that for all $\Delta \ge \Delta_1$,
\begin{align*}
    \mathbb{E}^*_{(2g, 0,\hat\bfz)}L_{k+1}(\sigma_{2l}) & \leq (\mathbb{E}^*_{(2g, 0,\hat\bfz)}L_{k+1}(\sigma_{2l})^2)^{1/2}(\mathbb{P}^*_{(2g, 0,\hat\bfz)}(L_{k+1}(\sigma_{2l}) > 0))^{1/2} \\
    & \leq b_1^{1/2} l^{5/2} (2c_3le^{-c_4\nunu \sqrt{\Delta}/l} + c_8e^{-c_9\Delta^{3/2}})^{1/2}=
	D_1l^{5/2}(\sqrt{l}e^{-D_2\sqrt{\Delta}/l}+e^{-D_3\Delta^{3/2}}),
\end{align*}
{ as desired.} \hfill \qed

\subsection{Proof of Proposition \ref{driftcondn}}
\label{sec:driftcon}
In this section we give the proof of Proposition \ref{driftcondn}. Proof relies on five preliminary lemmas which extend some estimates derived in Sections \ref{highlev} and \ref{sec:lowlev} to more general starting configurations. The first four are required to verify part (1) of the proposition and the last one is used to check part (2). Proof of the proposition is at the end of the section.
Recall the set $C^*$ from \eqref{eq:cstar} and stopping times $\sigma_1, \sigma_2$ from Section \ref{sec:lowlev}.
Recall $\Gamma$ from \eqref{eq928}.
\begin{lemma}
	\label{lem:800}
There exists a $\rho_0 >0$ and $b_1, b_2 >0$ such that, for all $\rho \in (0, \rho_0)$, there is a   $b_3(\rho) \in (0,\infty)$ such that for any
$(v,\bfz) \in \RR \times \RR_+^N$,
\begin{equation*}
\mathbb{E}^*_{(v,\bfz)}e^{\rho\Gamma} \leq b_3(\rho)e^{b_1\rho(|v| + z_1 + \bar z_2)}\mathbb{E}^*_{(v,\bfz)}e^{b_2\rho\sigma_1}.
\end{equation*}
\end{lemma}
\begin{proof}
Define the stopping time
$$\sigma^* \doteq \inf\{t\ge \sigma_2: (V(t), \bfZ(t)) \in C^*\}.$$
From Proposition \ref{gammabound}, there exist positive constants $d_0, c'$ such that for any $\lala \in (0, \lala_7/2)$, where $\lala_7$ is as in that lemma,
\begin{equation}
	\mathbb{E}^*_{(2g,0,\hat\bfz)} e^{\lala \Gamma} \le d_0e^{c'\gamma \bar z_2},  \ \ \hat \bfz \in \RR_+^{N-1}.
\end{equation}
Fix $\rho_0'>0$ such that
$ \rho_0' < \min\{\lala_7, \lala_5, \lala_4\}$ and 
$\rho_0' (1+\kappa_2' g) < \lala_4$,
where $\lala_5$ and $\kappa_2'$  are as in Lemma \ref{sec2prop3} and $\lala_4$ is as in Lemma \ref{sec2prop2}.
Then, for $(v,\bfz) \in \RR \times \RR_+^N$ and $\rho \in (0, \rho_0'/2)$,
\begin{align}
\mathbb{E}^*_{(v,\bfz)}	e^{\rho\Gamma}  \le \mathbb{E}^*_{(v,\bfz)}e^{\rho\sigma^*}
&= \mathbb{E}^*_{(v,\bfz)}\left[ \mathbb{E}^*_{(v,\bfz)}\left[e^{\rho\sigma^*}\mid \clf^*_{\sigma_2}\right]\right] \le d_0 \mathbb{E}^*_{(v,\bfz)} e^{\rho\sigma_2 + c'\rho \bar Z_2(\sigma_2)}\notag\\
& \le d_0\left(\mathbb{E}^*_{(v,\bfz)} e^{2\rho\sigma_2}\right)^{1/2}\left(\mathbb{E}^*_{(v,\bfz)} e^{2c'\rho \bar Z_2(\sigma_2)}\right)^{1/2}.\label{eq:649}
\end{align}
Recall the stopping time $\eta_1 \ge \sigma_1$ defined in \eqref{eq:etastop}. Proceeding as in the proof of Proposition \ref{sigma2} (see \eqref{eq:541}), with $d_1 = 1 + \kappa_1' e^{8\rho \kappa_2' g}$ and $d_2 = (1 + \kappa_2' g)$,
 using Lemma \ref{sec2prop2} and Lemma \ref{sec2prop3},
\begin{align}
\mathbb{E}^*_{(v,\bfz)} e^{2\rho\sigma_2} &\le  \mathbb{E}^*_{(v,\bfz)}\left[\textbf{1}_{\eta_1 <\sigma_2} \mathbb{E}^*_{(v,\bfz)}\left[e^{2\rho\sigma_2}\mid \clf^*_{\eta_1}\right]\right]
+ \mathbb{E}^*_{(v,\bfz)} e^{2\rho\eta_1}\\
&\le  \kappa_1' \mathbb{E}^*_{(v,\bfz)}e^{2\rho \kappa_2' V(\eta_1) + 2\rho \eta_1} + \mathbb{E}^*_{(v,\bfz)} e^{2\rho\eta_1}\\
&\le d_1 \mathbb{E}^*_{(v,\bfz)}e^{2\rho d_2 \eta_1}\\
&= d_1 \mathbb{E}^*_{(v,\bfz)}\left[e^{2\rho d_2 \sigma_1}\mathbb{E}^*_{(v,\bfz)}\left[e^{2\rho d_2 (\eta_1- \sigma_1)} \mid \clf^*_{\sigma_1} \right]\right]\\
&\le \kappa_1 d_1 \mathbb{E}^*_{(v,\bfz)}\left[e^{2\rho d_2 \sigma_1}e^{2\rho d_2\kappa_2 Z_1(\sigma_1)^{1/2}}\right]\\
& \le \kappa_1 d_1\left(\mathbb{E}^*_{(v,\bfz)}e^{4\rho d_2\kappa_2 Z_1(\sigma_1)^{1/2}}\right)^{1/2}
	\left(\mathbb{E}^*_{(v,\bfz)}e^{ 4\rho d_2\sigma_1}\right)^{1/2},
\label{eq:641a}
\end{align}
where we used Lemma \ref{sec2prop3} for the second inequality, $V(\eta_1) \le 4g + g\eta_1$ in the third inequality, and Lemma \ref{sec2prop2} in the penultimate inequality (and the observation that $\rho'_0 d_2 < \gamma_4$).
Now we estimate exponential moments of $Z_1(\sigma_1)^{1/2}$. Note that, under $\mathbb{P}^*_{(v,\bfz)}$,
$L_1(\sigma_1) = g\sigma_1+ v-4g$, from which it follows that
$$Z_1(\sigma_1) \le \sup_{0\le s \le \sigma_1} B_1(s) + z_1 + g \sigma_1^2 + g\sigma_1 + 2|v|\sigma_1 + |v|$$ and so, using $\sqrt{a + b} \le \sqrt{a} + \sqrt{b}$ and $\sqrt{ab} \le (a+b)/2$ for $a,b \ge 0$,
\begin{align*}
Z_1(\sigma_1)^{1/2} \le \left(\sup_{0\le s \le \sigma_1} B_1(s)\right)^{1/2}  + (\sqrt{g} +3/2)\sigma_1 + \frac{1}{2}(z_1 + 3|v| + g + 1).
\end{align*}
Using this bound in \eqref{eq:641a} and Young's inequality, we obtain a finite positive constant $d_3$ not depending on $v, \bfz, \rho$ such that
\begin{multline}
\mathbb{E}^*_{(v,\bfz)} e^{2\rho\sigma_2}\\
 \le d_3 e^{2\rho d_2 \kappa_2(z_1 + 3|v|)}\left[\left(\mathbb{E}^*_{(v,\bfz)} e^{8\rho d_2\kappa_2\left(\sup_{0\le s \le \sigma_1} B_1(s)\right)^{1/2}}\right)^{1/2} + \left(\mathbb{E}^*_{(v,\bfz)}e^{8\rho d_2\kappa_2(\sqrt{g} +3/2)\sigma_1}\right)^{1/2}\right]\left(\mathbb{E}^*_{(v,\bfz)}e^{4\rho d_2\sigma_1}\right)^{1/2}.\\\label{i1}
\end{multline}
The expectation involving $\sup_{0\le s \le \sigma_1} B_1(s)$ above is bounded as in the proof of Lemma \ref{sec2prop1} (see \eqref{eq:907}) to obtain  $\rho_0'', d_4,d_5 \in (0,\infty)$ such that $d_5\rho_0''<d_2$, and
for any $\rho \in (0, \rho_0'')$ and  $(v,\bfz) \in \RR \times \RR_+^N$,
\begin{align}\label{stopsup}
    \mathbb{E}^*_{(v,\bfz)}e^{8\rho d_2\kappa_2(\sup_{0\leq s \leq \sigma_1}B_1(s))^{1/2}} 
    & \leq e^{8\rho d_2\kappa_2} + d_4 e^{d_5\rho^2}\left(\mathbb{E}^*_{(v,\bfz)}e^{4\rho d_2\sigma_1}\right)^{1/2} \sum_{k=0}^{\infty}e^{d_5\rho^2k - 2\rho d_2 k}\notag\\
    &\le e^{8\rho d_2\kappa_2} + d_6(\rho) \left(\mathbb{E}^*_{(v,\bfz)}e^{4\rho d_2\sigma_1}\right)^{1/2},
\end{align}
where $d_6(\rho) \doteq d_4 e^{d_5\rho^2}(1-e^{-\rho d_2})^{-1}$.
%
Using this bound in \eqref{i1}, we conclude that for every $0 < \rho < \min\{\rho_0'/2, \rho_0''\}$, there exists a finite positive constant $d_7(\rho)$ satisfying
\begin{equation}\label{ff0}
\mathbb{E}^*_{(v,\bfz)} e^{2\rho\sigma_2} \le e^{2\rho d_2 \kappa_2(z_1 + 3|v|)} d_7(\rho)\mathbb{E}^*_{(v,\bfz)}e^{d_2'\rho\sigma_1},
\end{equation}
where $d_2' = \max\{4d_2, 8d_2\kappa_2(\sqrt{g} + 3/2)\}$.

Now, we estimate $\mathbb{E}^*_{(v,\bfz)} e^{2c'\rho \bar Z_2(\sigma_2)}$. From \eqref{eq:semmart} and \eqref{eql2l1}, $\bar Z_2(t) \le \bar z_2 + M(t) + \bar Y(t), t \ge 0$. Hence, writing $\tilde{Y}(t) \doteq M(t) + \bar Y(t)$,
$$
\mathbb{E}^*_{(v,\bfz)} e^{2c'\rho \bar Z_2(\sigma_2)} \le e^{2c'\rho\bar z_2}\mathbb{E}^*_{(v,\bfz)} e^{2c'\rho \tilde{Y}(\sigma_2)}. 
$$
Proceeding exactly as in \eqref{stopsup}, we obtain $\rho_0'''>0$ such that for every $\rho \in (0, \rho_0''')$, there exists a  $d_8(\rho) \in (0,\infty)$ such that
$$
\mathbb{E}^*_{(v,\bfz)} e^{2c'\rho \tilde{Y}(\sigma_2)} \le \sum_{k=0}^{\infty}\left(\mathbb{E}^*_{(v,\bfz)} e^{4c'\rho\sup_{0 \le s \le k+1} \tilde{Y}(s)}\right)^{1/2}(\mathbb{P}^*_{(v,\bfz)}(\sigma_2 \geq k))^{1/2} \le d_8(\rho)\left(\mathbb{E}^*_{(v,\bfz)}e^{2\rho \sigma_2}\right)^{1/2},
$$
which, along with \eqref{ff0}, gives
\begin{equation}\label{ff1}
\mathbb{E}^*_{(v,\bfz)} e^{2c'\rho \bar Z_2(\sigma_2)} \le e^{2c'\rho\bar z_2}d_8(\rho)\left(\mathbb{E}^*_{(v,\bfz)}e^{2\rho \sigma_2}\right)^{1/2} \le e^{2c'\rho\bar z_2}e^{\rho d_2 \kappa_2(z_1 + 3|v|)} d_8(\rho) d_7(\rho)^{1/2}\left(\mathbb{E}^*_{(v,\bfz)}e^{d_2'\rho\sigma_1}\right)^{1/2}.
\end{equation}
The result, with $\rho_0 \doteq \min\{\rho_0'/2, \rho_0'', \rho_0'''\}$, now follows upon using \eqref{ff0} and \eqref{ff1} in \eqref{eq:649}. 
\end{proof}
\begin{lemma}\label{sec5prop2}  Let $\vsig \doteq \hat \tau_{4g}\wedge \hat \tau_0.$
	Then there is a 
 $\beta_1 > 0$ and $D_1 > 0$ such that, for all $(v,\bfz) \in \RR\times \RR_+^N$,
$$\mathbb{E}^*_{(v,\bfz)}e^{\beta_1\vsig} \leq D_1 e^{\beta_1(|v|+z_1)}.$$
\end{lemma}

\begin{proof}
	Fix $(v,\bfz) \in \RR\times \RR_+^N$. We consider three cases.\\

\noindent {\bf Case 1: $v \in [0,4g]$.}	
In this case the result is immediate from Lemma \ref{sec1prop3}.

\noindent {\bf Case 2: $v>4g$.}
In this case, for all $t \leq \vsig$, $V(t) > 4g$. Thus, for such $t$, we have
\begin{align*}
    Z_1(t) & = z_1 + B_1(t) + L_1(t) - \frac{1}{2}L_2(t) - \int_0^tV(s)ds 
     \leq z_1 + \sup_{0 \leq s \leq t}B_1(s) + L_1(t) - 4gt.
\end{align*}
Consequently,
$- L_1(t) \leq z_1 + \sup_{0 \leq s \leq t}B_1(s) - 4gt$.
Thus we have, for $t \leq \vsig$,
$$V(t) = gt - L_1(t) + v \leq z_1 + \sup_{0 \leq s \leq t}B_1(s) - 3gt + v \doteq Q_1(t).$$
Letting $\tau^{Q_1}_{4g} \doteq  \inf\{t\ge 0: Q_1(t)= 4g\}$, we have
$\tau^{Q_1}_{4g} \geq \vsig$. Thus, for $\beta_1, \theta>0$,
\begin{align*}
    \mathbb{E}^*_{(v,\bfz)}e^{\beta_1\vsig} & \leq 1 + \int_1^{\infty}\mathbb{P}^*_{(v,\bfz)}(Q_1(\frac{1}{\beta_1}\ln(s)) > 4g)ds \\
    & \leq 1 + e^{-4g\theta\beta_1}\int_1^{\infty}\mathbb{E}^*_{(v,\bfz)}e^{\theta\beta_1 Q_1(\frac{1}{\beta_1}\ln(s))}ds \\
    & = 1 + e^{-4g\theta\beta_1 + \theta\beta_1 (z_1+v)}\int_1^{\infty}e^{-3\theta g\ln(s)}\mathbb{E}^*_{(v,\bfz)}e^{\theta\beta_1\sup_{0 \leq t \leq \frac{1}{\beta_1}\ln(s)}B_1(t)}ds \\
    & \leq 1 + \vrr_1e^{-4g\theta\beta_1 + \theta\beta_1 (z_1+v)}\int_1^{\infty}e^{-3\theta g\ln(s)+\vrr_2\theta^2\beta_1\ln(s)}ds
\end{align*}
where the last line uses the estimate \eqref{eq:elemconc}.
Taking $\theta = g^{-1}$ and $\beta_1 = g^2/\vrr_2$, we now see that
$$\mathbb{E}^*_{(v,\bfz)}e^{\beta_1\vsig}  \leq 1 +\vrr_1e^{-4g\theta\beta_1 + \theta\beta_1 (z_1+v)}$$
which completes the proof for Case 2.\\

\noindent {\bf Case 3: $v<0$.}
In this case,  for $t \leq \vsig$,  we have $V(t) < 0$. Thus for such $t$, from \eqref{eq:loctimes} and \eqref{eql2l1},
\begin{align*}
    L_1(t) & = \sup_{0 \leq s \leq t}(-z_1 - B_1(s) + \frac{1}{2}L_2(s) + \int_0^sV(u)\,du)^+ 
     \leq B_1^*(t) + \frac{1}{2}L_2(t) \le B_1^*(t) + \frac{N-1}{N} L_1(t) + \frac{1}{N} \bar Y(t).
\end{align*}
Consequently,
$$ L_1(t) \le NB_1^*(t) + \bar Y(t)$$
and so
\begin{align*}
    V(t) & = gt - L_1(t) + v 
     \geq gt - NB_1^*(t) - \bar Y(t) + v \doteq Q_2(t).
\end{align*}
Letting, $\tau_0^{Q_2} \doteq \inf\{t\ge 0: Q_2(t) =0\}$, we then have,
 $\tau_0^{Q_2} \geq \vsig$. Thus, for $\theta, \beta_1 >0$,
\begin{align*}
    \mathbb{E}^*_{(v,\bfz)}e^{\beta_1\vsig} & \leq 1 + \int_1^\infty \mathbb{P}^*_{(v,\bfz)}(Q_2(\frac{\ln(s)}{\beta_1}) < 0)\,ds \\
    & = 1 + \int_1^\infty \mathbb{P}^*_{(v,\bfz)}(g\frac{\ln(s)}{\beta_1} + v < NB_1^*(\frac{\ln(s)}{\beta_1}) + \bar Y(\frac{\ln(s)}{\beta_1}))\,ds \\
    & \leq  1 + e^{-\theta\beta_1 v}\int_1^\infty s^{-\theta g}\mathbb{E}^*_{(v,\bfz)}e^{\theta \beta_1(NB_1^*(\frac{\ln(s)}{\beta_1}) + \bar Y(\frac{\ln(s)}{\beta_1}))}\,ds \\
    & \leq 1 + \vrr_1e^{\theta\beta_1 |v|}\int_1^\infty s^{-\theta g}s^{\vrr_2 \beta_1 \theta^2}\,ds,
\end{align*}
where in the last line we have used the estimate \eqref{eq:elemconc}.
Take $\theta = 4g^{-1}$ and $\beta_1 = g^2/(8\vrr_2)$, then
$$ \mathbb{E}^*_{(v,\bfz)}e^{\beta_1\vsig}  \leq 1 + \vrr_1e^{\theta\beta_1 |v|}.$$
This completes the proof for Case 3 and thus the result follows.
\end{proof}

\begin{lemma}\label{sigma1est} There is a $\beta_2 > 0$ and $\kappa_1, \kappa_2>0$ such that, for all $(v,\bfz)\in \RR\times \RR_+^N$,
$$\mathbb{E}^*_{(v,\bfz)}\,e^{\beta_2\sigma_1} \leq \kappa_1 e^{\kappa_2(|v| + z_1)}.$$
\end{lemma}

\begin{proof}
From Proposition \ref{sigma1}, with $\gamma$ as in that proposition,
\begin{equation}\label{eq:eeq217}
	\sup_{\bfz \in \RR_+^N} \mathbb{E}^*_{(\frac{g}{2N},\bfz)}e^{\gamma \hat \tau_{4g}} \doteq d_1<\infty.
\end{equation}
Also, from Lemma \ref{sec1prop1}, with $\beta$ as in that lemma,
\begin{equation}\label{eq:222}
	\sup_{{\bfz} \in \RR_+^{N}}\,\mathbb{E}^*_{(0,{\bfz})}\,e^{\beta\,\hat\tau_{g/(2N)}} \doteq d_2< \infty.
\end{equation}
With $\beta_1$ as in Lemma \ref{sec5prop2},  let  $\beta_2 \in (0, \min\{\gamma, \beta, \beta_1\})$.
Recall the stopping time $\vsig$ from Lemma \ref{sec5prop2}. Define stopping times
$$\vsig_1 \doteq \inf\{t\ge \vsig: V(t) = g/(2N)\}, \;\; \vsig_2 \doteq \inf\{t\ge \vsig_1: V(t) = 4g\}.$$
Then, for $(v,\bfz)\in \RR\times \RR_+^N$,
\begin{align*}
\mathbb{E}^*_{(v,\bfz)}\,e^{\beta_2\sigma_1} \leq \mathbb{E}^*_{(v,\bfz)}\left[\textbf{1}_{\{\sigma_1=\vsig\}}\,e^{\beta_2\sigma_1}\right]
+ \mathbb{E}^*_{(v,\bfz)}\left[\textbf{1}_{\{\sigma_1 >\vsig\}}\,e^{\beta_2\sigma_1}\right].
\end{align*}
From Lemma \ref{sec5prop2}, with $\beta_1$ and $D_1$ as in the lemma, 
$$\mathbb{E}^*_{(v,\bfz)}\left[\textbf{1}_{\{\sigma_1=\vsig\}}\,e^{\beta_2\sigma_1}\right]
\le \mathbb{E}^*_{(v,\bfz)}\left[\,e^{\beta_2\vsig}\right]
\le D_1 e^{\beta_1(|v|+z_1)}.$$
Next, from \eqref{eq:eeq217},
\begin{align*}
\mathbb{E}^*_{(v,\bfz)}\left[\textbf{1}_{\{\sigma_1 >\vsig\}}\,e^{\beta_2\sigma_1}\right]
&\le \mathbb{E}^*_{(v,\bfz)}\left[\textbf{1}_{\{\sigma_1 >\vsig\}}\,e^{\beta_2\vsig_2}\right]
= 	\mathbb{E}^*_{(v,\bfz)}\left[\textbf{1}_{\{\sigma_1 >\vsig\}}\,\mathbb{E}^*_{(v,\bfz)}\left[e^{\beta_2\vsig_2}\mid \clf^*_{\vsig_1}\right]\right]\\
&\le d_1 \mathbb{E}^*_{(v,\bfz)}\left[\textbf{1}_{\{\sigma_1 >\vsig\}}e^{\beta_2\vsig_1}\right].
\end{align*}
Also, from \eqref{eq:222},
\begin{align*}
\mathbb{E}^*_{(v,\bfz)}\left[\textbf{1}_{\{\sigma_1 >\vsig\}}e^{\beta_2\vsig_1}\right] 
&= \mathbb{E}^*_{(v,\bfz)}\left[\textbf{1}_{\{\sigma_1 >\vsig\}} \mathbb{E}^*_{(v,\bfz)}\left[e^{\beta_2\vsig_1}\mid \clf^*_{\vsig}\right]\right] \\
&\le d_2 \mathbb{E}^*_{(v,\bfz)}\left[\textbf{1}_{\{\sigma_1 >\vsig\}}e^{\beta_2\vsig}\right] 
\le d_2 D_1 e^{\beta_1(|v|+z_1)},
\end{align*}
where the last line is from Lemma \ref{sec5prop2}. Combining the above estimates, for all $(v,\bfz)\in \RR\times \RR_+^N$,
$$\mathbb{E}^*_{(v,\bfz)}\,e^{\beta_2\sigma_1}  \le D_1 e^{\beta_1(|v|+z_1)}
+ d_1 d_2 D_1 e^{\beta_1(|v|+z_1)}.$$
The result follows.
\end{proof}

\begin{lemma}\label{sec5prop3} There is a $\kappa \in (0, \infty)$ such that for every $\alpha > 0$ there is a $s_\alpha > 0$ with
$$\mathbb{E}^*_{(v,\bfz)}e^{\alpha(|V(1)|+Z_1(1) + \bar Z_2(1))} \leq s_\alpha e^{\kappa\alpha(|v|+z_1 + \bar z_2)}, \mbox{ for all } (v,\bfz) \in \RR \times \RR_+^N.$$
\end{lemma}

\begin{proof}
Since $V(t) \leq g + |v|$ for $t \leq 1$, we have from \eqref{LTineq}
$$L_1(1) \le (g+|v|) W_{11} + \sum_{i = 1}^N W_{1,i}B^*_i(1) = N(g+|v|)+ \sum_{i = 1}^N W_{1,i}B^*_i(1).$$
Thus, under $\mathbb{P}^*_{(v,\bfz)}$,
\begin{align*}
    Z_1(1) + |V(1)| & \leq z_1 + B_1(1) + L_1(1) - \frac{1}{2}L_2(1) - \int_0^1V(s)ds + g + L_1(1) + |v|
	 \\
	 &\le 2|v| + \frac{g}{2} + z_1 + \sup_{0 \leq s \leq 1}B_1(s)  + 3 L_1(1)\\
    & \leq 2|v| + \frac{g}{2} + z_1  + \sup_{0 \leq s \leq 1}B_1(s) + 
	3(N(g+|v|)+ \sum_{i = 1}^N W_{1,i}B^*_i(1)).
\end{align*}
Moreover, from \eqref{eq:semmart} and \eqref{eql2l1}, 
$\bar Z_2(1) \le \bar z_2 + M(1) + \bar Y(1), \ t \ge 0.$
The result is now immediate from the estimate in \eqref{eq:elemconc}.
\end{proof}

For $K \in \NN$, with $K>256/g$, define
$$R_K \doteq [\frac{1}{K}, \frac{g}{128}- \frac{1}{K}] \times [\frac{1}{K}, K] \times [0, K]^{N-1}.$$
\begin{lemma}
	\label{lem:lem5}
	There is a $K\in \NN$, $K>256/g$, such that 
	$$\inf_{(v,\bfz) \in C^*} \mathbb{P}^1((v,\bfz), R_K) \doteq c_K >0.$$
\end{lemma}
\begin{proof}
	Suppose that, for every $K \in \NN$, $K>256/g$,
	$$\inf_{(v,\bfz) \in C^*} \mathbb{P}^1((v,\bfz), R_K)  =0.$$
	Then, we can find a sequence $\{(v_K, \bfz_K)\}_{K \in \NN} \subset C^*$ such that
	\begin{equation}\label{eq:gtz}
		\PP^1((v_K, \bfz_K), R_K) \le \frac{1}{K}.
	\end{equation}
	Since $C^*$ is compact, we can find $(v^*, \bfz^*) \in C^*$ such that, along a subsequence (labeled again with $K$), $(v_K, \bfz_K) \to (v^*, \bfz^*)$.
	From the second statement in Lemma \ref{uniqlem},
	$$\PP^1((v^*, \bfz^*), R) >0.$$
	Since $R_K$ increase to $R$ as $K\to \infty$, we can find a $K^* \in \NN$, $K^* > 256/g$, such that
	$$\PP^1((v^*, \bfz^*), R_{K^*})  \doteq a_{K^*}>0.$$
	Choose a real, continuous function $f: \RR \times \RR_+^N $ such that $0\le f \le 1$, $f=1$ on $R_{K^*}$ and $f = 0$ on $R_{2K^*}^c$.
	Then
	\begin{align*}
		\liminf_{K\to \infty}\PP^1((v_K, \bfz_K), R_{2K^*})   &\ge 
		\liminf_{K\to \infty} \int f(v, \bfz)\PP^1((v_K, \bfz_K),  (dv, d \bfz))\\
		&= \int f(v, \bfz)\PP^1((v^*, \bfz^*),  (dv, d\bfz))
		\ge \PP^1((v^*, \bfz^*), R_{K^*}) = a_{K^*}>0,
		\end{align*}
		where the middle equality follows from the Feller property of the transition probability kernel $\PP^1$. The Feller property can be verified by analyzing two copies of the process \eqref{eq:gapproc} starting from different initial conditions but driven by the same Brownian motion. Using the Lipschitz property of the Skorohod map and Gr\"onwall's lemma, the distance between the coupled processes in sup-norm on any given compact time interval can be made small (in a pathwise sense) if the initial conditions are close enough.
		
	On the other hand, from \eqref{eq:gtz}
	$$\limsup_{K\to \infty}\PP^1((v_K, \bfz_K), R_{2K^*}) \le \limsup_{K\to \infty}\PP^1((v_K, \bfz_K), R_{K}) = 0.$$
	This is a contradiction which completes the proof of the lemma.
\end{proof}

We can now complete the proof of Proposition \ref{driftcondn}.\\

\noindent \textbf{Proof of Proposition \ref{driftcondn}}. 
Fix $\eta>0$ such that $\eta < \rho_0$ and $b_2\eta \le \beta_2$, where $\rho_0$ and $b_2$ are as Lemma \ref{lem:800} and $\beta_2$ is as in Lemma \ref{sigma1est}.
Combining Lemmas \ref{lem:800} and \ref{sigma1est}, for all $(v, \bfz)\in \RR\times \RR_+^{N}$,
$$\mathbb{E}^*_{(v,\bfz)}e^{\eta\Gamma} \leq b_3(\eta)e^{b_1 \eta(|v| + z_1 + \bar z_2)}\mathbb{E}^*_{(v,\bfz)}e^{b_2\eta\sigma_1}
\le b_3(\eta)\kappa_1 e^{b_1 \eta(|v| + z_1 + \bar z_2) + \kappa_2(|v| + z_1)}.$$
Thus
$$\mathbb{E}^*_{(v,\bfz)}e^{\eta\tau_{C^*}(1)}  = \mathbb{E}^*_{(v,\bfz)}\left[
\mathbb{E}^*_{(v,\bfz)}\left[e^{\eta\tau_{C^*}(1)} \mid \clf^*_1\right]\right]
\le b_3(\eta)\kappa_1 e^{\eta}\mathbb{E}^*_{(v,\bfz)}e^{(b_1\eta + \kappa_2)(|V(1)|+Z_1(1) + \bar Z_2(1))}.
 $$
Consequently, with $\alpha = b_1\eta + \kappa_2$, for all $(v, \bfz)\in \RR\times \RR_+^{N}$,
$$\mathbb{E}^*_{(v,\bfz)}e^{\eta\tau_{C^*}(1)} \le 
b_3(\eta)\kappa_1 s_{\alpha} e^{\eta} e^{\kappa \alpha (|v|+z_1 + \bar z_2)},$$
where $s_{\alpha}$ and $\kappa$ are as in Lemma \ref{sec5prop3}.
This immediately implies part (1) of the proposition. 

We now consider part (2).  
Let $K$ be as in Lemma \ref{lem:lem5}.
From Theorem \ref{minorization}, with $M_1\doteq \inf_{(v, \bfz) \in R_K} K_{(v,\bfz)}$ (which is positive)
$$\inf_{(v, \bfz) \in R_K} \mathbb{P}^{\sn}(( v,\bfz), B ) \ge
 \lambda(B\cap D)\inf_{(v, \bfz) \in R_K} K_{(v,\bfz)}  = M_1 \lambda(B\cap D).$$
Also, from Lemma \ref{lem:lem5}, with $K$ as in the lemma, for $(v,\bfz) \in C^*$  and $B \in \clb(\RR \times \RR_+^N)$,
\begin{align*}
	\mathbb{P}^{1+\sn}((v,\bfz), B) &\ge
	\int_{R_K} \mathbb{P}^{1}((v,\bfz), (d\tilde v, d\tilde \bfz)) 
	\mathbb{P}^{\sn}((\tilde v,\tilde \bfz), B) \\
	&\ge  M_1 \lambda(B\cap D)  \mathbb{P}^{1}((v,\bfz), R_K) 
	\ge M_1 c_K\lambda(B\cap D).
\end{align*}
The result now follows on taking $\nu(\cdot) =M_1 c_K\lambda(\cdot\cap D)$ and $r_1 = 1 + \sn$.
\qedsymbol\\

\section{Law of large numbers}\label{sec:lln}

In this section, we prove Theorem \ref{lln}. We begin with the following lemma.

\begin{lemma}\label{fl}
For any $(v,\bfz) \in \RR \times \RR_+^N$, $\mathbb{P}^*_{(v,\bfz)}$-almost every $\omega$, there exists a  $m^*(\omega) \in (0, \infty)$ such that
$$
|V(t, \omega)| + \sum_{i=1}^N Z_i(t,\omega) \le m^*(\omega) (\log t)^2, \ \mbox{ for all } t \ge 2.
$$
\end{lemma}

\begin{proof}
Recall the set $C^*$ from \eqref{eq:cstar} and the stopping time $\tau_{C^*}(1)$ defined just after. Let $\Sigma(t) \doteq |V(t)| + \sum_{i=1}^N Z_i(t), t \ge 0$. We will first show that there exist positive constants $c_1, c_2$ such that
\begin{equation}\label{lln1}
\sup_{(v,\bfz) \in C^*} \mathbb{P}^*_{(v,\bfz)}\left(\sup_{t \le \tau_{C^*}(1)} \Sigma(t) \ge x\right) \le c_1e^{-c_2\sqrt{x}}, \ x >0.
\end{equation}
Note that, from \eqref{eq:gapproc} and Lemma \ref{locin}, for $(v,\bfz) \in C^*$ and $t \ge 0$,
\begin{align*}
\Sigma(t) &= |{ v}| + \sum_{i=1}^N { z_i} + B_N(t) + g t + L_1(t) + \frac{1}{2}L_N(t)\\ 
&\le 2g + \Delta + B_N(t) + gt + \sum_{i=1}^NW_{i,1}t\sup_{0 \leq s \leq t}(V(s))^+ +  \sum_{i,j=1}^NW_{i,j}B_j^*(t)\\
& \le 2g + \Delta + gt + 2gt\sum_{i=1}^NW_{i,1} +  gt^2\sum_{i=1}^NW_{i,1} + B_N(t) +  \sum_{i,j=1}^NW_{i,j}B_j^*(t).
\end{align*}
Using this bound along with part (1) of Proposition \ref{driftcondn} and \eqref{eq:elemconc}, we obtain positive constants $c_1, c_2, x_0$ such that, for any $(v,\bfz) \in C^*, x \ge x_0$, and choosing $\delta >0$ sufficiently small,
\begin{multline*}
\mathbb{P}^*_{(v,\bfz)}\left(\sup_{t \le \tau_{C^*}(1)} \Sigma(t) \ge x\right) \le \mathbb{P}^*_{(v,\bfz)}(\tau_{C^*}(1) \ge \delta \sqrt{x}) + \mathbb{P}^*_{(v,\bfz)}\left(\sup_{t \le \delta \sqrt{x}} \Sigma(t) \ge x\right)\\
\le \mathbb{P}^*_{(v,\bfz)}(\tau_{C^*}(1) \ge \delta \sqrt{x}) +  \mathbb{P}^*_{(v,\bfz)}\left(\sup_{t \le \delta \sqrt{x}} \left(B_N(t) +  \sum_{i,j=1}^NW_{i,j}B_j^*(t)\right) \ge \frac{x}{2}\right) \le c_1 e^{-c_2\sqrt{x}}.
\end{multline*}
This proves \eqref{lln1}.
Define the following stopping times:
$$
\taubar_{0} = 0, \ \ \taubar_{i+1} \doteq \inf\{t \ge \taubar_i + 1: (V(t), \mathbf{Z}(t)) \in C^*\}, \ i \ge 0.
$$
Using \eqref{lln1}, there exists a positive constant $c_3$ such that for any $(v,\bfz) \in C^*$, $n \ge 2$ and $m >0$,
 \begin{align*}
& \mathbb{P}^*_{(v,\bfz)}\left(\Sigma(t) \ge m (\log t)^2 \text{ for some } t \in [n,n+1]\right)\\
& \le \mathbb{P}^*_{(v,\bfz)}\left(\sup_{t \le n+1}\Sigma(t) \ge c_3 m (\log (n+1))^2\right)
 \le \mathbb{P}^*_{(v,\bfz)}\left(\sup_{t \le \taubar_{n+1}}\Sigma(t) \ge c_3 m (\log (n+1))^2\right)\\
 & \le (n+1)\sup_{(v,\bfz) \in C^*} \mathbb{P}^*_{(v,\bfz)}\left(\sup_{t \le \tau_{C^*}(1)} \Sigma(t) \ge c_3 m (\log (n+1))^2\right)
 \le c_1(n+1) e^{-c_2\sqrt{c_3m}\log(n+1)}, 
 \end{align*}
where we used the strong Markov property to obtain the third inequality. 
Choosing $m$ sufficiently large, we see from the Borel-Cantelli Lemma that, for any $(v,\bfz) \in C^*$,
\begin{equation}
	\mathbb{P}^*_{(v,\bfz)}\left( \limsup_{t\to \infty} \frac{\Sigma(t)}{(\log t)^2} <\infty\right) = 1.
\end{equation}
Finally for an arbitrary $(v,\bfz) \in \RR \times \RR_+^{N}$, and with $\Gamma$ as in \eqref{eq928} and $A = \{\limsup_{t\to \infty} \frac{\Sigma(t)}{(\log t)^2} <\infty\}$, applying  Lemma \ref{lem:800},
\begin{align*}
	\mathbb{P}^*_{(v,\bfz)}(A) = \mathbb{P}^*_{(v,\bfz)}(A, \Gamma <\infty)
	= \mathbb{E}^*_{(v,\bfz)}\left( \mathbb{P}^*_{(v,\bfz)}\left(A \mid \clf^*_{\Gamma}\right) 1_{\{\Gamma<\infty\}}\right) = \mathbb{P}^*_{(v,\bfz)}(\Gamma <\infty) = 1.
\end{align*}
The result follows.
\end{proof}

{\bf Proof of Theorem \ref{lln}.}
All limits in the proof hold $\mathbb{P}^*_{(v,\bfz)}$-almost surely for arbitrary $(v,\bfz) \in \RR \times \RR_+^{N}$. From \eqref{eq:rankloc}, 
\begin{equation}\label{ll1}
\lim_{t \rightarrow \infty} \frac{\sum_{j=1}^NX_{(j)}(t)}{Nt} = \lim_{t \rightarrow \infty} \frac{V(t) + \sum_{j=1}^NX_{(j)}(t)}{Nt} = \frac{g}{N}.
\end{equation}
where we used Lemma \ref{fl} in the first equality. Moreover, again using Lemma \ref{fl}, for any $i \in \{0,1,\dots,N\}$,
\begin{equation}\label{ll2}
\frac{1}{t}\left|X_{(i)}(t) - \frac{1}{N}\sum_{j=1}^NX_{(j)}(t)\right| \le \frac{1}{t}\left|X_{(N)}(t) - X_{(0)}(t)\right| \le \frac{1}{t} \sum_{i=1}^N Z_i(t) \rightarrow 0
\end{equation}
as $t \rightarrow \infty$.  The statement in \eqref{le1} now follows from \eqref{ll1} and \eqref{ll2}. Also, \eqref{le2} follows from Lemma \ref{fl} on noting
$$
0 = \lim_{t \rightarrow \infty}\frac{V(t)}{t} = g - \lim_{t \rightarrow \infty}\frac{L_1(t)}{t}.
$$
To prove \eqref{le3}, note that from \eqref{eq:rankloc}, \eqref{eq:gapproc} and Lemma \ref{fl},
$$
0 = \lim_{t \rightarrow \infty}\frac{Z_1(t)}{t}  = - \lim_{t \rightarrow \infty}\frac{X_{(0)}(t)}{t} - \frac{1}{2}\lim_{t \rightarrow \infty}\frac{L_2(t)}{t} + \lim_{t \rightarrow \infty}\frac{L_1(t)}{t},
$$
which gives $\lim_{t \rightarrow \infty}\frac{L_2(t)}{t} = \frac{2(N-1)g}{N}$. Again using \eqref{eq:gapproc} and Lemma \ref{fl},
$$
0 = \lim_{t \rightarrow \infty}\frac{Z_2(t)}{t}  = - \frac{1}{2}\lim_{t \rightarrow \infty}\frac{L_3(t)}{t} + \lim_{t \rightarrow \infty}\frac{L_2(t)}{t} - \lim_{t \rightarrow \infty}\frac{L_1(t)}{t},
$$
from which, we obtain $\lim_{t \rightarrow \infty}\frac{L_3(t)}{t} = \frac{2(N-2)g}{N}$.
Suppose $N \ge 4$ and, for some $i \in \{3,\dots,N-1\}$, the limit $\lim_{t \rightarrow \infty}\frac{L_j(t)}{t}$ exists and equals $\frac{2(N - j+1)g}{N}$ for all $3 \le j \le i$. Then, using \eqref{eq:gapproc},  $\lim_{t \rightarrow \infty}\frac{L_{i+1}(t)}{t}$ exists and
$$
0 = \lim_{t \rightarrow \infty}\frac{Z_i(t)}{t} = -\frac{1}{2}\lim_{t \rightarrow \infty}\frac{L_{i+1}(t)}{t} + \lim_{t \rightarrow \infty}\frac{L_{i}(t)}{t}-\frac{1}{2}\lim_{t \rightarrow \infty}\frac{L_{i-1}(t)}{t}
$$
which implies $\lim_{t \rightarrow \infty}\frac{L_{i+1}(t)}{t} = \frac{2(N-i)g}{N}$. The statement in \eqref{le3} now follows by induction.
\hfill \qed
\\\\

\noindent \textbf{Acknowledgement:} 
The research of SB was supported in part by the NSF CAREER award DMS-2141621. The research of AB was supported in part by the NSF awards DMS-1814894 and DMS-1853968. We thank two anonymous referees for their careful suggestions which improved the article.
\\\\

\noindent \textbf{Conflict of interest statement: }There are no conflicts of interests.

\bibliographystyle{plain}
\bibliography{inert_ref}

\vspace{\baselineskip}

\noindent{\scriptsize {\textsc{\noindent S. Banerjee, A. Budhiraja, and B. Estevez\newline
Department of Statistics and Operations Research\newline
University of North Carolina\newline
Chapel Hill, NC 27599, USA\newline
email: sayan@email.unc.edu
\newline
email: budhiraj@email.unc.edu
\newline
email: bestevez@live.unc.edu
 \vspace{\baselineskip} } }}

\end{document}